\newif\iffinal
\finalfalse 
\finaltrue 

\documentclass[reqno,twoside,11pt]{amsart}
\usepackage{amsmath}
\usepackage{amsfonts}
\usepackage{amssymb}
\usepackage{verbatim}
\usepackage{epsfig}

\IfFileExists{epsf.def}{\input epsf.def}{\usepackage{epsf}}

\newcommand{\esssup}{\operatorname{esssup}}

\newcommand{\prob}{{\bf P}}

\newcommand{\hetzi}{{\frac 12}}
\newcommand{\var}{{\texttt{var}}}
\newcommand{\cov}{{\texttt{cov}}}
\newcommand{\Epsilon}{{\mathcal E}}

\newcommand{\BB}{\mathcal B}

\newcommand{\DD}{\mathcal D}
\newcommand{\EE}{\mathcal E}
\newcommand{\FF}{\mathcal F}
\newcommand{\GG}{\mathcal G}

\newcommand{\LL}{\mathcal L}
\newcommand{\MM}{\mathcal M}

\newcommand{\PP}{\mathcal P}
\newcommand{\hPP}{\tilde{\mathcal P}}

\newcommand{\D}{\mathbb D}

\newcommand{\BbbL}{\mathbb L}

\newcommand{\N}{\mathbb N}

\newcommand{\BbbP}{\mathbb P}
\newcommand{\BbbE}{\mathbb E}

\newcommand{\R}{\mathbb R}

\newcommand{\T}{\mathbb T}
\newcommand{\U}{\mathbb U}

\newcommand{\Z}{\mathbb Z}
\newcommand{\veloc}{\includegraphics[width=0.08in]{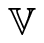}}

\newcommand{\critdim}{5}

\newcommand{\ignore}[1]{}

\newtheoremstyle{thm}{1.5ex}{1.5ex}{\itshape\rmfamily}{} {\bfseries\rmfamily}{}{2ex}{}

\newtheoremstyle{def}{1.5ex}{1.5ex}{\slshape\rmfamily}{} {\bfseries\rmfamily}{}{2ex}{}

\newtheoremstyle{rem}{1.3ex}{1.3ex}{\rmfamily}{} {\itshape}
{} {1.5ex}{}

\theoremstyle{thm}
\newtheorem{theorem}{Theorem}[section]
\newtheorem{lemma}[theorem]{Lemma}
\newtheorem{claim}[theorem]{Claim}
\newtheorem{proposition}[theorem]{Proposition}

\newtheorem*{Main Theorem}{Main Theorem.}
\newtheorem{corollary}[theorem]{Corollary}
\newtheorem*{special theorem}{Lindeberg-Feller Theorem for Martingales}

\newtheorem{conjecture}[theorem]{Conjecture}

\theoremstyle{def}
\newtheorem{definition}{Definition}

\theoremstyle{rem}
\newtheorem{remark}{{\itshape Remark}}[]

\numberwithin{equation}{section}

\renewcommand{\section}{\secdef\sct\sect}
\newcommand{\sct}[2][default]{%
\refstepcounter{section} \addcontentsline{toc}{section}{{\tocsection
{}{\thesection}{\!\!\!\!#1\dotfill}}{}} \vspace{0.7cm}
\centerline{\scshape\thesection.\ #1} \nopagebreak \vspace{0.2cm}}
\newcommand{\sect}[1]{%
\vspace{0.4cm} \centerline{\large\scshape\rmfamily #1}
\vspace{0.2cm}}

\renewcommand{\critdim} {4}
\renewcommand{\ignore}[1]{{}}

\newcommand{\paral}{block}
\newcommand{\parals}{blocks}

\def\myffrac#1#2 in #3{\raise 2.6pt\hbox{$#3 #1$}\mkern-1.5mu\raise 0.8pt\hbox{$#3/$}\mkern-1.1mu\lower 1.5pt\hbox{$#3 #2$}}

\title[Slowdown for RWRE]
{\fontsize{14}{20}\selectfont Slowdown estimates for ballistic random walk in 
random environment}
\author[N.~Berger]{Noam Berger \\ Hebrew university of Jerusalem\,}

\begin{document}
\thanks{\hglue-4.5mm\fontsize{9.6}{9.6}\selectfont The research was partially supported by grant 2006477 of the Isreal-U.S. binational science foundation, by grant 152/2007 of the German Israeli foundation and by ERC StG grant 239990.}
\maketitle

\begin{abstract}
We consider models of random walk in uniformly elliptic i.i.d. random environment in dimension greater
than or equal to \critdim, 
satisfying a condition slightly weaker than the ballisticity condition $(T^\prime)$. We show that for every $\epsilon >0$ and
$n$ large enough, the annealed probability of linear slowdown is bounded from above by
$
\exp\big(-(\log n)^{d-\epsilon}\big).
$
This bound almost matches the known lower bound of 
$
\exp\big(-C(\log n)^d\big),
$
and significantly improves previously known upper bounds. 
As a corollary we provide almost sharp estimates for the quenched probability of slowdown.
As a tool for obtaining the main result, we show an almost local version of the quenched central 
limit theorem under the assumption of the same condition.
\end{abstract}

\newcommand{\joint}{\Large \texttt{P}}
\newcommand{\quenchedP}{P}
\newcommand{\annealedP}{\BbbP}
\newcommand{\quenchedE}{E}
\newcommand{\annealedE}{\BbbE}
\newcommand{\st}{\mbox{ s.t. }}

\ignore{
\section{Notations}

In this section I list the notation used by me throughout the paper, for the purpose of avoiding multiple meanings.

$P$ distribution of environments

$\annealedP$ annealed measure

$\quenchedP$ quenched measure 

$\PP(z,N)$ basic \paral

$\hPP(z,N)$ middle third of basic \paral

$\partial^+$ The right boundary of a set.

$R(N)$ equals $e^{(\log\log N)^2}$

$A_N(\{X_n\})$ The even that the first $N$ regenerations have radii less than $R(N)$.

$H_l$ the plain at distance $l$ from the origin.
}
\section{Introduction}
\label{sec:intro}\noindent
\subsection{Background}
Let $d\geq 1$. A Random Walk in Random Environment (RWRE) on $\Z^d$ is defined as follows: Let
$\MM^d$ denote the space of all probability measures on $\{\pm e_i\}_{i=1}^d$ and let $\Omega=\left(\MM^d\right)^{\Z^d}$. An {\em environment} is a  point $\omega\in\Omega$. Let $P$ be a probability  measure on $\Omega$. For the purposes of this paper, we assume that $P$ is an i.i.d. measure, i.e.
\[
P=Q^{\Z^d}
\]
for some distribution $Q$ on $\MM^d$
and that $P$ is {\em uniformly elliptic}, i.e. there exist $\eta>0$ s.t. for every
neighbor $v$ of the origin,
\begin{equation}\label{eq:unifelliptic}
Q(\{\omega:\omega(v)<\eta\})=0.
\end{equation}
For an environment $\omega\in\Omega$, the {\em Random Walk} on $\omega$ is a time-homogenous
Markov chain with transition kernel
\[
\quenchedP_\omega\left(\left.X_{n+1}=z+e\right|X_n=z\right)=\omega(z,e).
\]
The {\bf quenched law} $\quenchedP_\omega^z$ is defined to be the law on $\left(\Z^d\right)^\N$ induced by
the kernel
$\quenchedP_\omega$ and $\quenchedP_\omega^z(X_0=z)=1$.  We let $\joint^z=P\otimes \quenchedP_\omega^z$ be the joint law of the environment and the walk, and the {\bf annealed} law is defined to be its marginal
\[
\annealedP^z=\int_{\Omega}P_\omega^zdP(\omega).
\]

For simplicity, we omit the superscript when the walk starts from zero.

We use the notations $\quenchedE^z_\omega$ and $\annealedE^z$ for the expectations with respect to the measures  $\quenchedP^z_\omega$ and $\annealedP^z$. 

In \cite{sznit_zer} and \cite{zerner}, Sznitman and Zerner proved that the limiting velocity
\begin{equation*}
\veloc=\lim_{n\to\infty}\frac{X_n}{n}
\end{equation*}
exists almost surely. A remaining open problem, which is one of the most important problems in this field, is whether this limiting velocity is always an almost sure constant.

We now introduce three important definitions:
\begin{definition}\label{def:ballistic}
The RWRE is said to be {\em ballistic} if the limiting velocity is a non-zero almost sure constant.
\end{definition}

\begin{definition}\label{def:localdrift}
The local drift at a point $z$ is defined to be the (quenched) quantity
\[
\Delta_\omega(z):=\sum_{e\in\{\pm e_i\}_{i=1}^d} e\omega(z,e)
=\quenchedE_\omega^z(X_1-z).
\]
\end{definition}

\begin{definition}\label{def:nestling}
The RWRE is said to be {\em plain nestling} if zero is contained in the interior of the convex hull of the 
support of the random variable $\Delta_\omega(0)$. It is said to be {\em marginally nestling} if zero is
on the boundary of the convex hull of the support, and {\em non-nestling} if zero is outside the
convex hull of the support.
\end{definition}

\subsection{Large deviations for RWRE}
In \cite{varadhan}, Varadhan considered large deviations for the sequence of random variables 
$\frac{X_n}{n}$ under the annealed measure $\annealedP$. he showed that a large deviation principle holds with a 
rate function $F$, and identified the zero set of the function $F$. For the ballistic case with limiting velocity $\veloc$, Varadhan showed that if the RWRE is non-nestling, then $F^{-1}(0)=\{\veloc\}$, while if the RWRE is plain nestling or marginally nestling, then $F^{-1}(0)=A$, with $A$ being the convex hull of $0$ and $\veloc$. We note here that recently Yilmaz \cite{yilmaz} and Peterson \cite{peterson} obtained more information about the structure of the rate function $F$.

In other words, for every $a\notin A$ and $\epsilon>0$ small enough, 
\begin{equation}\label{eq:ball}
\annealedP\left(\left\|\frac{X_n}{n}-a\right\|_\infty<\epsilon\right)
\end{equation}
 decays exponentially with $n$, and for every $a\in A$, \eqref{eq:ball} decays more slowly than exponentially. (Note that the choice of the $\ell^\infty$ norm is completely arbitrary, since in our finite-dimensional space,  all norms are equivalent)

It is therefore natural to ask what the decay rate of \eqref{eq:ball} is for $a\in A$. 

In the marginally nestling case, Sznitman \cite{sznitmanquenched} showed that there exist $C_1$ and $C_2$ such that
\begin{equation}\label{eq:marginnestle}
e^{-C_1n^{\frac{d}{d+2}}}<
\annealedP\left(\left\|\frac{X_n}{n}-a\right\|_\infty<\epsilon\right)
<e^{-C_2n^{\frac{d}{d+2}}}
\end{equation}
for large enough $n$. In \cite{sznitmanquenched} Sznitman phrased \eqref{eq:marginnestle} in the language of bounds on the distribution of the first regeneration time. Nevertheless, the way it is presented here follows immediately from Sznitman's result using the appropriate large deviation estimates.


\subsection{Main goal}
The purpose of this paper is to provide an estimate for the probability in \eqref{eq:ball} in the plain nestling case under some additional assumptions which we specify below.

\subsection{Ballisticity conditions}

In \cite{SznitmanT,SznitmanTprime} Sznitman introduced two criteria for ballisticity of the RWRE, 
which he called conditions $(T)$ and $(T^\prime)$. In order to define these conditions,
we need some preliminary definitions. 

\begin{definition}\label{def:hittimelayer}
Let $\ell\in S^{d-1}$ be a direction in $\R^d$. Let $L>0$. For a sequence $\{X_n\}$, we define
\begin{equation*}
T^{(\ell)}_L(\{X_n\})=\inf\{n\geq 0\,:\,\langle X_n,\ell\rangle\geq L\}.
\end{equation*}
If no confusion may arise, we may omit $\ell$ and $\{X_n\}$ from this notation.
\end{definition}

Equivalently to Definition \ref{def:hittimelayer}, we also define the first hitting time of a set.
\begin{definition}\label{def:hittimeset}
Let $A\subseteq\Z^d$.For a sequence $\{X_n\}$, we define
\begin{equation*}
T_A(\{X_n\})=\inf\{n\geq 0\,:\, X_n\in A\}.
\end{equation*}
Again, we may omit $\{X_n\}$ when no confusion may arise.
\end{definition}

We now return to Sznitman's ballisticity conditions. We start by defining 
the condition $(T_\gamma),\ 0<\gamma\leq 1$ as follows:

\begin{definition}
We say that $P$ satisfies condition $(T_\gamma)$ in direction $\ell_0$ if for every $\ell$ in a neighborhood of $\ell_0$ there exists a constant $C$  such that for every large enough $L$,
\begin{equation}\label{eq:tgamma}
\annealedP\big(T^{(-\ell)}_L<T^{(\ell)}_L\big)<C\exp\left(-L^\gamma\right)
\end{equation}
\end{definition}

\begin{definition}
We say that $P$ satisfies condition $(T)$ if it satisfies condition $(T_1)$. We say that it satisfies condition $(T^\prime)$ if it satisfies condition $(T_\gamma)$ for some $\gamma>1/2$.
\end{definition}

In \cite{SznitmanTprime}, it is  shown that the conditions $(T_\gamma)_{1/2<\gamma<1}$ are all equvalent.

The connection between the conditions mentioned above and ballisticity lies in the following theorem and conjecture:

\begin{theorem}[Sznitman, \cite{SznitmanTprime}]\label{thm:sznitmanspeed}
If  condition $(T^\prime)$ holds for some $\ell_0$, then the RWRE is ballistic, and the limiting velocity $\veloc$ satisfies $\langle \veloc, \ell_0\rangle>0$. Furthermore, in this case $(T^\prime)$ holds for all $\ell$ satisfying
$\langle \veloc, \ell \rangle>0$
\end{theorem}
\noindent
{\bf Remark:} This result was recently improved by Drewitz and Ram\'irez \cite{drewitz}.

\begin{conjecture}[Sznitman]
Condition $(T^\prime)$ is equivalent to ballisticity.
\end{conjecture}

\subsection{Known slowdown results}

Let $d\geq 2$ and let $a\neq \veloc$ be in the convex hull of $0$ and $\veloc$. Then for $\epsilon>0$ small enough, the following is known.

\begin{theorem}[Sznitman, \cite{SznitmanTprime}]\label{thm:sznitmanslowdown}
Assume that $P$ is plain nestling, uniformly elliptic and satisfies condition $(T^\prime)$.
\begin{enumerate}
\item
There exist $C$ such that for $n$ large enough
\begin{equation}\label{eq:lbnd}
\annealedP\left(\left\|\frac{X_n}{n}-a\right\|_\infty<\epsilon\right)
>e^{-C(\log n)^d}
\end{equation}
\item
Let $\alpha<\frac{2d}{d+1}$. There exist $C$ such that for $n$ large enough.
\begin{equation}\label{eq:known_ubnd}
\annealedP\left(\left\|\frac{X_n}{n}-a\right\|_\infty<\epsilon\right)
<e^{-C(\log n)^\alpha}
\end{equation}
\end{enumerate}

\end{theorem}

The easy bound in Theorem \ref{thm:sznitmanslowdown} is \eqref{eq:lbnd}, which follows from the analysis of the so called na\"ive trap (See Figure \ref{fig:naive}).

\begin{figure}[t]
\begin{center}
\epsfig{figure=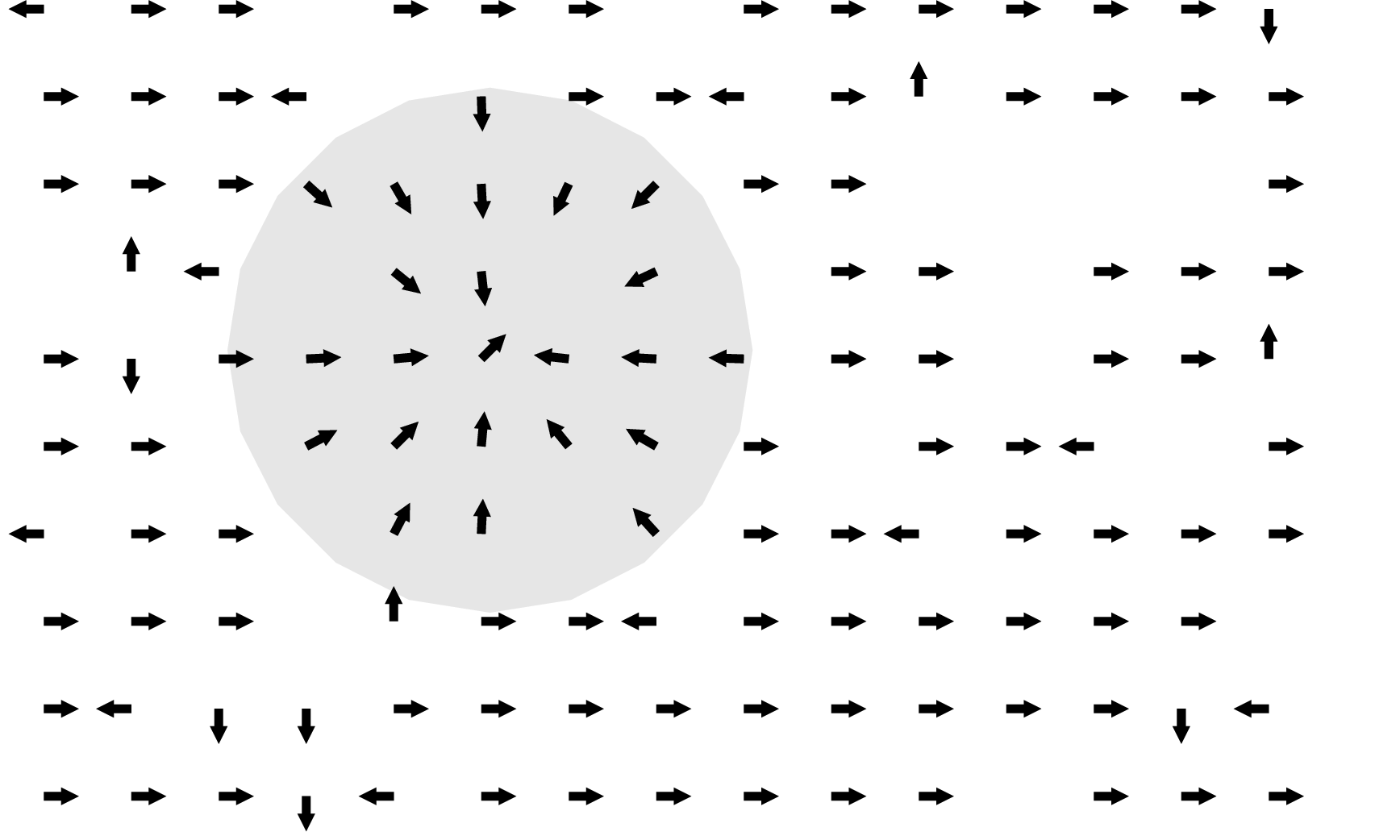, width=10cm}
\caption{\sl 
A na\"ive trap. In the shaded ball of radius $C\log n$ around the origin all local drifts are pointing towards the origin, while outside the ball the local drift goes mostly to the right. If $C$ is appropriately chosen, then this causes a linear slowdown. The probability that such configuration exists is exponential in $(\log n)^d$.
}
\label{fig:naive}
\end{center}
\end{figure}  

\subsection{Ballisticity under $(T_\gamma)$}
We prove the following result:
\begin{theorem}\label{thm:ballis}
Assume that the dimension is at least \critdim. Fix $\gamma>0$. Under the assumption of uniform ellipticity,
if  condition $(T_\gamma)$ holds for some $\ell_0$, then the RWRE is ballistic, and the limiting velocity $\veloc$ satisfies $\langle \veloc, \ell_0\rangle>0$. Furthermore, in this case $(T_\gamma)$ holds for all $\ell$ satisfying
$\langle \veloc, \ell \rangle>0$
\end{theorem}

\subsection{Main results}
Our main result is the following theorem:

\begin{theorem}\label{thm:main}
Let $d\geq\critdim$ and $\gamma>0$
and assume that $P$ is uniformly elliptic and satisfies condition $(T_\gamma)$.
Let $a\neq \veloc$ be in the convex hull of $0$ and $\veloc$ and let $\epsilon>0$ be small enough so that 
$\veloc$ is not in the closed $\epsilon$-neighborhood of $a$. Let $\alpha<d$. Then for all $n$ large enough,
\begin{equation}\label{eq:thmmain}
\annealedP\left(\left\|\frac{X_n}{n}-a\right\|_\infty<\epsilon\right)
<e^{-(\log n)^\alpha}
\end{equation}
\end{theorem}

Comparing Theorem \ref{thm:main} and \eqref{eq:lbnd} shows that the remaining gap between the upper and the lower bounds is quite small.

Theorem \ref{thm:main} deals with the annealed probability of slowdown. However, one can deduce from it a quenched bound.
\begin{corollary}\label{cor:mainquenched}
With the same assumptions as in Theorem \ref{thm:main}, for every $\alpha<d$, almost every $\omega$ and every large enough $n$,
\begin{equation}\label{eq:thmmainquenched}
\quenchedP_\omega\left(\left\|\frac{X_n}{n}-a\right\|_\infty<\epsilon\right)
<\exp\left(-\frac{n}{\exp\big((\log n)^{{\alpha}^{-1}}\big)}\right).
\end{equation}
\end{corollary}

Again, compare \eqref{eq:thmmainquenched} to the known lower bound
\begin{equation*}
\quenchedP_\omega\left(\left\|\frac{X_n}{n}-a\right\|_\infty<\epsilon\right)
>\exp\left(-\frac{Cn}{\exp\big((\log n)^{d^{-1}}\big)}\right)
\end{equation*}

which is proven in \cite{sznitmanquenched}. Corollary \ref{cor:mainquenched} follows from Theorem \ref{thm:main}
using the method developed by Gantert and Zeitouni in \cite{gantertzeitouni}, 
for transferring annealed slowdown estimates into quenched ones. This method was
 adjusted for the multi-dimensional case by Sznitman in \cite{sznitmanquenched}. The proof of Corollary \ref{cor:mainquenched} is identical to the proof of (5.45) in \cite{sznitmanquenched}, and will be
omitted from the present paper.

\subsection{Remark about lower dimensions}
In this paper we only prove Theorem \ref{thm:main} for dimensions 4 and higher. 
Here we discuss the situation in lower dimensions.

For $d=1$, the annealed slowdown probability was calculated by Dembo, Peres and Zeitouni in 1996 \cite{DPZ} and the quenched slowdown probability was calculated by Gantert and Zeitouni in 1998 \cite{gantertzeitouni}.
These results give bounds that are significantly sharper than the bounds in Theorem \ref{thm:main} and Corollary \ref{cor:mainquenched}. Nevertheless, a comparison between the results shows that the estimates in the present paper are true for dimension 1.

I conjecture that the results in this paper hold for dimensions 2 and 3. The difficulty in the proof occurs in Proposition \ref{prop:quenched}, which is currently only proved for dimensions 4 and higher. In dimension 3 I expect that a more sophisticated version of the arguments in this paper should be able to work. In dimension 2, Part \ref{item:toch0} of Proposition \ref{prop:quenched} does not hold, and therefore a new proof idea is needed.

\subsection{Structure of the paper}
In Section \ref{sec:reg} we bring the definition of regeneration times as introduced in \cite{sznit_zer}. We then reformulate Theorem \ref{thm:main} in the language of regenerations and get Proposition \ref{prop:main}. Then in Section \ref{sec:prel} we introduce some very useful notation and give some basic definitions. In Section \ref{sec:clt} we give a number of CLT type results. In particular, we give an almost local version of the quenched central limit theorem (Proposition \ref{prop:quenched}) and a general lemma about sums of approximately Gaussian variables (Lemma \ref{lem:sumapprox}). In Section \ref{sec:qrp} we
reformulate Theorem \ref{thm:main} as a statement about quenched exit properties from a large box. The first half of this construction is very similar to Sznitman's construction in \cite{SznitmanTprime}. Then in Section \ref{sec:auxwalk} we define an auxiliary walk $\{Y_n\}$ and explore its connection to the original walk $\{X_n\}$. In Section \ref{sec:randirev} we define an event regarding the walk $\{Y_n\}$, and in Section 
\ref{sec:prfmain} we use all that information in order to prove the main result.

\subsection{Remark about the writing style}
In order to avoid notational overload, language is abused in three ways in this paper: (a) the value of a constant $C$ may change from one line to the next, (b) some of the inequalities in the paper only hold for $n$ which is large enough without this being explicitly mentioned, and (c) for a probability measure $\mu$ on $\Z^d$, we use the symbol $E_\mu$ for the expectation $\sum_{x}x\mu(x)$.

In addition we use the highly convenient notations from Computer Science $O(n)$, $o(n)$, $\Omega(n)$ and $\xi(n)$, whose meanings are as described in the table below. Note that Computer Scientists write $\omega$ rather than $\xi$. However, due to the use of the letter $\omega$ for the environment in this paper, we use $\xi$ as described below.

\vspace{0.3cm}

\begin{tabular}{|l|l|}
\hline
Symbol & Meaning \\
\hline
$k=O(n)$ & as the parameter goes to infinity, $\limsup \frac{k}{n}<\infty$.\\
\hline
$k=o(n)$ & as the parameter goes to infinity, $\lim \frac{k}{n}=0$.\\
\hline
$k=\Omega(n)$ & as the parameter goes to infinity, $\liminf \frac{k}{n}>0$.\\
\hline
$k=\xi(n)$ & as the parameter goes to infinity, $\lim \frac{k}{n}=\infty$.\\
\hline
\end{tabular}

\vspace{0.3cm}

For example, if we write $f(N)=N^{-\xi(1)}$, we mean that as $N$ goes to infinity, $f(N)$ goes to zero
faster than any power of $N$.

Whenever the norm sign $\|\cdot\|$ appears without mentioning which norm we are referring to, we refer to the
$\ell^\infty$ norm on $\Z^d$.

\section{Regeneration times}\label{sec:reg}

We first define the notion of a regeneration time. Our definition is slightly different from
that given by Sznitman and Zerner in \cite{sznit_zer}. Nevertheless, all the lemmas that we quote
from \cite{SznitmanTprime} and \cite{sznit_zer} and collect into Theorem \ref{thm:inf_reg} apply equally well to the definition below.

\begin{definition}\label{def:regtimes}
Let $\{X_n\}$ be a nearest-neighbor sequence in $\Z^d$, and let $\ell\in S^{d-1}$ be a direction.
We say that $t$ is a {\em regeneration time for $\{X_n\}$ in the direction $\ell$} if the following hold.
\begin{enumerate}
\item $\langle X_s,\ell \rangle < \langle X_t,\ell \rangle$ for every $s<t$.
\item $\langle X_{t+1},\ell \rangle > \langle X_t,\ell \rangle$.
\item $\langle X_s,\ell \rangle > \langle X_{t+1},\ell \rangle$ for every $s>t+1$.
\end{enumerate}
\end{definition}

\begin{theorem}\label{thm:inf_reg}[\cite{sznit_zer,SznitmanTprime}]
Assume that $P$ satisfies condition $(T_\gamma)$ in direction $\ell_0$ for some $\gamma>0$. Then
\begin{enumerate}
\item\label{item:sz_inf}
with probability $1$, there exist infinitely many regeneration times. We call them
$\tau_1<\tau_2<\ldots$.
\item\label{item:sz_iid}
The ensemble 
\[
\{(\tau_{n+1}-\tau_n,X_{\tau_{n+1}}-X_{\tau_n})\}_{n\geq 1}
\]
is an i.i.d. ensemble under the annealed measure.
\item\label{item:sz_cond}
There exists $C$ such that for every $n$
\[
\annealedP(\tau_2-\tau_1=n)\leq C\annealedP(\tau_1=n),
\]
and for every $y\in\Z^d$
\[
\annealedP(X_{\tau_2}-X_{\tau_1}=y)\leq C\annealedP(X_{\tau_1}=y).
\]
\item\label{item:sz_decay}
There exists $C$ such that for every $n$,
\[
\annealedP\big(\exists_{k\leq\tau_1}\mbox{ s.t. } \|X_k\|>n\big)
\leq e^{-Cn^{\gamma}}.
\]
\end{enumerate}
\end{theorem}

The main technical statement in this paper is the following proposition.

\begin{proposition}\label{prop:main}
For any $\gamma>0$, if the dimension $d$ is greater than or equal to \critdim, and $\annealedP$ satisfies condition $(T_\gamma)$ in one of the $2d$ principle directions, then for every $\alpha<d$ and every $u$ large enough,
\begin{equation}\label{eq:main}
\annealedP(\tau_1>u)\leq \exp\left(-(\log u)^\alpha\right).
\end{equation}
\end{proposition}

We now show how to prove Theorem \ref{thm:main} assuming Proposition \ref{prop:main}. The rest of the paper will be dedicated to the proof of Proposition \ref{prop:main}.

\begin{proof}[Proof of Theorem \ref{thm:ballis} assuming Proposition \ref{prop:main}] Theorem \ref{thm:ballis} follows from Proposition \ref{prop:main} exactly the same way ballisticity is proved in \cite{sznit_zer} and \cite{SznitmanTprime}.
\end{proof}
\begin{proof}[Proof of Theorem \ref{thm:main} assuming Proposition \ref{prop:main}] Fix $\alpha<d$.
Assume without loss of generality that $\langle \veloc,e_1\rangle>0$. Note that in this case condition $(T_\gamma)$ holds with respect to the direction $e_1$. For simplicity, in this proof we denote $\bar x=\langle x,e_1\rangle$ for every $x\in\R^d$.
Fix $a$ and $\epsilon$ as in the statement of Theorem \ref{thm:main}.
Let $\rho=\annealedE(\tau_2-\tau_1)$ and let $\beta=\annealedE(\bar X_{\tau_2}-\bar X_{\tau_1})$.
Let $r$ be such that $r<\bar \veloc$ but $r>\bar x$ for every $x$ in the $\epsilon$-neighborhood of $a$.
Then it is sufficient to show that for all $n$ large enough,
\begin{equation}\label{eq:comptor}
\annealedP\left(\bar X_n<rn\right)
<e^{-(\log n)^\alpha}.
\end{equation}

Choose $b$ so that $r/\bar \veloc<b<1$, and let $m=nb/\rho$.

Then
\begin{equation}\label{eq:decomp1}
\annealedP(\bar X_n<rn)
\leq \annealedP(\tau_{m+1}>n) + \annealedP(\bar X_{\tau_{m+1}}<rn).
\end{equation}

Now, remembering that $\bar \veloc=\beta/\rho$,
\begin{eqnarray*}
\annealedP(\bar X_{\tau_{m+1}}<rn)
\leq \annealedP(\bar X_{\tau_{m+1}}-\bar X_{\tau_1}<rn)
=\annealedP\left(\sum_{k=1}^m \bar X_{\tau_{k+1}}-\bar X_{\tau_k} < \frac rb \rho m\right).
\end{eqnarray*}

Remembering that the sequence $\{\bar X_{\tau_{k+1}}-\bar X_{\tau_k}\}$ is i.i.d. and that 
$\bar X_{\tau_{k+1}}-\bar X_{\tau_k}$ is positive and its expectation $\beta$ is larger than $\frac rb \rho$,
we get that
\begin{equation*}
\annealedP(\bar X_{\tau_{m+1}}<rn) < e^{-Cn}
\end{equation*}
for some constant $C$.

We now estimate $\annealedP(\tau_{m+1}>n)$. Let $A$ be the event that $\tau_{k+1}-\tau_k<n^{1/8}$ for all $k=1,\ldots,m$ and that $\tau_1<n^{1/8}$.
Then 
\[
\annealedP(\tau_{m+1}>n)\leq \annealedP(A^c)+\annealedP(\tau_{m+1}>n|A).
\]

Fix $\alpha^\prime$ between $\alpha$ and $d$. Then, by part \ref{item:sz_cond} of Theorem \ref{thm:inf_reg} and Proposition \ref{prop:main}, for all $n$ large enough,
\[
\annealedP(A^c)\leq e^{-(\log n^{1/8})^{\alpha^\prime}}+mCe^{-(\log n^{1/8})^{\alpha^\prime}}
\leq \frac 12 e^{-(\log n)^\alpha}.
\]
Conditioned on $A$, the variables $\tau_{k+1}-\tau_k$ are independent, bounded by $n^{1/8}$ and their expectation is less than $\rho$.
Then by Azuma's inequality, for $n$ large enough,
\begin{eqnarray*}
\annealedP(\tau_{m+1}>n|A)
&\leq& \annealedP(\tau_{m+1}-\tau_1>n-n^{1/8}|A)\\
\leq \exp \left( -\frac{(n-n^{1/8}-\rho m)^2}{2mn^{1/4}}\right)
&\leq& \exp (-n^{3/4}).
\end{eqnarray*}

\eqref{eq:comptor} follows.

\end{proof}


\section{Preliminaries}\label{sec:prel}
We define $R_{k}(N)=\big[\exp((\log\log N)^{k+1})\big]$, and $R(N)=R_1(N)$.
Note that $R_{0}(N)=[\log N]$ and that for every $k$, every $M$ and for every large enough $N$,
\begin{equation}\label{eq:rs}
R_k^M(N)<R_{k+1}(N)<N.
\end{equation}


We let 
\begin{equation}\label{eq:speed_dir}
\vartheta:=\lim_{n\to\infty}\frac{X_n}{\|X_n\|}
\end{equation}
be the direction of the speed.
Note that the existence of $\vartheta$ follows from $T_\gamma$ even without ballisticity assumption, 
and is always non-zero.
We assume without loss of generality
that $\langle \vartheta,e_1\rangle>0$. Note that,
by the results of \cite{SznitmanT} and \cite{SznitmanTprime}, $(T_\gamma)$ holds both in direction $e_1$ and in direction $\vartheta$.

\begin{definition}\label{def:parall}
For $z\in\Z^d$ and $N\in\N$, we define the {\em basic \paral} of size $N$ around $z$ to be
\begin{equation*}
\PP(z,N):=
\left\{
x\in\Z^d\ : \ 
\left|\langle x,e_1\rangle -\langle z,e_1\rangle \right|<N^2 \ ; \ 
\left\|
x-u(z,x)
\right\|_\infty<NR_5(N)
\right\}
\end{equation*}
where 
\begin{equation}\label{eq:uzx}
u(z,x):=z+\vartheta\cdot\frac{\langle x-z,e_1\rangle}{\langle \vartheta,e_1\rangle}.
\end{equation}\label{}
The {\em middle third} of $\PP(z,N)$ is defined to be
\begin{equation*}
\hPP(z,N):=
\left\{
x\in\Z^d\ : \ 
\left|\langle x,e_1\rangle -\langle z,e_1\rangle \right|<\frac{1}{3}N^2 \ ; \ 
\left\|
x-u(z,x)
\right\|_\infty<\frac{1}{3}NR_5(N)
\right\}
\end{equation*}
\end{definition}

We let
\begin{equation*}
\partial\PP(z,N):=
\left\{
x\in\Z^d\setminus\PP(z,N)\ : \
\exists_{y\in\PP(z,N)} \mbox{ s.t. } \|x-y\|_1=1
\right\},
\end{equation*}
and
\begin{equation*}
\partial^+\PP(z,N):=
\left\{
x\in\partial\PP(z,N)\ : \ 
\langle x,e_1\rangle - \langle z,e_1\rangle = N^2
\right\}.
\end{equation*}

\begin{figure}[h]
\begin{center}
\epsfig{figure=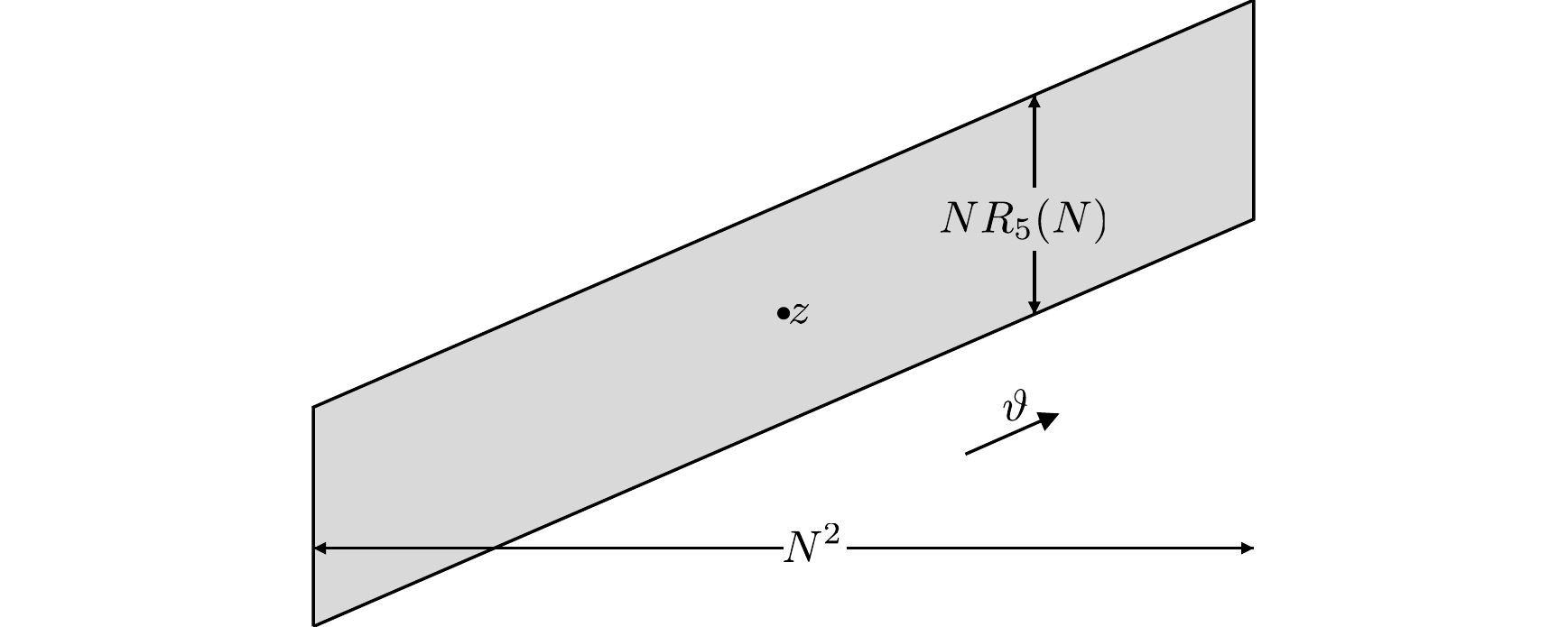,width=14cm}
\caption{\sl 
The basic block $\PP(z,N)$.
}
\label{fig:block}
\end{center}
\end{figure}

\ignore{
\begin{definition}\label{def:lattice}
The {\em basic lattice} of size $N$ is defined to be
\begin{equation*}
\LL_N:=\left[\frac{N^2}{4}\right]\Z
\times
\left(
\left[\frac{NR_5(N)}{4}\right]\Z
\right)^{d-1}.
\end{equation*}
\end{definition}

The following simple fact will be useful in what follows.
\begin{lemma}\label{lem:parandlat}
\begin{enumerate}
\item
For every $x\in\Z^d$ there exists $z\in\LL_N$ such that 
$x\in\hPP(z,N)$.
\item
$\LL_N$ can be presented as the disjoint union of $9^d$ lattices, such that
if $\LL$ is one of these lattices then for every $z_1\neq z_2$ in $\LL$,
\begin{equation*}
\PP(z_1,N)\cap\PP(z_2,N)=\emptyset.
\end{equation*} 
\end{enumerate}
\end{lemma}
}

\begin{definition}\label{def:lattice}
The {\em basic lattice} of size $N$ is defined to be
\begin{equation*}
\LL_N:=N^2\Z
\times
\left(
\left[\frac{NR_5(N)}{4}\right]\Z
\right)^{d-1}.
\end{equation*}
\end{definition}

The following simple fact will be useful in what follows.
\begin{lemma}\label{lem:parandlat}
\begin{enumerate}
\item
For every $x\in N^2\Z\times\Z^{d-1}$ there exists $z\in\LL_N$ such that 
$x\in\hPP(z,N)$.
\item
$\LL_N$ can be presented as the disjoint union of $9^d$ lattices, such that
if $\LL$ is one of these lattices then for every $z_1\neq z_2$ in $\LL$,
\begin{equation*}
\PP(z_1,N)\cap\PP(z_2,N)=\emptyset.
\end{equation*} 
\end{enumerate}
\end{lemma}

\section{CLT type estimates}\label{sec:clt}
In this section we derive two clt-type estimates that will be important for the proof of
the main result.
All throughout this section, we assume that our RWRE model satisfies the following requirements:
\begin{enumerate}
\item\label{item:assgamma} There exists $\gamma>0$ such that $T_\gamma$ holds.
\item\label{item:assunel} The RWRE is uniformly elliptic.
\item\label{item:assdim} The dimension $d$ is at least $4$.
\end{enumerate}

We begin with some preliminary estimates in subsections \ref{subsec:regrad} and \ref{subsec:ann_der},
and then prove CLT type estimates in subsections \ref{subsec:quenched} and \ref{subsec:sumapprox}.
\subsection{Regeneration radii}\label{subsec:regrad}
Let $X^{*(1)}:=\max\{\|X_t-X_0\|\ : \ 0\leq t\leq \tau_1\}$, and for $n>1$ let
$X^{*(n)}:=\max\{\|X_t-X_{\tau_{n-1}}\|\ : \ \tau_{n-1}\leq t\leq \tau_n\}.$

The following lemma appears in \cite{SznitmanTprime}.
\begin{lemma}\label{lem:regrad}
Let $\gamma>0$ and
assume that condition $(T_\gamma)$ holds.
Then there exist $C$ and $c$ such that for every $L$ and $n=1,2,\ldots$,
\begin{equation}\label{eq:regrad}
\annealedP(X^{*(n)}>L)\leq Ce^{-cL^{\gamma}}.
\end{equation}
for every $n$.
\end{lemma}

We call $X^{*(n)}$ {\em the radius of the $n$-th regeneration.}
Recall that $R(N)=R_1(N)=\big[e^{(\log\log N)^2}\big]=\big[(\log N)^{\log\log N}\big].$
Let $A_N(\{X_n\})$ be the event that the radii of the first $N$ regenerations are all smaller than $R(N)$, namely
\begin{equation}\label{eq:defan}
A_N(\{X_n\}):=
\left\{
\forall_{1\leq n\leq N},\ X^{*(n)}<R(N)
\right\}.
\end{equation}

Then, by Lemma \ref{lem:regrad},
\begin{equation}\label{eq:probA}
\annealedP(A_N(\{X_n\}))\geq 1-CNe^{-cR(N)^\gamma}=1-N^{-\xi(1)}.
\end{equation}

\subsection{Derivatives of the annealed exit distribution}\label{subsec:ann_der}
In this subsection we show the following result on the annealed exit distribution from a \paral.
The proof is standard and straightforward using regenerations.
\begin{lemma}\label{lem:ann_der}
Assume the assumptions \ref{item:assgamma}--\ref{item:assdim} from Page \pageref{item:assgamma}.
Fix $z\in\Z^d$ and $N$, and let $z_1\in\hPP(z,N)$. Let $\{X_n\}$ be an RWRE starting at $z_1$,
and let $u:=X_{T_{\partial\PP(z,N)}}.$
Then, for large enough $N$,
\begin{enumerate}
\item\label{item:exit_from_right}
$\annealedP^{z_1}(u\notin \partial^+\PP(z,N))\leq e^{-R_5(N)}+\annealedP(A_N^c)=N^{-\xi(1)},$ where $A_N$ is as defined in \eqref{eq:defan}.
\item\label{item:zero_der}
For every $x$ in $\partial^+\PP(z,N)$,
\begin{equation*}
\annealedP^{z_1}(u=x)<CN^{{1-d}}
\end{equation*}
\item\label{item:first_der}
For every $x$ and $y$ in $\partial^+\PP(z,N)$ s.t. $\|x-y\|_1=1$, 
\begin{equation*}
\left|\annealedP^{z_1}(u=x)-\annealedP^{z_1}(u=y)\right|<CN^{{-d}}
\end{equation*}
\item\label{item:first_der_ort}
Let $\{X^\prime_n\}$ be an RWRE starting at $z_1+e_1$, and let $u^\prime:=X^\prime_{T_{\partial\PP(z,N)}}.$ Then for every $x$  in $\partial^+\PP(z,N)$,
\begin{equation*}
\left|\annealedP^{z_1}(u=x)-\annealedP^{z_1+e_1}(u^\prime=x)\right|<CN^{{-d}}
\end{equation*}
\item\label{item:second_der}
For every $x$, $y$ and $w$ in $\partial^+\PP(z,N)$ s.t. $\|x-y\|_1=1$ and $w-y=y-x$, 
\begin{equation*}
\left|\annealedP^{z_1}(u=x)+\annealedP^{z_1}(u=w)-2\annealedP^{z_1}(u=y)\right|<CN^{{-d-1}}
\end{equation*}
\item\label{item:mix_der}
For every $x$, $y$ and $w$ and $o$ in $\partial^+\PP(z,N)$ s.t. there exist $i\neq j$ such that
$x-y=w-v=e_i$ and $x-w=y-o=e_j$,
\begin{equation*}
\left|\annealedP^{z_1}(u=x)+\annealedP^{z_1}(u=o)-\annealedP^{z_1}(u=y)-\annealedP^{z_1}(u=w)\right|<CN^{{-d-1}}
\end{equation*}
\end{enumerate}
\end{lemma}

\begin{proof}
To prove part (\ref{item:exit_from_right}), we need to show that 
\[
\annealedP^{z_1}(u\notin \partial^+\PP(z,N)|A_N)\leq e^{-R_5(N)}
\]
To this end, note that conditioned on $A_N$, the regenerations are independent and are all bounded by $R(N)$.
For $k=1,\ldots,N$ we now estimate the difference between $\annealedE^{z_1}(X_{\tau_k}|A_N)$ and $\annealedE^{z_1}(X_{\tau_k})$. Let $\Xi_k=X_{\tau_k}-X_{\tau_{k-1}}$ with $\Xi_1=X_{\tau_1}$.
Remember that $\annealedP^{z_1}(A_N^c)=N^{-\xi(1)}$. For a given $k$,
\begin{eqnarray*}
\annealedE^{z_1}\big(\|\Xi_k\|\cdot {\bf 1}_{A_N^c}\big)
&\leq&
\annealedE^{z_1}\big(\|\Xi_k\|\cdot {\bf 1}_{\exists_{j\neq k}X^{\star(j)}\geq R(N)}\big)
+ \annealedE^{z_1}\big(\|\Xi_k\|\cdot {\bf 1}_{X^{\star(k)}\geq R(N)}\big)\\
&\leq&
\annealedP^{z_1}(A_N^c)\annealedE^{z_1}(\|\Xi_k\|)
+\sum_{h>R(N)}h\annealedP^{z_1}(X^{\star(k)}=h)\\
&\leq&
\annealedP^{z_1}(A_N^c)\annealedE^{z_1}(\|\Xi_k\|)
+\sum_{h>R(N)}he^{-ch^{\gamma}}=N^{-\xi(1)}.
\end{eqnarray*}
Therefore, for every $k=1,\ldots,N$,
\begin{eqnarray*}
&&\|\annealedE^{z_1}(\Xi_k | A_N)-\annealedE^{z_1}(\Xi_k)\|\\
\leq \|\annealedE^{z_1}(\Xi_k \cdot {\bf 1}_{A_N})-\annealedE^{z_1}(\Xi_k)\|
&+& \|\annealedE^{z_1}(\Xi_k | A_N)-\annealedE^{z_1}(\Xi_k \cdot{\bf 1}_{A_N})\|\\
\leq \annealedE^{z_1}\big(\|\Xi_k\|\cdot {\bf 1}_{A_N^c}\big) 
&+& \annealedE^{z_1}\big(\|\Xi_k\| | A_N\big)\annealedP^{z_1}(A_N^c)
=N^{-\xi(1)}.
\end{eqnarray*}
Therefore
$\|\annealedE^{z_1}(X_{\tau_k} | A_N)-\annealedE^{z_1}(X_{\tau_k})\|=N^{-\xi(1)}$, again for every $k=1,\ldots,N$.

Now, using Azuma's inequality,
\begin{eqnarray*}
\annealedP^{z_1}(u\notin \partial^+\PP(z,N)|A_N)
&\leq& \sum_{k=1}^{N^2}\annealedP^{z_1}\left[\|X_{\tau_k}-\annealedE^{z_1}(X_{\tau_k}|A_N)\|_\infty>\frac{1}{4}NR_5(N)\right]\\
&\leq& d\sum_{k=1}^{N^2}\exp\left(\frac{-N^2R_5^2(N)/16}{2kR^2(N)}\right)\leq e^{-R_5(N)},
\end{eqnarray*}
for $N$ large enough.

To prove Parts (\ref{item:zero_der}), (\ref{item:first_der}), (\ref{item:first_der_ort}), (\ref{item:second_der}) and (\ref{item:mix_der}) we need the following standard claim.
\begin{claim}\label{claim:sum_der}
Let $\{Y_i\}_{i=1}^\infty$ be $d$-dimensional independent random variables, with joint distribution $\prob$, such that $\{Y_n\}_{n\geq 2}$ are identically distributed and such that there exists $v\in\Z^d$ such that $\prob(Y_2=v)>0$ and $\prob(Y_2=v+e_i)>0$ for $i=1,\ldots,d$. Let 
$S_n=\sum_{i=1}^nY_i$. Then there exists $C<\infty$ which is determined by the distributions of $Y_1$ and $Y_2$ such that for every $n$ and every $x$, $y$ and $w$ s.t. $\|x-y\|_1=1$ and $w-y=y-x$,
\begin{equation}\label{eq:zero_der_fr}
\prob(S_n=x)\leq Cn^{\frac{-d}{2}},
\end{equation}
\begin{equation}\label{eq:first_der_fr}
\left|\prob(S_n=x)-\prob(S_n=y)\right|\leq Cn^{\frac{-1-d}{2}}
\end{equation}
and
\begin{equation}\label{eq:second_der_fr}
\left|\prob(S_n=x)+\prob(S_n=w)-2\prob(S_n=y)\right|\leq Cn^{\frac{-2-d}{2}}.
\end{equation}
In addition, for every $x$, $y$ and $w$ and $o$  s.t. there exist $i\neq j$ such that
$x-y=w-v=e_i$ and $x-w=y-o=e_j$,
\begin{equation}\label{eq:mix_der_fr}
\left|\prob(S_n=x)+\prob(S_n=o)-\prob(S_n=y)-\prob(S_n=w)\right|<Cn^{\frac{-2-d}{2}}
\end{equation}
\end{claim}
We now use Claim \ref{claim:sum_der}. We will prove it later.
For $k$ and $l$ in $\N$, we let $B(l,k)$ be the event that $\langle X_{\tau_k},e_1\rangle=l$, we let
$B(l)=\cup_k B(l,k)$ and we let $\hat{B}(l)$ be the event 
\[
\hat{B}(l):=B(l)\cap\bigcap_{j=l+1}^{N^2}B^c(j).
\]
In addition, we define $Z_l=X_{T_l}$.
Then, for $x$ and $y$ s.t. $\|x-y\|_1=1$ and
$\langle x,e_1\rangle = \langle y,e_1\rangle = l$,
\begin{eqnarray*}
& &\annealedP^{z_1}(Z _l=x|\hat{B}(l))-\annealedP^{z_1}(Z_l=y|\hat{B}(l))\\
&=&\annealedP^{z_1}(Z _l=x|B(l))-\annealedP^{z_1}(Z_l=y|B(l))\\
&=&\frac{1}{\annealedP^{z_1}(B(l))}\sum_{k}\annealedP^{z_1}(X_{\tau_{k}}=x)-\annealedP^{z_1}(X_{\tau_{k}}=y).\\
\end{eqnarray*}
Let
\[
M=\frac{l}{\annealedE^{z_1}\left[\langle X_{\tau_2}-X_{\tau_1}, e_1\rangle\right]}.
\]

For $x$ and $y$ satisfying $\langle x,e_1\rangle = \langle y,e_1\rangle = l$ and $\|x-y\|=1$, and for $k\in\N$, 
we now estimate
\[
\left|
\annealedP^{z_1}(X_{\tau_{k}}=x)-\annealedP^{z_1}(X_{\tau_{k}}=y)
\right|.
\]

We consider two different cases: $k\geq M$ and $k<M$. Assume first that $k\geq M$.
Then either $\langle X_{\tau_{[k/2]}},e_1\rangle\leq l/2$ or
$\langle X_{\tau_k}-X_{\tau_{[k/2]}},e_1\rangle\leq l/2$. So
\begin{eqnarray}
\nonumber
&&\left|
\annealedP^{z_1}(X_{\tau_{k}}=x)-\annealedP^{z_1}(X_{\tau_{k}}=y)
\right|\\
\nonumber
&\leq&  \annealedP^{z_1}(X_{\tau_{k}}=x\ ;\ \langle X_{\tau_{[k/2]}},e_1\rangle\leq l/2 )\\
\label{eq:tk2}
&&\ -\ \annealedP^{z_1}(X_{\tau_{k}}=y\ ;\ \langle X_{\tau_{[k/2]}},e_1\rangle\leq l/2 )\\
\nonumber
&+& \annealedP^{z_1}(X_{\tau_{k}}=x ; \langle X_{\tau_k}-X_{\tau_{[k/2]}},e_1\rangle\leq l/2)\\
\label{eq:tk-k2}
&&\ -\ \annealedP^{z_1}(X_{\tau_{k}}=y ; \langle X_{\tau_k}-X_{\tau_{[k/2]}},e_1\rangle\leq l/2 ).
\ \ \ \ \ \ 
\end{eqnarray}
We now estimate \eqref{eq:tk2}. \eqref{eq:tk-k2} is estimated the same way.
\begin{eqnarray}
\nonumber
&&
\annealedP^{z_1}(X_{\tau_{k}}=x\ ;\ \langle X_{\tau_{[k/2]}},e_1\rangle\leq l/2 )
-\annealedP^{z_1}(X_{\tau_{k}}=y\ ;\ \langle X_{\tau_{[k/2]}},e_1\rangle\leq l/2 )\\
\nonumber
&=&
\sum_{w:\langle w,e_1\rangle\leq l/2}
\annealedP^{z_1}(X_{\tau_{[k/2]}}=w)
\big[
\annealedP^{z_1}(X_{\tau_{k}}=x|X_{\tau_{[\frac{k}{2}]}}=w)-
\annealedP^{z_1}(X_{\tau_{k}}=y|X_{\tau_{[\frac{k}{2}]}}=w)
\big]\\
\nonumber
&\leq&
Ck^{\frac{-1-d}{2}} \sum_{w:\langle w,e_1\rangle\leq l/2}
\annealedP^{z_1}(X_{\tau_{[k/2]}}=w)\\
\label{eq:kd12}
&=&
Ck^{\frac{-1-d}{2}} \annealedP^{z_1}(\langle X_{\tau_{[k/2]}},e_1\rangle \leq l/2)
\end{eqnarray}

Where the inequality follows from
\eqref{eq:first_der_fr}.

With a similar calculation for $k < M$, we get that
\begin{eqnarray}
\nonumber
&&
\left|\annealedP^{z_1}(Z _l=x|\hat{B}(l))-\annealedP^{z_1}(Z_l=y|\hat{B}(l))\right|\\
\nonumber
&\leq&
C\sum_{k=1}^M k^{\frac{-1-d}{2}}\annealedP^{z_1}\big(\langle X_{\tau_{[k/2]}},e_1\rangle\geq l/2\big)\\
\label{eq:arriveatl1}
&+&
C\sum_{k=M}^\infty k^{\frac{-1-d}{2}}\annealedP^{z_1}\big(\langle X_{\tau_{[k/2]}},e_1\rangle\leq l/2\big).
\end{eqnarray}
{}From Lemma \ref{lem:regrad} we learn that $X^{\star(n)}$ has (in particular) a finite $2d$ moment, and from standard estimates on sums of i.i.d. variables (namely that the $2d$ moment of the sum of $k$ i.i.d. mean zero variables grows like $O(k^d)$), we get that for $k<M$
\[
\annealedP^{z_1}\big(\langle X_{\tau_{[k/2]}},e_1\rangle \geq l/2\big)\leq 
\min\left[1,C\frac{k^d}{(M-k)^{2d}}\right],
\]
and for $k\geq M$
\[
\annealedP^{z_1}\big(\langle X_{\tau_{[k/2]}},e_1\rangle\leq l/2\big)\leq
\min\left[1,C\frac{k^d}{(k-M)^{2d}}\right]=\min\left[1,C\frac{k^d}{(M-k)^{2d}}\right]
.\]

Combining this with \eqref{eq:arriveatl1} we get that
\begin{eqnarray}
\nonumber
&&\left|\annealedP^{z_1}(Z _l=x|\hat{B}(l))-\annealedP^{z_1}(Z_l=y|\hat{B}(l))\right|\\
\nonumber
&\leq& C\sum_{k=1}^M k^{\frac{-1-d}{2}}\min\left[1,C\frac{k^d}{(M-k)^{2d}}\right]\\
\nonumber
&+& C\sum_{k=M}^\infty k^{\frac{-1-d}{2}}\min\left[1,C\frac{k^d}{(k-M)^{2d}}\right]\\
\label{eq:arriveatlehdhzi}
&\leq& C\sum_{k=1}^{[M/2]} k^{\frac{-1-d}{2}}\min\left[1,C\frac{k^d}{(M-k)^{2d}}\right]\\
\label{eq:arriveatlhzisrs}
&+& C\sum_{k=[M/2]+1}^{M-[M^{1/2}]} k^{\frac{-1-d}{2}}\min\left[1,C\frac{k^d}{(M-k)^{2d}}\right]\\
\label{eq:arriveatlsrssrs}
&+& C\sum_{k=M-[M^{1/2}]+1}^{M+[M^{1/2}]} k^{\frac{-1-d}{2}}\min\left[1,C\frac{k^d}{(M-k)^{2d}}\right]\\
\label{eq:arriveatlsrsinf}
&+& C\sum_{k=M+[M^{1/2}]+1}^\infty k^{\frac{-1-d}{2}}\min\left[1,C\frac{k^d}{(k-M)^{2d}}\right]\\
&\leq& Cl^{-\frac{d}{2}}.\label{eq:arriveatl}
\end{eqnarray}

To see the inequality \eqref{eq:arriveatl}, we bound each of the four sums \eqref{eq:arriveatlehdhzi}--\eqref{eq:arriveatlsrsinf} by $Cl^{-\frac{d}{2}}$.
For the expression in \eqref{eq:arriveatlehdhzi}, we note that we have $O(l)$ summands, each of which is bounded by $O(l^{-d})$, so the sum is bounded by $O(l^{1-d})$.
The expression in \eqref{eq:arriveatlhzisrs} is (up to a constant) bounded by
\[
\int_{M/2}^{M-\sqrt{M}}x^{\frac{d-1}{2}}(M-x)^{-2d}dx
\leq M^{\frac{d-1}2}\int_{\sqrt{M}}^{M/2}y^{-2d}dy
\leq CM^{\frac{d-1}2}\sqrt{M}^{1-2d}
=O(l^{-\frac d2}).
\]
The expression in \eqref{eq:arriveatlsrssrs} contains $O(\sqrt{l})$ summands, each of which is $O(l^{\frac{-1-d}{2}})$,
so the sum is $O(l^{-\frac d2})$.
The expression in \eqref{eq:arriveatlsrsinf} is taken care of the similarly to the one in \eqref{eq:arriveatlhzisrs} --- it is bounded by a constant times the integral
\begin{eqnarray*}
\int_{M+\sqrt{M}}^\infty x^{\frac{d-1}{2}}(x-M)^{-2d}dx
&=&
\int_{M+\sqrt{M}}^{2M} x^{\frac{d-1}{2}}(x-M)^{-2d}dx
+\int_{2M}^\infty x^{\frac{d-1}{2}}(x-M)^{-2d}dx\\
&\leq& CM^{\frac{d-1}{2}}\int_{\sqrt{M}}^M y^{-2d}dy
+C\int_M^\infty y^{\frac{d-1}2-2d}dy
=O(l^{-\frac d2}),
\end{eqnarray*}
where we substituted, for both integrals, $y=x-M$.

Let $H_l=\{w:\langle w,e_1 \rangle=l\}$.
Now,
\begin{eqnarray*}
\annealedP^{z_1}(u=x)=\sum_{l\leq N^2}\annealedP^{z_1}(\hat{B}(l))
\sum_{w\in H_l}\annealedP^{z_1}\left(X_{T_l}=w\ \left|\ \hat{B}(l)\right.\right)
\annealedP^{z_1}\left(u=x\ \left|\ \hat{B}(l)\ ;\ X_{T_l}=w\right.\right)
\end{eqnarray*}
and, using shift invariance and the fact that we condition on the occurrence of a regeneration at $l$, for $y=x+e_j$ ($j\neq 1$),
\begin{eqnarray*}
\annealedP^{z_1}(u=y)=\sum_{l\leq N^2}\annealedP^{z_1}(\hat{B}(l))
\sum_{w\in H_l}\annealedP^{z_1}\left(X_{T_l}=w+e_j\ \left|\ \hat{B}(l)\right.\right)
\annealedP^{z_1}\left(u=x\ \left|\ \hat{B}(l)\ ;\ X_{T_l}=w\right.\right)
\end{eqnarray*}
Noting that due to shift invariance,
\begin{eqnarray*}
\sum_{w\in H_l}\annealedP^{z_1}\left(u=x\ \left|\ \hat{B}(l)\ ;\ X_{T_l}=w\right.\right)=1,
\end{eqnarray*}

we get that 
\begin{eqnarray*}
\nonumber
\left|\annealedP^{z_1}(u=x)-\annealedP^{z_1}(u=y)\right|\\
\leq
\sum_{l\leq N^2}\annealedP^{z_1}(\hat{B}(l))\max_{w\in H_l}
\left|
\annealedP^{z_1}\left(X_{T_l}=w+e_j\ \left|\ \hat{B}(l)\right.\right)
-\annealedP^{z_1}\left(X_{T_l}=w\ \left|\ \hat{B}(l)\right.\right)
\right|
\end{eqnarray*}

and breaking the last sum 
to $l\leq\frac{N^2}{2}$ and $l>\frac{N^2}{2}$, then
controlling the former using Lemma \ref{lem:regrad} and the latter using \eqref{eq:arriveatl},
we get part (\ref{item:first_der}) of the lemma. Part (\ref{item:zero_der})
follows from the 
exact same calculations using \eqref{eq:zero_der_fr}. 
Part (\ref{item:first_der_ort}) follows from \eqref{eq:first_der_fr} similarly.
To see Parts (\ref{item:second_der}) and (\ref{item:mix_der}),
we run the same calculation with one main difference:
When we do the calculation equivalent to \eqref{eq:kd12}, instead of \eqref{eq:first_der_fr}, we use
\eqref{eq:second_der_fr} (for Part (\ref{item:second_der})) or \eqref{eq:mix_der_fr}
(for Part (\ref{item:mix_der})).
We then get a factor of $k^{\frac{-2-d}{2}}$ instead of $k^{\frac{-1-d}{2}}$ (this is because
the inequalities \eqref{eq:second_der_fr} and \eqref{eq:mix_der_fr} give this extra factor of $\sqrt{k}$), and we
continue to carry this factor of $\sqrt{k}$ all the way through.


\end{proof}

\begin{proof}[Proof of Claim \ref{claim:sum_der}]

The proof is a standard Fourier calculation, and therefore we do not give complete details.
By the assumptions, the characteristic function $\chi$ of $Y_2$ has period $2\pi$ in every coordinate. In addition, since $\prob(Y_2=v)>0$ and $\prob(Y_2=v+e_i)>0$ for $i=1,\ldots,d$, we get that there exists $D>0$ and $\delta>0$ such that
\begin{enumerate}
\item
$|\chi(x)|<1-D$ for every $x\in[-\pi,\pi]^d$ such that $\|x\|\geq \delta$, and
\item
$|\chi(x)|<1-D\|x\|^2$ for every $x$ such that $\|x\|<\delta$.
\end{enumerate}

Let $S^\prime=\sum_{k=2}^nY_k$ and let $A=\{x:\|x\|<\delta\}$. Now, to see \ref{eq:zero_der_fr}, we note that
\begin{eqnarray*}
\prob(S^\prime=z) &=& \frac{1}{(2\pi)^d}\int_{[-\pi,\pi]^d}e^{-i\langle x,z\rangle}\chi^{n-1}(x)dx\\
\leq \int_{[-\pi,\pi]^d}|\chi(x)|^{n-1}dx &\leq& (1-D)^{n-1}+\int_{A}(1-D\|x\|^2)^{n-1}dx \leq Cn^{-d/2},
\end{eqnarray*}
and convolution with the distribution of $Y_1$ only decreases the supremum.

To see \eqref{eq:first_der_fr}, we see that
\begin{eqnarray*}
|\prob(S^\prime=z)-\prob(S^\prime=z+e_i)|
&=& \frac{1}{(2\pi)^d}\left|\int_{[-\pi,\pi]^d}\big(e^{-i\langle x,z\rangle}-e^{-i\langle x,z+e_i\rangle}\big)
\chi(x)^{n-1}dx\right|\\
&\leq& \frac{1}{(2\pi)^d}\int_{[-\pi,\pi]^d}\big|e^{-i\langle x,z\rangle}-e^{-i\langle x,z+e_i\rangle}\big|
|\chi(x)|^{n-1}dx.
\end{eqnarray*}
Note that $\big|e^{-i\langle x,z\rangle}-e^{-i\langle x,z+e_i\rangle}\big|\leq|\langle x,e_i\rangle|$.

Therefore,

\begin{eqnarray*}
\left|
\prob(S^\prime=z+e_1)-\prob(S^\prime=z)
\right|
&\leq&
\frac{1}{(2\pi)^d}\int_{[-\pi.\pi]^d}|\chi(x)|^{n-1}\langle x,e_1\rangle dx\\
&\leq&
 (1-D)^{n-1}+C\int_Ae^{-Dn\|x\|^2}\langle x,e_1\rangle dx,
\end{eqnarray*}
and substituting $y=x\sqrt{n}$ we get $dx=n^{-d/2}dy$ and $\langle x,e_1\rangle=n^{-1/2}\langle y,e_1\rangle$,
so the last integral is bounded by
\begin{eqnarray*}
&&Cn^{-d/2-1/2}\int_{\sqrt{n}A}e^{-D\|y\|^2}\langle y,e_1\rangle dy\\
&\leq&
Cn^{-d/2-1/2}\int_{\R^d}e^{-D\|y\|^2}\langle y,e_1\rangle dy=O\left(n^{\frac{-d-1}{2}}\right).
\end{eqnarray*}


To see \eqref{eq:second_der_fr}, we note that
$\left|e^{-i\langle x,z+e_1\rangle}+e^{-i\langle x,z-e_1\rangle}-2e^{-i\langle x,z\rangle}\right|\leq \langle x,e_1\rangle^2$
for every $z\in\Z^d$. Then,
\begin{eqnarray*}
&&\left|
\prob(S^\prime=z+e_1)+\prob(S^\prime=z-e_1)-2\prob(S^\prime=z)
\right|\\
&\leq&
\frac{1}{(2\pi)^d}\int_{[-\pi.\pi]^d}|\chi(x)|^{n-1}\langle x,e_1\rangle^2dx\\
&\leq&
 (1-D)^{n-1}+C\int_Ae^{-Dn\|x\|^2}\langle x,e_1\rangle^2dx,
\end{eqnarray*}
and again substituting $y=x\sqrt{n}$, we still get that $dx=n^{-d/2}dy$ but this time $\langle x,e_1\rangle^2=n^{-1}\langle y,e_1\rangle^2$.
Therefore, this time the last integral is bounded by
\begin{eqnarray*}
&&Cn^{-d/2-1}\int_{\sqrt{n}A}e^{-D\|y\|^2}\langle y,e_1\rangle^2dy\\
&\leq&
Cn^{-d/2-1}\int_{\R^d}e^{-D\|y\|^2}\langle y,e_1\rangle^2dy=O\left(n^{\frac{-d-2}{2}}\right).
\end{eqnarray*}

The way to see  \eqref{eq:mix_der_fr} is similar -- this time what we need to notice is that
$\left|
e^{-i\langle x,z\rangle}-e^{-i\langle x,z+e_1\rangle}-e^{-i\langle x,z+e_2\rangle}+e^{-i\langle x,z+e_1+e_2\rangle}
\right|
\leq
\langle x,e_1\rangle\cdot\langle x,e_2\rangle$
and that when we substitute $y=\sqrt{n}x$ we get 
$\langle x,e_1\rangle\cdot\langle x,e_2\rangle=n^{-1}\langle y,e_1\rangle\cdot\langle y,e_2\rangle$.
\end{proof}

\begin{lemma}\label{lem:lbound}
Assume the assumptions \ref{item:assgamma}--\ref{item:assdim} from Page \pageref{item:assgamma}.
Let $\{X_n\}$ be a RWRE starting at the origin.
Let $\sigma^2$ be the (annealed) covariance matrix of $X_{\tau_2}-X_{\tau_1}$, and let $U$ be the expectation of
$X_{\tau_2}-X_{\tau_1}$. Let $\Sigma$ be the inverse matrix of $\sigma^2$. Let $\bar U=\annealedE(X_{T_{\partial\PP(0,N)}})$. Fix $a>0$. There exists a constant $c$ such that
for every $x\in\partial^+\PP(0,N)$, if
\[
(x-\bar U)^{\tiny T}\Sigma(x-\bar U)<a\cdot\frac{N^2}{\langle U,e_1\rangle},
\]
then
\[
\annealedP\big(X_{T_{\partial\PP(0,N)}}=x\big)>
cN^{1-d}e^{-3a}.
\]
\end{lemma}

\begin{proof}
The proof is very similar to that of Lemma \ref{lem:ann_der}, but slightly simpler. We continue to use the
notations $B(l,k)$, $B(l)$ and $\hat B(l)$.
Let $\delta=\delta(a)$ be a small number.
By the local limit Theorem (see, e.g. \cite{durrett}), for $l>N^2/2$ and every $y\in\PP(0,N)$ such that $\langle y,e_1\rangle=l$,
\begin{equation}\label{eq:llt}
\annealedP\big(Z_l=y\,;\,B(l,k)\big)\geq (2\pi\beta)^{-\frac 12}k^{-\frac{d}{2}}
\left(e^{-(y-Uk)^{\tiny T}\Sigma(y-Uk)/2k}-\delta\right)
\end{equation}
where $\beta$ is the determinant of $\sigma^2$.

Fix $l=N^2$ and let $M=\frac{{N^2}}{\langle U,e_1\rangle}$. Then, 
using
\eqref{eq:llt}, for $x\in\partial^+\PP(0,N)$,
\begin{eqnarray*}
\annealedP\big(Z_{N^2}=x\big)&\geq&\annealedP\big(Z_{N^2}=x\,;\,B({N^2})\big)\\
&\geq& \sum_{k=\lceil M-\sqrt{M}\rceil}^{\lceil M+\sqrt{M}\rceil}
\annealedP\big(Z_{N^2}=x\,;\,B({N^2},k)\big)\\
&\geq& (\pi\beta)^{-\frac 12} M^{-\frac{d}{2}}
\sum_{k=\lceil M-\sqrt{M}\rceil}^{\lceil M+\sqrt{M}\rceil}
\left(e^{-\frac{(x-Uk)^{\tiny T}\Sigma(x-Uk)}{M-\sqrt{M}}}-\delta\right)\\
&\geq& (\pi\beta)^{-\frac 12} M^{-\frac{d-1}{2}} 
\left(
e^{-\frac{(x-\bar U)^{\tiny T}\Sigma(x-\bar U)}{M-\sqrt{M}}
-\frac{(\sqrt{M}U)^{\tiny T}\Sigma(\sqrt{M}U)}{M-\sqrt{M}}}
-\delta
\right)\\
&\geq& cN^{1-d}e^{-3a}
\end{eqnarray*}

\end{proof}

\subsection{Quenched exit estimates}\label{subsec:quenched}
In this subsection we show that with very high probability the quenched exit distribution from a basic {\paral}  is similar to the annealed one. This is the only part of the paper that requires the high dimension assumption.

The goal of this subsection is the following proposition:
\begin{proposition}\label{prop:quenched}
Assume the assumptions \ref{item:assgamma}--\ref{item:assdim} from Page \pageref{item:assgamma}.
Fix $0<\theta\leq 1$.
There exists an event $G(N)=G(\theta,N)\subseteq\Omega$ such that $P(G(N))=1-N^{-\xi(1)}$ and such that
for all $\omega\in G(N)$,
\begin{enumerate}
\item For every $z\in\tilde{\PP}(0,N)$,
\begin{equation*}
\quenchedP_\omega^z\big(T_{\partial\PP(0,N)}\neq T_{\partial^+\PP(0,N)}\big)=N^{-\xi(1)}.
\end{equation*}
\item\label{item:toch0} For every $z\in\tilde{\PP}(0,N)$,
\begin{equation}\label{eq:toch0}
\left\|
\quenchedE_\omega^z\left[
X_{T_{\partial\PP(0,N)}}
\right]
-\annealedE^z\left[
X_{T_{\partial\PP(0,N)}}
\right]
\right\|\leq R_3(N).
\end{equation}
\item
For every $z\in\tilde{\PP}(0,N)$ and every ($d-1$ dimensional) cube $Q\subseteq\partial^+\PP(0,N)$ of size length $\left[N^\theta\right]$,
\begin{equation}\label{eq:quenched_exit}
\left|
\quenchedP_\omega^z\left[X_{T_{\partial\PP(0,N)}}\in Q\right]
-\annealedP^z\left[X_{T_{\partial\PP(0,N)}}\in Q\right]
\right|
<N^{(\theta-1)(d-1)-\theta\big(\frac{d-1}{d+1}\big)}.
\end{equation}
\end{enumerate}
\end{proposition}

From Proposition \ref{prop:quenched} we get the following corollary:

\ignore{
Old statement of corollary in rec_bin.
}

\begin{corollary}\label{cor:quenched}
Assume the assumptions \ref{item:assgamma}--\ref{item:assdim} from Page \pageref{item:assgamma}.
Fix $\theta<1/2$ and let $G(N)$ be as in Proposition \ref{prop:quenched}.
Let $\omega\in G(N)$, and let $z\in\tilde{\PP}(0,N)$. Let $D=D(\omega,z)$ be the quenched exit distribution 
from $\PP(0,N)$, and let $\bar{D}=\bar{D}(\omega,z)$ be $D$ conditioned on $\partial^+\PP(0,N)$. Let $\D=\D(z)$ be the annealed exit distribution, and let $\bar\D$ be the annealed exit distribution conditioned on
$\partial^+\PP(0,N)$.
Then
\begin{enumerate}
\item\label{item:from_right} $D(\partial^+\PP(0,N))= 1-N^{-\xi(1)}$.
\item\label{item:y+z} If $X\sim\bar{D}$, then it can be written as $X=Y+Z$, where $\|Z\|\leq (d+1)N^\theta$ a.s. and 
$Y\sim (\bar{\D}+D_2)$, where 
$D_2$ is a signed measure such that
\begin{enumerate}
\item\label{item:mass} $\|D_2\|:=\sum_x|D_2(x)|
\leq\lambda=N^{-{\theta}\frac{d-1}{2(d+1)}}$.
\item\label{item:balanced} $\sum_{x}D_2(x)=0$.
\item\label{item:first_mom} $\sum_{x}xD_2(x)=0$.
\item\label{item:sec_mom} $\sum_{x}|D_2(x)|\|x-E_{\bar\D}\|_1^2  \leq \lambda N^2$,
where $E_{\bar\D}$, a vector in $\R^d$, is the expectation of the probability distribution $\bar\D$.
\end{enumerate}
\end{enumerate}
\end{corollary}

\begin{proof}[Proof of Corollary \ref{cor:quenched}]
Part \ref{item:from_right} is trivial, and therefore we will prove Part \ref{item:y+z}.
Partition $\partial^+\PP(0,N)$ into disjoint cubes
$Q_1,Q_2,\ldots Q_n$ of side-length $N^\theta$. We get $n=R_5(N)^{d-1}N^{(d-1)(1-\theta)}$ such cubes.
For every $1\leq k\leq n$,
\[
|\bar{D}(Q_k)-\bar{\D}(Q_k)| \leq
N^{(\theta-1)(d-1)-\theta\big(\frac{d-1}{d+1}\big)}.
\]
We define $Y^\prime$ as follows: For every $k$, we take $Y^\prime$ to be in $Q_k$ whenever $X\in Q_k$.
Conditioned on the event $Y^\prime\in Q_k$, we take $Y^\prime$ to be independent of $X$, with
\[
\prob(Y^\prime=x|Y^\prime\in Q_k)=\frac{\bar\D(x)}{\bar\D(Q_k)}
\]
for every $x\in Q_k$.
Then clearly $\|X-Y^\prime\|<dN^\theta$.

Therefore, $\|E(Y^\prime)-E(X)\|<dN^\theta$. By \eqref{eq:toch0}, $\|E(X)-E_{\bar\D}\|\leq R_3(N)$ and
thus $\|E(Y^\prime)-E_{\bar\D}\|<(d+1)N^\theta$. Then there exists a variable $U$, independent 
of $Y^\prime$ and $X$, such that 
$\|U\|<(d+1)N^\theta+1$ and $E(Y^\prime+U)=E_{\bar\D}$. Define $Y=Y^\prime+U$. Then Parts \ref{item:balanced} and \ref{item:first_mom} are immediate.

To see Part \ref{item:mass}, we first note that
\[
\prob(Y=x)=\sum_{u\,:\,\|u\|<(d+1)N^\theta+1}\prob(U=u)\prob(Y^\prime=x-u).
\]
Therefore,
\begin{eqnarray}
\nonumber
&&\sum_x|\prob(Y=x)-\bar\D(x)|\\
\label{eq:fmqnc}
&\leq&\sum_{u\,:\,\|u\|<(d+1)N^\theta+1}\prob(U=u)\sum_x|\prob(Y^\prime=x-u)-\bar\D(x)|.
\end{eqnarray}
By Part \ref{item:first_der} of Lemma \ref{lem:ann_der}, for every $x$ and every $u$ such that $\|u\|<(d+1)N^\theta+1$,
\[
|\bar\D(x-u)-\bar\D(x)|\leq C(d+1)N^\theta\cdot N^{-d}=C(d+1)N^{\theta-d}.
\]
Therefore, with $D_2$ as defined in Part \ref{item:y+z} of the corollary,
\begin{eqnarray*}
\sum_x|D_2(x)|&=&\sum_x|\prob(Y=x)-\bar\D(x)| \\
&\leq& \sum_x \big(|\prob(Y^\prime=x)-\bar\D(x)|+C(d+1)N^{\theta-d}\big)\\
&=& \left(\sum_{k=1}^n|\bar{D}(Q_k)-\bar\D(Q_k)|\right)+R_5(N)^{d-1}N^{d-1}\cdot C(d+1)N^{\theta-d} \\
&\leq& C(d+1)R_5(N)^{d-1}N^{\theta-1} + R_5(N)^{d-1}N^{(d-1)(1-\theta)}\cdot
N^{(\theta-1)(d-1)-\theta\big(\frac{d-1}{d+1}\big)}\\
&=& R_5(N)^{d-1}\left(C(d+1)N^{\theta-1}+N^{-\theta\big(\frac{d-1}{d+1}\big)}\right)
\leq R_6(N)N^{-\theta\big(\frac{d-1}{d+1}\big)}<\lambda
\end{eqnarray*}

To see Part \ref{item:sec_mom}, note that $\|x-E_{\bar\D}\|_1\leq dNR_5(N)$ for every $x$ in the support of $D_2$.
Therefore,
\begin{eqnarray*}
\sum_{x}|D_2(x)|\|x-E_{\bar\D}\|_1^2 
&\leq& d^2R_5^2(N)N^2\sum_{x}|D_2(x)|\\
\leq d^2R_5^2(N)N^2\cdot R_6(N)N^{-\theta\big(\frac{d-1}{d+1}\big)}
&\leq& \lambda N^2.
\end{eqnarray*}

\end{proof}

Corollary \ref{cor:quenched} can be formulated slightly differently in the language of couplings.
We need a definition.
\begin{definition}\label{def:close}
For two probability measures $\mu_1$ and $\mu_2$ on $\Z^d$, and for $\lambda<1$ and $k\in\N$ we say that $\mu_2$ is  \underline{$(\lambda,k)$-close} to $\mu_1$  if there exists a joint distribution ("coupling") $\mu$ of three random variables, $Z_1$, $Z_2$ and $Z_0$ such that
\begin{enumerate}
\item\label{item:marg} $Z_1\sim\mu_1$ and $Z_2\sim\mu_2$.
\item\label{item:masbound} $\mu(Z_1\neq Z_0)\leq \lambda$.
\item\label{item:distbound} $\mu(\|Z_0-Z_2\|<k)=1$.
\item\label{item:stmom} $\sum_{x}x[\mu(Z_1=x)-\mu(Z_0=x)]=E_\mu(Z_1)-E_\mu(Z_0)=0$.
\item\label{item:ndmom} $\sum_{x}|\mu(Z_1=x)-\mu(Z_0=x)|\|x-E_\mu(Z_1)\|_1^2  \leq \lambda \var(Z_1)$
\end{enumerate}
\ignore{
Let $C$ be a constant. If in addition 
$\mu(Z_0=x)\geq C\mu(Z_1=x)$ for all $x$ such that $\|x-E(Z_1)\|<\sqrt{\var(Z_1)}$, then we say that $\mu_2$ is $C$-locally close to $\mu_1$.
}
\end{definition}

Using Definition \ref{def:close}, Part \ref{item:y+z} of Corollary \ref{cor:quenched} can be formulated as
saying that if $\omega\in G(N)$, then $\bar{D}$ is $\big(N^{-\theta\frac{d-1}{2(d+1)}},(d+1)N^{\theta}\big)$-close to $\bar{\D}$.
(We need to see that the variance of a $\bar{\D}$ distributed variable is at least at the order of magnitude of $N^2$. This follows, e.g, from the annealed lower bound in Lemma \ref{lem:lbound})

The following claim is immediate and useful.
\begin{claim}\label{claim:interm}
In the language of Definition \ref{def:close}, the distribution of $Z_0$ is $(\lambda,0)$-close to $\mu_1$.
\end{claim}

\ignore{
Another useful corollary is the following.
\begin{corollary}\label{cor:pntqnc}
Let $\omega\in G(n)$, and let $z\in\tilde\PP(0,N)$. Let
$\hat z=\annealedE^z(X_{T_{\partial \PP(0,N)}})$.
\begin{enumerate}
\item Using the language of Part \ref{item:y+z} of Corollary \ref{cor:quenched}, there exists a constant $C$ such that
$\quenchedP_\omega^z(Y=y)\geq CN^{1-d}$
for every
 $y\in\partial^+\PP(0,N)$ satisfying $\|y-\hat z\|\leq N$.
 \item
For every
 $y\in\partial^+\PP(0,N)$ satisfying $\|y-\hat z\|\leq N$,
 \[
 \quenchedP_\omega^z(X_{T_{\partial \PP(0,N)}}=y)\geq CN^{1-d}\eta^{dN^{\theta}},
 \] 
where $\eta$ is the ellipticity constant, as in \eqref{eq:unifelliptic}.
\end{enumerate}
\end{corollary}
}

We now proceed to proving Proposition \ref{prop:quenched}. We start with a version of Azuma's inequality.
Let $\{M_k\}_{k=1}^n$ be a zero mean martingale with respect to a filtration $\{\FF_k\}_{k=1}^n$ on the sample space $\Omega$. For simplicity we denote $M_0=0$ and $\FF_0=\{\emptyset,\Omega\}$. for $k=1,\ldots,n$, let
$D_k=M_k-M_{k-1}$. Define 
\[
U_k=\esssup(|D_k|\ |\ \FF_{k-1})=\lim_{p\to\infty}\big[E(|D_k|^p|\FF_{k-1})\big]^\frac{1}{p}
\]
 and we define the {\em essential variance} of the martingale to be
\begin{equation*}
U:=\esssup\left(\sum_{k=1}^nU_k^2\right).
\end{equation*}

\begin{lemma}\label{lem:azuma}
For every $K$,
\begin{equation*}
\prob(|M_n|>K)\leq 2e^{-\frac{K^2}{2U}}.
\end{equation*}
\end{lemma}

\begin{proof}
The proof is similar to that of Azuma's inequality:
First we show that for every $k$,
\begin{equation}\label{eq:essbnd}
{\bf E} \left(e^{\sum_{j=k}^nD_j}|\FF_{k-1}\right)
\leq
e^{\hetzi\esssup\left(\sum_{j=k}^nU_j^2|\FF_{k-1}\right)}.
\end{equation}
Indeed, for $k=n$ \eqref{eq:essbnd} is clear, and assuming \eqref{eq:essbnd} for
$k+1$, we get
\begin{eqnarray*}
{\bf E} \left(e^{\sum_{j=k}^nD_j}|\FF_{k-1}\right)
&=&
{\bf E} \left( e^{D_k} {\bf E}\left({e^{\sum_{j=k+1}^nD_j}}|\FF_{k}\right)|\FF_{k-1}\right)\\
\leq
{\bf E} \left(e^{D_k}e^{\hetzi\esssup\left(\sum_{j=k+1}^{n}U_j^2|\FF_k\right)}|\FF_{k-1}\right)
&\leq&
{\bf E} \left(e^{D_k}e^{\hetzi\esssup\left(\sum_{j=k+1}^{n}U_j^2|\FF_{k-1}\right)}|\FF_{k-1}\right)\\
=
e^{\hetzi\esssup\left(\sum_{j=k+1}^{n}U_j^2|\FF_{k-1}\right)}{\bf E} \left(e^{D_k}|\FF_{k-1}\right)
&\leq&
e^{\hetzi\esssup\left(\sum_{j=k+1}^{n}U_j^2|\FF_{k-1}\right)}e^{\hetzi U_k^2}\\
&=&
e^{\hetzi\esssup\left(\sum_{j=k}^nU_j^2|\FF_{k-1}\right)}.
\end{eqnarray*}

For $k=0$ this gives us that
\begin{equation*}
{\bf E}\left(e^{M_n}\right)\leq e^{\hetzi U}
\end{equation*}
and that for every $\lambda$,
\begin{equation*}
{\bf E}\left(e^{\lambda M_n}\right)\leq e^{\hetzi \lambda^2U}.
\end{equation*}
Using Markov's inequality once with $\lambda=\frac{K}{U}$ and once with $\lambda=-\frac{K}{U}$ gives the desired result.
\end{proof}

Next we discuss the intersection structure of two independent walks in the same environment.

\begin{lemma}\label{lem:inter}
Assume the assumptions \ref{item:assgamma}--\ref{item:assdim} from Page \pageref{item:assgamma}.
Let $X^{(1)}:=\{X_n^{(1)}\}$ and $X^{(2)}:=\{X_n^{(2)}\}$ be two independent random walks running in the same environment $\omega$.
Let $[X^{(i)}]$ be the set of points visited by $\{X_n^{(i)}\}$.
Then there exists $C$ such that for every $n$,
\begin{equation*}
E\left[
\quenchedP_{\omega,\omega}\left(
\left\{\left|
\left[X^{(1)}\right]\cap\left[X^{(2)}\right]\cap\PP(0,N)\right|>nR_1^d(N) 
\right\}
\ \cap A_N(X^{(1)}) \cap  A_N(X^{(2)})
\right)
\right]<e^{-Cn}
\end{equation*}
\end{lemma}

\begin{proof}
Let $k\geq 0$ be such that $k+R_1(N)<N$. Then from the definition of the event $A_N$, for a random walk $X=\{X_n\}$,
\begin{equation}\label{eq:inslab}
{\bf 1}_{A_N(X)}\cdot \big| \{x\,:\,x\in[X]\,;\,k<\langle x,e_1\rangle<k+R_1(N)\}\big| < 2^dR_1^d(N).
\end{equation}

For every $k$, let $Q^-_k=\PP(0,N)\cap\{x\,:\,\langle x,e_1\rangle<kR_1(N)\}$ and
$Q^+_k=\PP(0,N)\cap\{x\,:\,\langle x,e_1\rangle\geq kR_1(N) \}$. In addition, let
$\hat A_N=A_N\big(X^{(1)}\big)\cap A_N\big(X^{(2)}\big)$.
Using Propositions 3.1, 3.4 and 3.7  of \cite{BZ07}, as well as uniform ellipticity,  and, again, recalling the definition of $A_N$, we see that there exists $\rho>0$ such that for every $k$
\begin{equation}\label{eq:frombz}
E\left[
\quenchedP_{\omega,\omega}\left(
\left\{
\left[X^{(1)}\right]\cap\left[X^{(2)}\right]\cap Q^+_{k+1}=\emptyset 
\right\}
\left| 
\hat A_N\,;\, \left[X^{(1)}\right]\cap Q^-_{k};\left[X^{(2)}\right]\cap Q^-_{k}
\right.\right)
\right]>\rho.
\end{equation}

{\bf Remark:} As stated in \cite{BZ07}, Propositions 3.1, 3.4 and 3.7  of \cite{BZ07} require moment assumptions on the regeneration times. Nevertheless, examining their proofs, all they need are moment assumptions on the number of sites visited before $\tau_1$, and these moment assumptions are satisfied by Lemma \ref{lem:regrad}.

Now, let 
\[
J^{(\mbox{even})}=\bigl\{k\,:\,k\mbox{ is even and } 
\left[X^{(1)}\right]\cap\left[X^{(2)}\right]\cap Q^+_{k}\cap Q^-_{k+1}\neq\emptyset
\bigr\}
\]

and

\[
J^{(\mbox{odd})}=\bigl\{k\,:\,k\mbox{ is odd and } 
\left[X^{(1)}\right]\cap\left[X^{(2)}\right]\cap Q^+_{k}\cap Q^-_{k+1}\neq\emptyset
\bigr\}.
\]

Then, by \eqref{eq:frombz}, conditioned on $\hat A_N$, both $J^{(\mbox{even})}$ and $J^{(\mbox{odd})}$ are dominated 
by a geometric variable with parameter $\rho$.

The lemma now follows when we remember that by \eqref{eq:inslab},

\[
{\bf 1}_{\hat A_N}\cdot
\left|\left[X^{(1)}\right]\cap\left[X^{(2)}\right]\cap\PP(0,N)\right|
\leq 2^d R_1^d(N)\big(J^{(\mbox{even})}+J^{(\mbox{odd})}\big)
\]

\end{proof}

As a corollary we get the following estimate:

\begin{lemma}\label{lem:interquenched}
Assume the assumptions \ref{item:assgamma}--\ref{item:assdim} from Page \pageref{item:assgamma}.
With the same notation as in Lemma \ref{lem:inter},
\begin{equation}\label{eq:interquenched}
P\left[\omega\ :\ 
\quenchedE_{\omega,\omega}\left(\left|
\left[X^{(1)}\right]\cap\left[X^{(2)}\right]\cap\PP(0,N)\right|
\cdot{\bf 1}_{A_N(X^{(1)})\cap A_N(X^{(2)})}
\right)\geq R_2(N)
\right]=N^{-\xi(1)}
\end{equation}
\end{lemma}

Let $J(N)\subseteq\Omega$ be the event that
for every starting point $z$ in the middle third of the \paral,
\begin{equation*}
\quenchedE^{z,z}_{\omega,\omega}\left(\left|
\left[X^{(1)}\right]\cap\left[X^{(2)}\right]\cap\PP(0,N)\right|
\cdot{\bf 1}_{A_N(X^{(1)})\cap A_N(X^{(2)})}
\right)\leq R_2(N).
\end{equation*}
Then, by Lemma \ref{lem:interquenched},
$P(J(N))=1-N^{-\xi(1)}.$

Fix $z\in\tilde\PP(0,N)$. For every $\omega$ and $x\in\PP(0,N)$, we let
\[
H^z(\omega,x):=\quenchedP^z_\omega(x\in [X] \mbox{ and } A_N(\{X_n\}))
\]
be the hitting probability of $x$.
Then for $\omega\in J(N)$
\begin{equation}\label{eq:heak}
\sum_{x\in\PP(0,N)}(H^z(\omega,x))^2\leq R_2(N).
\end{equation}

\begin{lemma}\label{lem:toch}
Assume the assumptions \ref{item:assgamma}--\ref{item:assdim} from Page \pageref{item:assgamma}.
There exists an event $K(N)\subseteq\Omega$ such that $P(K(N))=1-N^{-\xi(1)}$
and for every $\omega\in K(N)$ and $z\in\tilde{\PP}(0,N)$,
\begin{equation}\label{eq:toch}
\left\|
\quenchedE_\omega^z\left[
X_{T_{\partial\PP(0,N)}}
\right]
-\annealedE^z\left[
X_{T_{\partial\PP(0,N)}}
\right]
\right\|\leq 
R_3(N).
\end{equation}
\end{lemma}

\begin{proof}
Define
\begin{eqnarray*}
U(\omega,z):=
\left\| \begin{array}{c}
\quenchedE^z_\omega\left[
X_{T_{\partial\PP(0,N)}}\cdot {\bf 1}_{A_N} \cdot {\bf 1}_{T_{\partial\PP(0,N)}=T_{\partial^+\PP(0,N)}}
\right]\\
-\annealedE^z\left[
X_{T_{\partial\PP(0,N)}} \cdot {\bf 1}_{A_N} \cdot {\bf 1}_{T_{\partial\PP(0,N)}=T_{\partial^+\PP(0,N)}}
| J(N)\right] \end{array}
\right\|.
\end{eqnarray*}
It is sufficient to show that for a large enough set of $\omega$-s,
\begin{equation}\label{eq:toch2}
U(\omega,z)
\leq R_2^{2d+2}(N).
\end{equation}

\eqref{eq:toch2} is sufficient, because for a set $M$ of environments of measure $1-N^{-\xi(1)}$, for every $\omega\in M$
we have that $\quenchedP^z_\omega(T_{\partial\PP(0,N)}=T_{\partial^+\PP(0,N))})=N^{-\xi(1)}$. Since
$\|X_{T_{\partial\PP(0,N)}}\|_\infty<CN^2$, on every $\omega\in M$ the contribution of the event
$\big\{T_{\partial\PP(0,N)}\neq T_{\partial^+\PP(0,N))}\big\}$ to the expectation of $X_{T_{\partial\PP(0,N)}}$ is bounded
by $1$.

To show \eqref{eq:toch2}, we order the vertices in $\PP(0,N)$ lexicographically, $x_1,x_2,\ldots$, with the first coordinate being
the most significant. Let $\FF_n$ be the $\sigma$-algebra on the sample space $\big(J(N)\subseteq\Omega,P(\cdot|J(N))\big)$ which is determined by $\omega|_{x_1,\ldots,x_n}$ and let $\{M_k\}$ be the martingale
\[
M_k:=E\left[\left.
\quenchedE_\omega\left[
X_{T_{\partial\PP(0,N)}}\cdot {\bf 1}_{A_N} \cdot {\bf 1}_{T_{\partial\PP(0,N)}=T_{\partial^+\PP(0,N)}}
\right]
\right|\ \FF_k\right]
\]

Next we calculate $\esssup(M_k-M_{k-1}|\FF_{k-1})$. The argument is similar to the one used in \cite{BZ07}, which is based on ideas from \cite{bolthausen_sznitman}.
Let
\begin{equation*}
B(x):=\{y\ :\ \langle y,e_1\rangle = \langle x,e_1\rangle-1\mbox{ and } \|y-x\|\leq R_1^2(N)\}.
\end{equation*}
Note that if $x$ is visited and $A_N$ holds, then the first visit to the layer
\begin{equation*}
H(x):=\{y\ :\ \langle y,e_1\rangle = \langle x,e_1\rangle-1\}
\end{equation*}
is in $B(x)$.

Therefore,

\begin{eqnarray}
\nonumber
U_k=
\esssup(M_k-M_{k-1}|\FF_{k-1})
&\leq&
R^2(N)\joint(x_k \in [X]\ | \ \FF_{k-1})\\
\nonumber
\leq
R^2(N)\sum_{y\in B(x_k)} \joint(X_{T_{H(x_k)}}=y\ |\ \FF_{k-1})
&=&
R^2(N)\sum_{y\in B(x_k)} \quenchedP_\omega(X_{T_{H(x_k)}}=y)\\
\nonumber
&\leq&
R^2(N)\sum_{y\in B(x_k)} \quenchedP_\omega(y\in[X]),\\ \label{eq:411a}
\end{eqnarray}

where the first inequality follows from the fact that the regeneration containing $x_k$ is of size no more than
$R^2(N)$, and after this regeneration the distribution of the walk is the annealed distribution.
Remembering that $|B(x_k)|\leq 2^dR_2^{d}(N)$ and that every $y$ is in $B(x)$ for at most
$2^dR_2^{d}(N)$ points $x$,

\begin{eqnarray}
\nonumber
\sum_{k=1}^n U_k^2
&\leq&
\sum_{k=1}^n R^4(N) \left[\sum_{y\in B(x_k)} \quenchedP_\omega(y\in[X])\right]^2\\
\nonumber
&\leq&
2^dR_2^d(N)R^{4}(N)\sum_{k=1}^n\sum_{y\in B(x_k)} \quenchedP_\omega(y\in[X])^2\\
\nonumber
&\leq&
2^{2d}R_2^{2d}(N)R^{4}(N)\sum_{y\in\PP(0,N)} H^2(\omega,y)
\leq 2^{2d}R_2^{2d}(N)R^{4}(N)\cdot R_2(N) \leq R_2^{2d+2}(N).\\ \label{eq:411b}
\end{eqnarray}

Therefore, by Lemma \ref{lem:azuma}, 
\begin{eqnarray*}
P\left(\left.\omega\ :\ 
U(\omega)
>R_2^{2d+2}(N) \ \right| \ J(N) \right)
<2e^{-\frac{R_2^{4d+4}(N)}{2R_2^{2d+2}(N)}}=N^{-\xi(1)}.
\end{eqnarray*}

\eqref{eq:toch2} follows.

\end{proof}

We now estimate the quenched exit distribution from $\PP(0,N)$. Fix a starting point for the walk
$z\in\tilde{\PP}(0,N)$. We start with the following 
lemma. Recall that for every $k$, we define
$H_k$ to be the hyper-plain $H_k=\{v\in\Z^d:\langle v,e_1\rangle=k\}$.

\begin{lemma}\label{lem:dpachotechad}
Assume the assumptions \ref{item:assgamma}--\ref{item:assdim} from Page \pageref{item:assgamma}.
Fix $0<\theta\leq 1$. 
Let $B^{\theta}(N)\subseteq\Omega$ be the event that for every $\frac 25N^2\leq M\leq N^2$ and every ($d-1$ dimensional) cube $Q$ of side length $N^{\theta}$ which is contained in $H_M$,
\begin{equation*}
\left|
\quenchedP_\omega^z\left(X_{T_M}\in Q\ ; \ A_N\right)
-\annealedP^z\left(X_{T_M}\in Q\ ; \ A_N\right)
\right|
\leq N^{(\theta-1)(d-1)}.
\end{equation*}

then for $\theta>\frac{d-1}{d}$,
\[
P\left(B^{\theta}(N)\right) = 1-N^{-\xi(1)}.
\]
\end{lemma}

\begin{proof}[Proof of Lemma \ref{lem:dpachotechad}]
Fix $\theta$, and let $\frac{d-1}{d}<\theta^\prime<\theta$.
Let 
\[
V=\bigl[N^{2\theta^\prime}\bigr].
\]
Fix $\frac 25N^2\leq M\leq N^2$.
Let $v\in H_{M+V}$, and let $\GG$ be the $\sigma$-algebra that is determined by the configuration on 
\[
\PP^M(0,N)=\PP(0,N)\cap\{x\,:\,\langle x,e_1\rangle\leq M\}.
\]

 We are interested in the quantity 
\[
J^{(M)}(v)=E\bigl[
\quenchedP_\omega(X_{T_{M+V}}=v\,;\, A_N)\, |
\GG
\bigr].
\]

\begin{figure}[h]
\begin{center}
\epsfig{figure=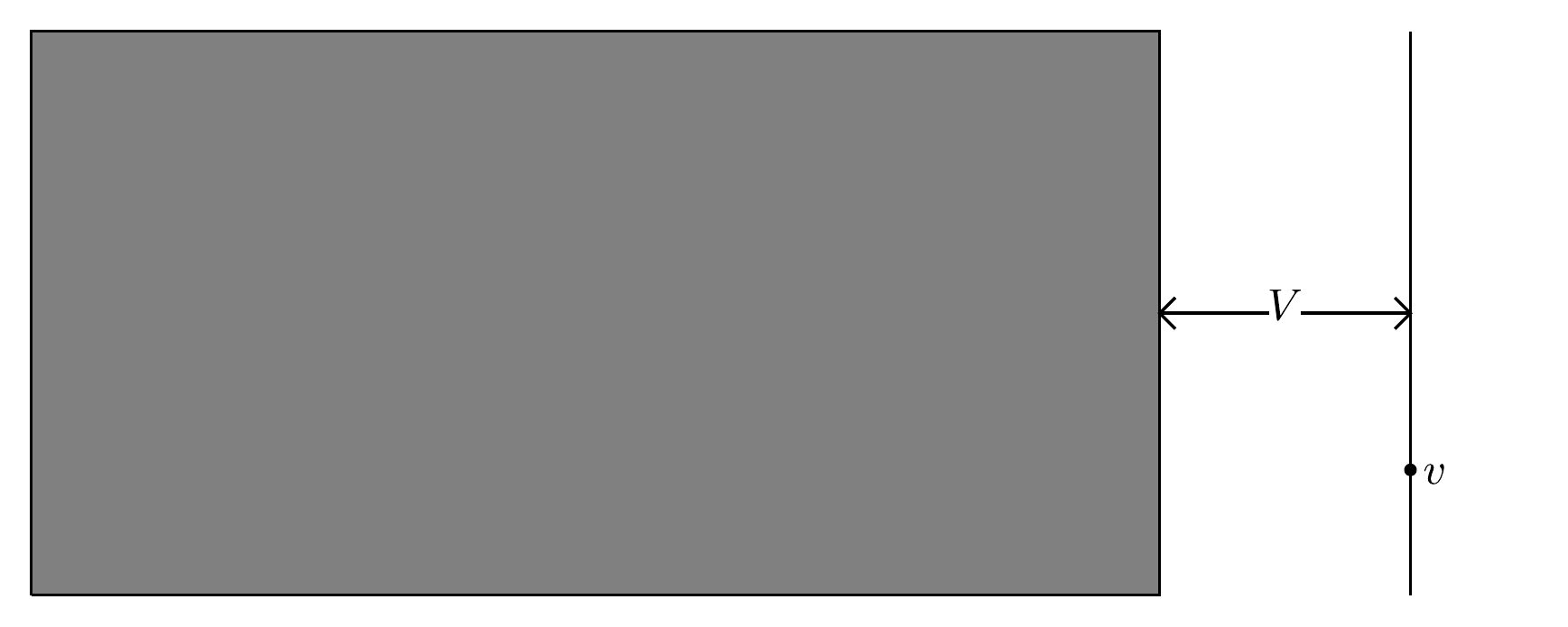, width=10cm}
\caption{\sl 
The quantity $J^{(M)}(v)$ is the probability of hitting the point $v$, conditioned on the environment in the shaded area, and averaged over the environment elsewhere.
}
\label{fig:sigalg}
\end{center}
\end{figure}

Similar to the proof of Lemma \ref{lem:toch}, we let $\{x_i\}_{i=1}^n$ be a lexicographic ordering of the 
vertices in 
$
\PP^M(0,N),
$ and let $\{\FF_i\}$ be the $\sigma$-algebra on $J(N)$ which is determined by $\omega|_{x_1,\ldots,x_i}$.

We consider the martingale 
$
M_i=E\bigl[
\quenchedP_\omega(X_{T_{M+V}}=v\,;\,A_N)\, |
\FF_i
\bigr].
$
In order to use Lemma \ref{lem:azuma}, we will need to bound
$
U_i
=\esssup(M_{i}-M_{i-1}\,|\,\FF_{i-1}).
$
Remember that
$x_i$ is the vertex s.t. $\omega_{x_i}$ is measurable with respect to $\FF_i$ but not with respect to $\FF_{i-1}$.
Then we claim that
\begin{equation}\label{eq:main_est_d-1}
U_i
\leq CR(N)E\big[\quenchedP_\omega(x_i\mbox{ is hit })\,|\,\FF_{i-1}\big]V^{-d/2}.
\end{equation}
We now show the main estimate \eqref{eq:main_est_d-1}.
Let $\omega^\prime$ be an environment that agrees with $\omega$ everywhere except, possibly, $x_i$.
We let $\prob$ be the distribution of a walk that follows the law $\omega$ on $\{x_k:k\leq i\}$ and the annealed distribution on $\Z^d\setminus\{x_k:k\leq i\}$. Equivalently, let $\prob^\prime$ be the distribution of a walk that follows the law $\omega^\prime$ on $\{x_k:k\leq i\}$ and the annealed distribution on $\Z^d\setminus\{x_k:k\leq i\}$.
More precisely, for an event $B\subseteq (\Z^d)^\N$ on the space of possible paths for the walk,
\[
\prob(B)=\joint(B\times\Omega | \omega_{x_1},\ldots,\omega_{x_i}; A_N),
\]
and equivalently for $\prob^\prime$. Then
\begin{equation}\label{eq:uibound}
U_i\leq \sup_{\omega^\prime}\big|\prob^\prime(X_{T_{M+V}}=v)-\prob(X_{T_{M+V}}=v)\big|, 
\end{equation}
where the supremum is taken over all environments $\omega^\prime$ that agree with $\omega$ on $\Z^d\setminus\{x_i\}$.
Note that conditioned on the event that $x_i$ is not visited, the distributions $\prob$ and $\prob^\prime$ are the same.
Now, for both measures $\prob$ and $\prob^\prime$, condition on the event that $x_i$ is visited. Let $u$ be the first
regeneration point after $x_i$. Then $\prob$ and $\prob^\prime$ a.s, $\|u-x_i\|_1<dR(N)$. This follows from the
conditioning on $A_N$.
Therefore, from Parts \ref{item:first_der} and \ref{item:first_der_ort} of Lemma \ref{lem:ann_der} we get that
\[
|\prob(X_{T_{M+V}}=v | x_i \mbox{ is visited})-\annealedP^{x_i}(X_{T_{M+V}}=v)|<CR(N)V^{-d/2}
\]
and 
\[
|\prob^\prime(X_{T_{M+V}}=v| x_i \mbox{ is visited})-\annealedP^{x_i}(X_{T_{M+V}}=v)|<CR(N)V^{-d/2}
\]
Therefore,
\[
U_i\leq CR(N)V^{-d/2}\prob(x_i \mbox{ is visited}).
\]
\eqref{eq:main_est_d-1} follows.

Using \eqref{eq:main_est_d-1}, conditioned on $J(N)$, and based on the same calculation as in \eqref{eq:411a} and \eqref{eq:411b},
\begin{eqnarray*}
U&=&\esssup(\sum_{i=1}^nU_i^2)\\
&\leq& R^6(N)V^{-d}.
\end{eqnarray*}

Therefore, by Lemma \ref{lem:azuma}, for every $v\in H_{M+V}$ and every number $\delta$,
\begin{eqnarray*}
P\left(\left|
E\bigl[
\quenchedP_\omega(X_{T_{M+V}}=v)\,;\, A_N\, |
\GG
\bigr]-
\annealedP(X_{T_{M+V}}=v\,;\, A_N)\,
\right|>\delta
\right)\\
\leq 2P(J(N)^c) + 2e^{-\frac{\delta^2}{2R^6(N)V^{-d}}}
\end{eqnarray*}
In particular, if 
$
\delta=\frac 14N^{1-d}=\frac 14V^{-d/2}V^{\eta},
$
with $\eta=\frac{d+\frac{1-d}{\theta^\prime}}{2}>0$, then we get that
\[
P\left(\left|
E\bigl[
\quenchedP_\omega(X_{T_{M+V}}=v\,\,;\, A_N)\, |
\GG
\bigr]-
\annealedP(X_{T_{M+V}}=v\,;\, A_N\,)
\right|>\frac 14N^{1-d}
\right)
=N^{-\xi(1)}
\]

and

\begin{eqnarray*}
P\left(\left|
E\bigl[
\quenchedP_\omega(X_{T_{M+V}}=v) |
\GG
\bigr]-
\annealedP(X_{T_{M+V}}=v)
\right|>\frac 12N^{1-d}
\right)
\leq P\big(\omega\,:\,\quenchedP_\omega(A_N^c)\geq \frac 14N^{1-d}\big)\\
+
P\left(\left|
E\bigl[
\quenchedP_\omega(X_{T_{M+V}}=v\,\,;\, A_N)\, |
\GG
\bigr]-
\annealedP(X_{T_{M+V}}=v\,;\, A_N\,)
\right|>\frac 14N^{1-d}
\right)
=N^{-\xi(1)}.
\end{eqnarray*}

Let $T(N)$ be the event that 
\[
\left|
E\bigl[
\quenchedP_\omega(X_{T_{M+V}}=v)\, |
\GG
\bigr]-
\annealedP(X_{T_{M+V}}=v)
\right|\leq\frac 12N^{1-d}
\]
for every $\frac 25N^2\leq M\leq N^2$ and every $v\in H_{M+V}\cap \PP(0,2N)$.
Then $P(T(N))=1-N^{-\xi(1)}$. Now consider $\omega\in T(N)$, and fix $\frac 25N^2\leq M\leq N^2$
and a cube $Q$ of side length $N^{\theta}$ which is contained in $H_M$.

We want to estimate
\begin{equation}\label{eq:hefreshcube}
L(Q)=
\left|
\quenchedP_\omega^z\left(X_{T_M}\in Q\ ; \ A_N\right)
-\annealedP^z\left(X_{T_M}\in Q\ ; \ A_N\right)
\right|.
\end{equation}

Let $c(Q)$ be the center of the cube $Q$, and let
$
c^\prime(Q)=c(Q)+V\frac{\vartheta}{\langle \vartheta,e_1\rangle}.
$
Then we let
\[
Q^{(1)}=\{v\in H_{V+M}\,:\,\|v-c^\prime(Q)\|_\infty<\frac 12(0.9)^{1/d}N^{\theta}\}
\]
and
\[
Q^{(2)}=\{v\in H_{V+M}\,:\,\|v-c^\prime(Q)\|_\infty<\frac 12(1.1)^{1/d}N^{\theta}\}.
\]

Then by simple annealed estimates,
\begin{equation}\label{eq:qlobndann}
\annealedP^z(X_{T_{V+M}}\in Q^{(1)})<\annealedP^z(X_{T_{M}}\in Q)+N^{-\xi(1)},
\end{equation}

\begin{equation}\label{eq:qupbndann}
\annealedP^z(X_{T_{V+M}}\in Q^{(2)})>\annealedP^z(X_{T_{M}}\in Q)-N^{-\xi(1)},
\end{equation}

\begin{equation}\label{eq:qlobndque}
E\big[\quenchedP^z_\omega(X_{T_{V+M}}\in Q^{(1)})|\GG\big]<\quenchedP^z_\omega(X_{T_{M}}\in Q)+N^{-\xi(1)},
\end{equation}
and
\begin{equation}\label{eq:qupbndque}
E\big[\quenchedP^z_\omega(X_{T_{V+M}}\in Q^{(2)})|\GG\big]>\quenchedP^z_\omega(X_{T_{M}}\in Q)-N^{-\xi(1)}.
\end{equation}

From the definition of $T(N)$ and \eqref{eq:qlobndann}, \eqref{eq:qupbndann}, \eqref{eq:qlobndque} and \eqref{eq:qupbndque}, it follows that $T(N)\subseteq B^{\theta}(N)$.

Therefore, $P(B^{\theta}(N))\geq P(T(N))=1-N^{-\xi(1)}$.

\end{proof}

Using Lemma \ref{lem:dpachotechad} as a building block, we can get a similar yet weaker result for every choice of
$\theta$.

\begin{lemma}\label{lem:alltheta}
Assume the assumptions \ref{item:assgamma}--\ref{item:assdim} from Page \pageref{item:assgamma}.
For every $0<\theta\leq 1$ and $h$ let $\bar{B}^{(\theta,h)}(N)$ be the event that
for every $z\in\hPP(0,N)$, every $\frac 12N^2\leq M\leq N^2$ and every cube $Q$ of side length $N^{\theta}$ which is contained in $H_M$,
\begin{equation}\label{eq:alltheta}
\quenchedP_\omega^z\left(X_{T_M}\in Q\ ; \ A_N\right)
\leq R_h(N) N^{(\theta-1)(d-1)}.
\end{equation}
Then for every $0<\theta\leq 1$ there exists $h=h(\theta)$ such that  $P(\bar{B}^{(\theta,h)}(N))=1-N^{-\xi(1)}$
\end{lemma}

\begin{proof} 
We prove the lemma by descending induction on $\theta$. 
From Lemma \ref{lem:dpachotechad}, $P\big(\bar{B}^{(\theta,1)}(N)\big)=1-N^{-\xi(1)}$ for every
$1\geq\theta>\frac{d-1}{d}.$ For the induction step, fix $\theta$ and assume that the statement of the lemma holds for some $\theta^\prime$ such that $\theta>\frac{d-1}{d}\theta^\prime$, and let $h^\prime=h(\theta^\prime)$. We write $\rho=\theta/\theta^\prime.$ Let $\sigma$ be the natural shift of $\Z^d$.
Let 
\[
L = \bar{B}^{(\rho,1)}(N) \cap \bigcap_{z\in\PP(0,2N)}
\sigma_z\bigl(\bar{B}^{(\theta^\prime,h^\prime)}([N^\rho])\bigr)\cap T(N,\rho)
,
\]
where 
\[
T(N,\rho)=\{
\omega\in\Omega\,:\, \forall_{v\in\PP(0,N)}\,,\
P_\omega^v\big(X_{T_{\partial\PP(v,[N^\rho])}}\notin\partial^+\PP(v,[N^\rho])\big)<e^{-R_1(N)}
\}.
\]
Clearly, $P(L)=1-N^{-\xi(1)}$. Therefore, all we need to show is that for some $h$ and all $N$ large enough, we have that $L\subseteq\bar{B}^{(\theta,h)}(N).$
To this end we fix $\omega\in L$, fix $z$, fix $\frac 12N^2\leq M\leq N^2$ and fix a cube $Q$ of side length $N^{\theta}$ in $\PP(0,N)\cap H_M$. Let $x$ be the center of $Q$, let
$V=[N^\rho]^2$ and let $x^\prime=x-V\frac{\vartheta}{\langle \vartheta,e_1\rangle}.$

Since 
\[
\omega\in\bigcap_{z\in\PP(0,2N)}
\sigma_z\bigl(\bar{B}^{(\theta^\prime,h^\prime)}([N^\rho])\bigr)
, 
\]
we get that for every $v\in H_{M-V}$,
\begin{equation}\label{eq:katan}
\quenchedP_\omega^v(X_{T_M}\in Q)
<R_{h^\prime}(N) N^{\rho(\theta^\prime-1)(d-1)}=R_{h^\prime}(N) N^{(\theta-\rho)(d-1)}.
\end{equation}
We Remember that by the Markov property and the fact that $\omega\in T(N,\rho)$,

\begin{equation}\label{eq:markov}
\quenchedP_\omega^z(X_{T_M}\in Q)
=\sum_{v\in H_{M-V}\cap\PP(x^\prime,[N^\rho])}
\quenchedP_\omega^z(X_{T_{M-V}}=v)
\quenchedP_\omega^v(X_{T_M}\in Q)+N^{-\xi(1)}
\end{equation}

Now, $H_{M-V}\cap\PP(x^\prime,[N^\rho])$ is the union of $2^{d-1}R_5(N)^{d-1}<R_6(N)$ cubes of side length $N^\rho$.

Since $\omega\in \bar{B}^{(\rho,1)}(N)$, we get that for every cube $Q^\prime$ of side length $N^\rho$ that is contained in $H_{M-V}\cap\PP(0,N)$,
\begin{equation}\label{eq:gadol}
\quenchedP_\omega^z(X_{T_{M-V}}\in Q^\prime)
<R_1(N)N^{(\rho-1)(d-1)}.
\end{equation}

Combining \eqref{eq:katan}, \eqref{eq:markov} and \eqref{eq:gadol}, we get that
\begin{eqnarray*}
&&\quenchedP_\omega^z(X_{T_M}\in Q)\\
&\leq& R_6(N)R_{h^\prime}(N) N^{(\theta-\rho)(d-1)}\cdot
R_1(N)N^{(\rho-1)(d-1)}+N^{-\xi(1)}\\
&\leq& R_h(N)N^{(\theta-1)(d-1)}
\end{eqnarray*}
for $h=\max(6,h^\prime)+1.$

\end{proof}

Next we prove a lemma which significantly strengthens the previous lemma.
For the proof of this lemma we will
use Lemma \ref{lem:alltheta} and a more careful treatment of the proof technique of Lemma 
\ref{lem:dpachotechad}. We start with the following preliminary lemma:

\begin{lemma}\label{lem:distu}
Assume the assumptions \ref{item:assgamma}--\ref{item:assdim} from Page \pageref{item:assgamma}.
Let $\GG$ be the $\sigma$-algebra generated by $\{\omega(z)\,:\,\langle z,e_1\rangle\leq N^2\}.$ Let $\eta>0$, let $V=\big[N^\eta]$ and let $B(N,V)$ be the event that for every $z\in\hPP(0,N)$ and every $v\in H_{N^2+V}$,
\[
\left|
E\bigl[\quenchedP_\omega^z(X_{T_{N^2+V}}=v)\, |\GG\bigr]
-\annealedP^z\bigl[X_{T_{N^2+V}}=v\bigr]
\right|\leq N^{1-d}V^{\frac{1-d}{6}}
.\]
Then $P(B(N,V))=1-N^{-\xi(1)}$.
\end{lemma}

\begin{proof}
Let $v\in H_{N^2+V}$ and let $\theta>0$ be such that $\theta < \frac 1{20}\eta$.
Let $K$ be an integer such that $2^{-K}{N^2}>V\geq 2^{-K-1}{N^2}$, and for $1\leq k< K$ we define
\[
\PP^{(k)}=\PP(0,N)\cap\{x\,:\,2^{-k-1}{N^2}\leq {N^2}-\langle x,e_1\rangle< 2^{-k}{N^2}\}
.\]
In addition we take
\[
\PP^{(K)}=\PP(0,N)\cap\{x\,:\,0\leq {N^2}-\langle x,e_1\rangle< 2^{-K}{N^2}\}
,\]
and
\[
\PP^{(0)}=\PP(0,N)\cap\{x\,:\,{N^2}/2\leq {N^2}-\langle x,e_1\rangle\}
.\]

In addition, 
we define 
\[
F(v)=\{x\in\PP(0,N)\,:\,\|x-u(v,x)\|\leq|\langle v-x,e_1\rangle|^{1/2}R_2(N)\},
\]
where $u(v,x)$ is as in \eqref{eq:uzx}.
Then, for $0\leq k\leq K$, we define
\[
\PP^{(k)}(v)=\PP^{(k)}\cap F(v),
\]
and
\[
\hat{\PP}^{(k)}(v)=\{y: \exists_{x\in\PP^{(k)}(v)}\,\mbox{s.t.}\,\|x-y\|<R_2(N)\}.
\]
Note that $\PP^{(k)}(v)\subseteq\hat{\PP}^{(k)}(v)$.
\begin{figure}[h]
\begin{center}
\epsfig{figure=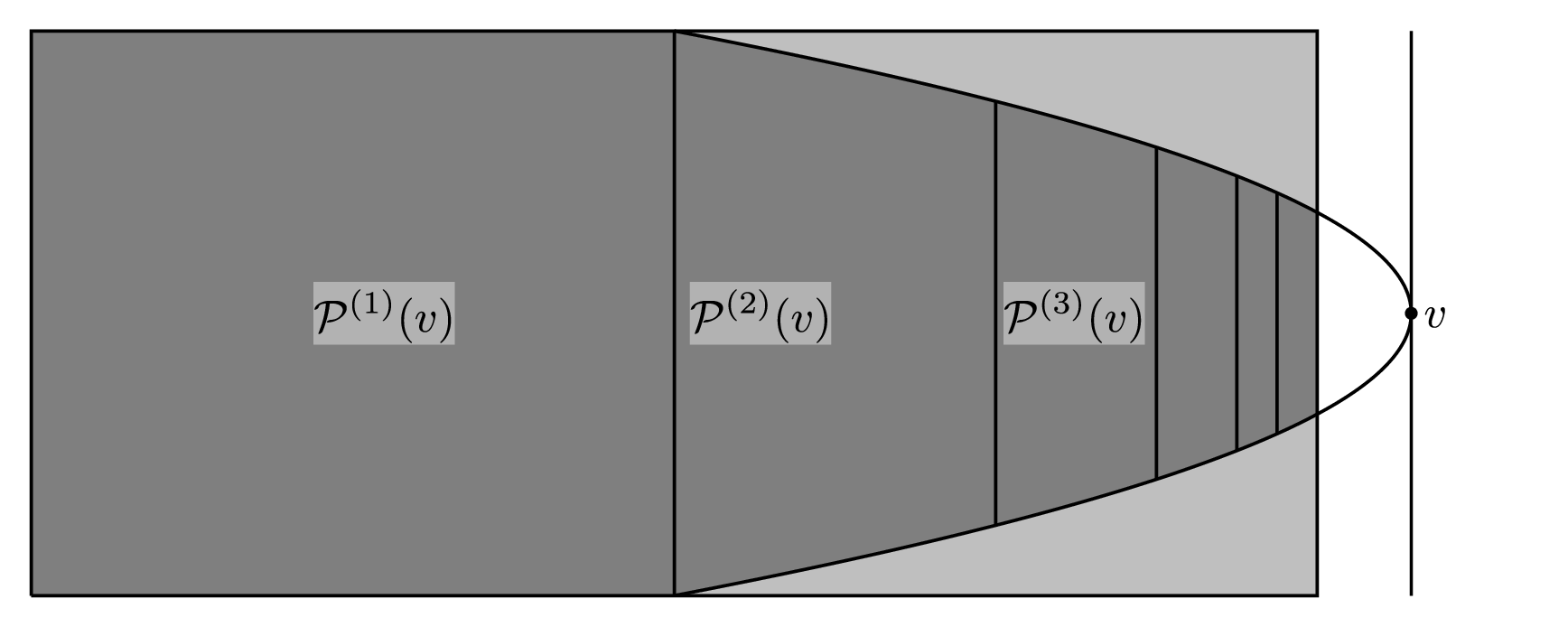, width=12cm}
\caption{\sl 
The darker areas are ${\PP}^{(k)}(v)$ for different values of $k$. The environment in the 
light-gray area has negligible influence on the probability of hitting $v$.
}
\label{fig:withpar}
\end{center}
\end{figure}

Condition on the event $\bar{B}^{(\theta,h)}$, with $h$ such that by Lemma \ref{lem:alltheta}
$P(\bar{B}^{(\theta,h)})=1-N^{\xi(1)}$.

For $0\leq k\leq K$ and $\omega\in\bar{B}^{(\theta,h)},$ we want to estimate

\[
V(k)=E_{\omega,\omega}\left[\big[X^{(1)}\big]\cap\big[X^{(2)}\big]\cap\PP^{(k)}(v)\right].
\]

For $k=0$,

\[
V(0)\leq E_{\omega,\omega}\left[\big[X^{(1)}\big]\cap\big[X^{(2)}\big]\cap\PP(0,N)\right]
\leq R_2(N).
\]

For $k>0$,
\begin{eqnarray}
\nonumber
V(k)&=&\sum_{x\in \PP^{(k)}(v)}\big[\quenchedP^z_\omega(x \mbox{ is visited})\big]^2\\
\nonumber
&\leq& \sum_{x\in \PP^{(k)}(v)}\left[
\sum_{y:\|y-x\|<R(N)}P^z_\omega(X_{T_{\langle y,e_1\rangle}}=y)\right]^2+N^{-\xi(1)}\\
\nonumber
&\leq& R^{2d}(N)
\sum_{y\in\hat{\PP}^{(k)}(v)}\big[P^z_\omega(X_{T_{\langle y,e_1\rangle}}=y)\big]^2+N^{-\xi(1)}\\
&\leq& R_2(N)\sum_{y\in\hat{\PP}^{(k)}(v)}R_h(N)N^{2(1-\theta)(1-d)}\label{eq:bcbt} \\
\nonumber
&\leq&
R_{h+1}(N)N^{2\bigl(\frac{d+1}{2}+(1-\theta)(1-d)\bigr)}2^{-k\left[\frac{d+1}{2}\right]}
\end{eqnarray}

where the inequality \eqref{eq:bcbt} follows from the fact that $\omega\in\bar B^{(\theta,h)}(N)$.

As before, we now use the same filtration $\{\FF_i\}$ as in the proof of Lemma \ref{lem:toch}, and consider the martingale 
$
M_i=E\bigl[
\quenchedP_\omega^z(X_{T_{N^2+V}}=v\,;\, A_N)\, |
\FF_i
\bigr].
$
Again, in order to use Lemma \ref{lem:azuma}, we need to bound
$
U_i
=\esssup(|M_{i}-M_{i-1}|\,|\,\FF_{i-1}).
$
Let $x$ be s.t. $\omega_x$ is measurable with respect to $\FF_i$ but not with respect to $\FF_{i-1}$.
Then $U_i=N^{-\xi(1)}$ if $x\notin F(v)$, while if $x\in F(v)$, then
\[
U_i
\leq R(n)
E[\quenchedP_\omega^z(x\mbox{ is hit })\,|\,\FF_{i-1}]
D(N^2+V-\langle x,e_1\rangle)
\]
where $D(n)$ is the maximal first derivative of the annealed distribution at distance $n$.
By Lemma \ref{lem:ann_der}, $D(N^2+V-\langle x,e_1\rangle)\leq CN^{-d}2^{k\frac{d}{2}}$ for $x\in\PP^{(k)}(v)$.
Therefore,
\begin{eqnarray*}
U&=&\esssup\bigl(\sum_{i}U_i^2\bigr)\\
&\leq& C\sum_{k=0}^K V(K)N^{-2d}2^{kd}+N^{-\xi(1)}\\
&\leq& CR_{h+1}(N)N^{-2d}\\
&+& CR_{h+1}(N)N^{2\bigl(\frac{d+1}{2}+(1-\theta)(1-d)\bigr)-2d}\sum_{k=1}^K2^{kd-k\frac{d+1}{2}}
+N^{-\xi(1)}\\
&\leq& CR_{h+1}(N)\left(N^{-2d}+N^{ 3-3d + 2(d-1)\theta}2^{K\frac{d-1}{2}}\right)\\
&\leq& CR_{h+1}(N)\left(N^{-2d}+N^{ 2-2d + 2(d-1)\theta}V^{-\frac{d-1}{2}}\right)\\
&\leq& CN^{ 2-2d}V^{-\frac{d-1}{6}+\epsilon}
\end{eqnarray*}
for small enough $\epsilon$.

Therefore, using Lemma \ref{lem:azuma}, with probability $1-N^{-\xi(1)}$,

\[
\left|
E\bigl[\quenchedP_\omega^z(X_{T_{N^2+V}}=v\,;\, A_N)\, |\GG\bigr]
-\annealedP\bigl[X_{T_{N^2+V}}=v\,;\, A_N\bigr]
\right|\leq N^{1-d}V^{\frac{1-d}{6}}
.\]

A simple union bound coupled with the fact that $\annealedP(A_N)=N^{-\xi(1)}$ completes the proof of the lemma.

\end{proof}

\begin{lemma}\label{lem:goodbound}
Assume the assumptions \ref{item:assgamma}--\ref{item:assdim} from Page \pageref{item:assgamma}.
For every $0<\theta\leq 1$ let $D^{(\theta)}(N)\subseteq\Omega$ be the event that
for every $z\in\hPP(0,N)$ and every cube $Q$ of side length $N^{\theta}$ which is contained in $\partial^+\PP(0,N)$,
\begin{equation}\label{eq:goodbound}
\left|
\quenchedP_\omega^z\left(X_{T_{\partial \PP(0,N)}}\in Q\ \right)
-\annealedP^z\left(X_{T_{\partial \PP(0,N)}}\in Q\ \right)
\right|
\leq N^{(\theta-1)(d-1)-\theta\big(\frac{d-1}{d+1}\big)}.
\end{equation}
Then $P(D^{(\theta)}(N))=1-N^{-\xi(1)}$
\end{lemma}

\begin{proof}
Take $\frac 34\theta<\theta^\prime<\theta$ and $V=\left[N^{\frac{8\theta^\prime}{d+1}}\right]$.
Then by Lemma \ref{lem:distu} we know that $P\big(B(N,V)\big)=1-N^{-\xi(1)}$. As before,
all we need to show is that $B(N,V)\subseteq D^{(\theta)}(N)$. The way we do this will be completely identical to the last step of the proof of Lemma \ref{lem:dpachotechad}. Let $\omega\in B(N,V)$, and let $Q$ be a cube of side length $N^{\theta}$ which is contained in $\partial^+\PP(0,N)$. Let $x$ be the center of $Q$, and let $x^\prime=x+V\frac{\vartheta}{\langle \vartheta,e_1\rangle}$.

Let $Q^{(1)}$ and $Q^{(2)}$ be $d-1$ dimensional cubes that are contained in $H_{N^2+V}$ and are centered in $x^\prime$, such that the side length of $Q^{(1)}$ is $N^{\theta}-R_3(N)\sqrt{V}$ and the side length of $Q^{(2)}$ is $N^{\theta}+R_3(N)\sqrt{V}$.

Then, on $B(N,V)$, for $i=1,2$
\begin{equation}
\left|
E\bigl[\quenchedP_\omega^z(X_{T_{N^2+V}}\in Q^{(i)})\, |\GG\bigr]
-\annealedP\bigl[X_{T_{N^2+V}}\in Q^{(i)} \bigr]
\right|
\leq |Q^{(i)}|N^{1-d}V^{\frac{1-d}{6}}.
\end{equation}

In addition, exactly as in the proof of Lemma \ref{lem:dpachotechad},

\begin{equation}\label{eq:qlobndannad}
\annealedP^z(X_{T_{V+N^2}}\in Q^{(1)})<\annealedP^z(X_{T_{N^2}}\in Q)+N^{-\xi(1)},
\end{equation}

\begin{equation}\label{eq:qupbndannad}
\annealedP^z(X_{T_{V+N^2}}\in Q^{(2)})>\annealedP^z(X_{T_{N^2}}\in Q)-N^{-\xi(1)},
\end{equation}

\begin{equation}\label{eq:qlobndquead}
E\big[\quenchedP_\omega^z(X_{T_{V+N^2}}\in Q^{(1)})|\GG\big]<\quenchedP^z_\omega(X_{T_{N^2}}\in Q)+N^{-\xi(1)},
\end{equation}
and
\begin{equation}\label{eq:qupbndquead}
E\big[\quenchedP_\omega^z(X_{T_{V+N^2}}\in Q^{(2)})|\GG\big]>\quenchedP^z_\omega(X_{T_{N^2}}\in Q)-N^{-\xi(1)}.
\end{equation}

Therefore, for $\omega\in B(N,V)$,
\begin{eqnarray*}
\left|
\quenchedP_\omega^z\left(X_{T_{\partial \PP(0,N)}}\in Q\ \right)
-\annealedP^z\left(X_{T_{\partial \PP(0,N)}}\in Q\ \right)
\right|\\
\leq
\big(|Q^{(1)}|+|Q^{(2)}|\big)N^{1-d}V^{\frac{1-d}{6}}
+C\big(|Q^{(2)}|-|Q^{(1)}|\big)N^{1-d}+N^{-\xi(1)}\\
\leq
C\left(
N^{(1-\theta)(1-d)}V^{\frac{1-d}{6}} + R_3(N)N^{(1-d)+(d-2)\theta}\sqrt{V}
\right).
\end{eqnarray*}

The lemma follows from the choice of $V$.


\end{proof}

\begin{proof}[Proof of Proposition \ref{prop:quenched}]
Proposition \ref{prop:quenched} follows from Lemma \ref{lem:toch} and Lemma \ref{lem:goodbound}.
\end{proof}

\subsection{Sums of approximate gussians}\label{subsec:sumapprox}
The purpose of this subsection is to prove Lemma \ref{lem:sumapprox} below.
\label{page:defdn}
Let $\DD(N)$ be the annealed distribution starting from zero of
$
X_{T_{\partial\PP(0,N)}}
$
conditioned on $\partial^+\PP(0,N)$.

\ignore{
Old statement in rec_bin with its proof.
}

\begin{lemma}\label{lem:sumapprox} Assume the assumptions \ref{item:assgamma}--\ref{item:assdim} from Page \pageref{item:assgamma}.
Let $0<\lambda<1$ and $n$ be so that $n<\lambda^{-1}$. Let
$K$ be so that $N>K\geq 1$. Let $h\geq 5$. Assume further that  $N>K^4$ and $N>\lambda^{-4}$,
and that $\lambda N>2KnR_{h+1}(N)$. 
Let $\{X_i\}_{i=1}^n$  be random variables such that for every $i$, 
conditioned on $X_1,\ldots,X_{i-1}$, the distribution of $X_i$ 
is $(\lambda,K)$-close to $\DD(N)$.

Let $S=\sum_{i=1}^n X_i$. Then the distribution of $S$ is $(\lambda R_{h+1}(N),2nKR_{h+1}(N))$-close to $\DD(N\sqrt{n})$.

\ignore{
If in addition the conditional distribution of $X_i$ is $C$-local to $\DD(N)$ for some constant $C$, then
$S$ is $C/2$-local to $\DD(N\sqrt{n})$.
}
\end{lemma}

{\bf Remark:} We need the assumptions  \ref{item:assgamma}--\ref{item:assdim} because they give us some control over the distribution $\DD(N)$.

We use the following simple fact, which follows from the decomposition of the annealed RWRE into regenerations.
\ignore{
\begin{claim}\label{claim:annealerr}
Let $\{U_i\}$ be i.i.d. $\DD(N)$, and let $\bar{U}:=\sum_{i=1}^nU_i.$ Then $\bar{U}$ can be represented as $\bar{U}=\hat{U}+U^\prime$ s.t. $\hat{U}\sim\DD(N\sqrt{n})$ and for every $k$,
\[
\prob\left(\frac{U^\prime}{n}>k\right)<Cke^{-ck^{\gamma}}
\]
for some constants $C$ and $c$.

\end{claim}
}

\begin{claim}\label{claim:annealerr}
Assume the assumptions \ref{item:assgamma}--\ref{item:assdim} from Page \pageref{item:assgamma}.
For $j>1$,
let $\hat{\DD}^{(j)}$ be the convolution of $\DD(N)$ and $\DD(N\sqrt{j-1})$.
Let $U\sim \hat{\DD}^{(j)}$.
Then $U$ can be represented as $U=\hat{U}+U^\prime$ s.t. $\hat{U}\sim\DD(N\sqrt{j})$
and for every $k$,
\begin{equation}\label{eq:bndregsq}
\prob\left(\|U^\prime\|>k\right)<Ce^{-ck^{\gamma}} + N^{-\xi(1)}
\end{equation}
for some constants $C$ and $c$. In particular, there exists some constant $C$, independent 
of $N$ and $j$ such that
\begin{equation}\label{eq:bndregsqE}
\|E(U^\prime)\|\leq E(\|U^\prime\|)<C.
\end{equation}
\end{claim}

\begin{proof}
\eqref{eq:bndregsqE} follows immediately from \eqref{eq:bndregsq} (In order to handle the $N^{-\xi(1)}$ error, note that $U^\prime$ is bounded by $3NR_5(N)$), and therefore we shall only prove 
\eqref{eq:bndregsq}.

We will define a coupling between a random variable $U$ which is approximately $\hat{\DD}^{(j)}$ distributed
and a random variable  $\hat U$ which is approximately $\DD(N\sqrt{j})$ distributed
such that
\[
\prob\left(\|U-\hat U\|>k\right)<Ce^{-ck^{\gamma}}+ N^{-\xi(1)}.
\]

We now construct the coupling.

We define an ensemble $\BbbL=\{\U,\T\}$ where $\U$ is a positive integer, and
$\T$ is a nearest neighbor path of length $\U$, taking values in $\Z^d$ and starting at $0$.

Let $\{\BbbL_n=\{\U_n,\T_n\}\}_{n=1}^\infty$ be i.i.d. ensembles, such that 
$\U_1$ is sampled according to the annealed distribution of $\tau_2-\tau_1$, and the path $\T_1$ is distributed according to the annealed distribution of $X_{\tau_1+.}-X_{\tau_1}$, run up to time $\tau_2-\tau_1$ and conditioned on $\tau_2-\tau_1=\U_1$.

Additionally, define $\hat\BbbL_1=\{\hat\U_1,\hat\T_1\}$ and $\hat\BbbL_2=\{\hat\U_2,\hat\T_2\}$ to be two independent and identically distributed ensembles s.t. $\hat\U_1$ is sampled according to the annealed distribution of $\tau_1$ and
$\hat\T_1$ is distributed according to the annealed distribution of $X_{.}$, run up to time $\tau_1$ and conditioned on $\tau_1=\U_1$. In addition, we require independence of $\hat\BbbL_1$ and $\hat\BbbL_2$ and $\{\BbbL_n\}_{n=1}^\infty$.

In other words, $\hat\BbbL_1$ and $\hat\BbbL_2$ are distributed according to the annealed distribution of the first regeneration slab, and $\{\BbbL_n\}$ are distributed according to the annealed distribution of regeneration slabs that are not the first one.

We now construct paths from the ensembles that we defined. The choice of the distribution of the ensembles will guarantee that the paths are distributed according to the annealed RWRE distribution. The variables $U$ and $\hat U$ will be taken to be certain hitting locations of these paths, and the fact that $U$ and $\hat U$ will be built from the same ensembles will make it easy for us to estimate the difference $U-\hat U$.

Let $\Gamma_n=\hat\T_1(\hat\U_1)+\sum_{k=1}^n\T_k(\U_k)$, and let
$T_1=\max(h:\langle e_1,\Gamma_h\rangle<N^2j)$ We take 
\[
\hat U=\Gamma_{T_1}+\T_{T_1+1}\big(\min(i:\langle e_1,\T_{T_1+1}(i)+\Gamma_{T_1}\rangle=N^2j)\big).
\]
Let $T_2=\max(h:\langle e_1,\Gamma_h\rangle<N^2(j-1))$, and
\[
V_1=\Gamma_{T_2}+\T_{T_2+1}\big(\min(i:\langle e_1,\T_{T_2+1}(i)+\Gamma_{T_2}\rangle=N^2(j-1))\big)
\]
Let $\Gamma^\prime_n=\hat\T_2(\hat\U_2)+\Gamma_{T_2+n}-\Gamma_{T_2}$.
Let $T_3=\max(h:\langle e_1,\Gamma^\prime_h\rangle<N^2)$, and
\[
V_2=\Gamma^\prime_{T_3}+\T_{T_2+T_3+1}\big(\min(i:\langle e_1,\T_{T_2+T_3+1}(i)+\Gamma_{T_2+T_3}\rangle=N^2)\big)
\]
We now take $U=V_1+V_2$.

By Lemma \ref{lem:regrad} and Part \ref{item:exit_from_right} of Lemma \ref{lem:ann_der}, 
up to an error of $N^{-\xi(1)}$, the variables $U$ and $\hat U$ are distributed (respectively) according to 
$\hat{\DD}^{(j)}$ and $\DD(N\sqrt{j})$.

The difference $U-\hat U$ is bounded by the sums of the radii of the regeneration slabs $\hat\BbbL_2$,
$\BbbL_{T_2}$, and $\BbbL_{h}$ for $h$ between $T_2+T_3$ and $T_1$. Lemma \ref{lem:regrad} now gives us the desired bound. 
\end{proof}

We also use the following lemma, which is nothing but a second order Taylor expansion.

\begin{lemma}\label{lem:arit}
Let $\mu$ be a finite signed measure on $\Z^d$, and let $f:\Z^d\to\R$. Assume that $m$, $k$, $J$, $L$ in $\N$  and $\varrho\in\Z^d$ are such that
\begin{enumerate}
\item for every $x,y$ such that $x-y\in\{\pm e_i\}_{i=1}^d$, we have that $|f(x)-f(y)|<m$.
\item for every $x,y,z,w$ and $1\leq i,j\leq d$ such that $x-y=z-w=e_i$ and $x-z=y-w=e_j$,
we have that $|f(x)+f(w)-f(y)-f(z)|<k$ (note that if $i=j$ then this is the discrete pure second derivative,
and if $i\neq j$ it is the discrete mixed second derivative).
\item $\sum_x\mu(x)=0$.
\item $\big\|\sum_xx\mu(x)\big\|_1<L$.
\item $\sum_x\|x-\varrho\|_1^2|\mu(x)|<J$.
\end{enumerate}
Then
\[
\left|\sum_{x}\mu(x)f(x)\right|\leq Lm+\frac{1}{2}Jk.
\]
\end{lemma}

\begin{proof}
$\sum_x\mu(x)=0$ and therefore, $\sum_{x}\mu(x)f(x)=\sum_{x}\mu(x)(f(x)+c)$ for every $c$.
Therefore, without loss of generality we may assume that $f(\varrho)=0$.
Let $g:\R^d\to\R$ be the affine function such that $g(\varrho)=f(\varrho)=0$ and $g(\varrho+e_i)=f(\varrho+e_i)$ for $i=1,\ldots,d$. Then
$|f(x)-g(x)|<\frac 12 k\|x-\varrho\|_1^2$ for $x\in\Z^d$.
\ignore{
why?
look at $h=f-g$. it has the exact second derivatives as $f$, and the maximum directional derivative at 
$x$ and along the shortest path o it is bounded by $k\|x\|_1$. Thus, we get a bound
of 
$|f(x)-g(x)|<k\|x\|_1^2$ for $x\in\Z^d$, which is within a constant from what's written.
Whether what's written is really true 
}
Note also that since $\sum_x\mu(x)=0$, we get that $\sum_x(x-\varrho)\mu(x)=\sum_xx\mu(x)$ and thus
$\big\|\sum_x(x-\varrho)\mu(x)\big\|<L$.
Therefore,
\[
\left|\sum_{x}\mu(x) f(x)-\sum_{x}\mu(x) g(x)\right|\leq
\sum_x |\mu(x)||f(x)-g(x)|\leq\frac 12 Jk.
\]
In addition,
\[
\left|\sum_{x}\mu(x)g(x)\right|=\left|g\left(\sum_{x}(x-\varrho)\mu(x)\right)\right|
\leq Lm.
\]
The lemma follows.
\end{proof}

\ignore{
Old proof in rec_bin.
}
\begin{proof}[Proof of Lemma \ref{lem:sumapprox}]

For $k=1,\ldots,n$, conditioned on $X_1,\ldots,X_{k-1}$, the distribution of $X_k$ is
$(\lambda,K)$-close to $\DD(N)$. Therefore there exist variables $\{Y_k\}_{k=1}^n$,
playing the role of $Z_0$ in Definition \ref{def:close}, such that for every $k$, conditioned on
$X_1,Y_1,\ldots,X_{k-1},Y_{k-1}$, the following hold:
\begin{enumerate}
\item\label{item:y_masbound} $\sum_x|\prob(Y_k=x)-\DD(N)(x)|   \leq \lambda$.
\item\label{item:y_distbound} $\prob(\|Y_k-X_k\|<K)=1$.
\item\label{item:y_stmom} $E(Y_k)=E_{\DD(N)}$.
\item\label{item:y_ndmom} $\sum_{x}|\prob(Y_k=x)-\DD(N)(x)|\|x-E_{\DD(N)}\|_1^2  \leq \lambda N^2$.
\end{enumerate}

What we need to show is that there exists a random variable $Y^\prime$ such that
\begin{enumerate}
\item\label{item:yp_masbound} $\sum_x|\prob(Y^\prime=x)-\DD(\sqrt{n}N)(x)|   \leq \lambda R_{h+1}(N)$.
\item\label{item:yp_distbound} $\prob(\|Y^\prime-S\|<2nKR_{h+1}(N))=1$.
\item\label{item:yp_stmom} $E(Y^\prime)=E_{\DD(\sqrt{n}N)}$.
\item\label{item:yp_ndmom} $\sum_{x}|\prob(Y^\prime=x)-\DD(\sqrt{n}N)(x)|\|x-E_{\DD(\sqrt{n}N)}\|_1^2 
\leq \lambda nN^2R_{h+1}(N)$.
\end{enumerate}

\ignore{
If in addition for every $k$, conditioned on $X_1,Y_1,\ldots,X_{k-1},Y_{k-1}$,
\begin{equation}\label{eq:localif}
P(Y_k=x)\geq C\DD(N)(x)
\end{equation}
for all $x$ such that $\|x-E_{\DD(N)}\|<N$, then we need to show that 
\begin{equation}\label{eq:localthen}
P(Y^\prime=x)\geq \frac{C}{2}\DD(N)(x)
.\end{equation}
for all $x$ such that $\|x-E_{\DD(\sqrt{n}N)}\|<\sqrt{n}N$.
}

To this end, we let

\[
S^{(j)}=\sum_{k=j}^nY_k.
\]

First we will show using descending induction, that conditioned on $X_1,\ldots,X_{j-1}$, we can represent $S^{(j)}$ as
$S^{(j)}=Y^{(j)}+Z^{(j)}$ such that $\|Z^{(j)}\|\leq (n-j)R_h(N)$ a.s. and $Y^{(j)}\sim(\DD(N\sqrt{n-j+1})+D_2^{(j)})$ where $D_2^{(j)}$ is a signed measure such that
$\|D_2^{(j)}\|\leq\lambda^{(j)}$ with $\lambda^{(n)}=\lambda$ and
$\lambda^{(j)}\leq \lambda^{(j+1)}+\frac{2}{n-j}\lambda\cdot R_5(N)$ for $j<n$.

For $j=n$ the statement clearly holds, with $Z^{(n)}=0$. We now assume that the statement holds for $j+1$, and prove it for $j$.

Let $\prob$ be the joint distribution of $Y_j$ and $Y^{(j+1)}$ conditioned on $X_1,\ldots,X_{j-1}$. Let
$H=Y_j+Y^{(j+1)}$. For each $z$,
\begin{eqnarray*}
\prob(H=z)
=\sum_{x}\prob(Y_j=x)\prob\big(Y^{(j+1)}=z-x\big|Y_j=x\big)
\end{eqnarray*}
Let $\DD^{(j)}$ be the convolution of $\DD(N\sqrt{n-j})$ and the  $\prob$ distribution of $Y_j$. Then
\begin{eqnarray}
\nonumber
&&\sum_z\big|\prob(H=z)-\DD^{(j)}(z)\big|\\
\nonumber
&\leq&
\sum_z\sum_x \prob(Y_j=x)\left|\prob\big(Y^{(j+1)}=z-x\big|Y_j=x\big)-\DD(N\sqrt{n-j})(z-x)\right|\\
\nonumber
&=&
\sum_{x,y}\prob(Y_j=x)\left|\prob\big(Y^{(j+1)}=y\big|Y_j=x\big)-\DD(N\sqrt{n-j})(y)\right|\\
\label{eq:inhar}
&\leq& \esssup\|D_2^{(j+1)}\|.
\end{eqnarray}

As in Claim \ref{claim:annealerr} let $\hat{\DD}^{(j)}$ be the convolution of $\DD(N)$ and $\DD(N\sqrt{n-j})$.
Then for given $z$,
by Lemma \ref{lem:arit} and Parts \ref{item:second_der} and \ref{item:mix_der} of Lemma \ref{lem:ann_der},
\begin{eqnarray}
\nonumber
&& |\hat{\DD}^{(j)}(z)-\DD^{(j)}(z)|\\
\nonumber
&=&\sum_x \DD(N\sqrt{n-j})(x) \big(\prob(Y_j=z-x)-\DD(N)(z-x)\big)\\
\label{eq:compk} 
&\leq& 
\lambda N^2\cdot N^{-d-1}(n-j)^{\frac{-d-1}{2}}
= \lambda N^{1-d}(n-j)^{\frac{-d-1}{2}}.
\end{eqnarray}

Note that for $z$ such that
$\|z-E_{\hat{\DD}^{(j)}}\|_1>R_5(N)N(n-j)^{\frac{1}{2}}$,
both $\hat{\DD}^{(j)}(z)$ and $\DD^{(j)}(z)$ are bounded by 
\begin{equation}\label{eq:monst}
\exp\left(-
\left(\left.\frac{\|z-E_{\hat{\DD}^{(j)}}\|_1}{R_1(N)}\right)^2
\right/ N^{2(d-1)}(n-j)^{d-1}
\right)\leq e^{-R_4(N)}.
\end{equation}

From \eqref{eq:inhar}, \eqref{eq:compk} and \eqref{eq:monst},
we get that the distribution of $H$ can be presented as
$\hat{\DD}^{(j)}+\bar{D}_2^{(j)}$ such that
\[
\|\bar{D}_2^{(j)}\|\leq \|D_2^{(j+1)}\|+ \lambda(N)R_5(N)(n-j)^{-1}.
\]

By Claim \ref{claim:annealerr}, and again conditioned on $X_1,Y_1,\ldots,X_{j-1},Y_{j-1}$, there exists $Z^\prime(j)$ such that
$\prob(Z^\prime(j))>R_h(N))<\exp(-R_{h-1}(N))$, and the distribution of $H+Z^\prime(j)$ is 
$\DD(N\sqrt{n-j+1})^{(j)}+\bar{D}_2^{(j)}$.

Let 
\[
\bar{H}(j)=H+Z^\prime(j)\cdot{\bf 1}_{\|Z^\prime(j)\|<R_h(N)}.
\]

Then the distribution of $\bar{H}(j)$ is $\DD(N\sqrt{n-j+1})^{(j)}+\hat{D}_2^{(j)}$ with 
\[
\|\hat{D}_2^{(j)}\|\leq\|\bar{D}_2^{(j)}\|+\exp(-R_{h-1}(N))\leq \|D_2^{(j+1)}\|+ 2\lambda(N)R_5(N)(n-j)^{-1}.
\]

\ignore{
The expectation $E(\bar{H}(j))$ satisfies
\[
E(\bar{H}(j))
=
\sum_{k=j}^{n}E(Y_k)+\sum_{k=j}^{n}E
\]
}

We let 
\[
Z^{(j)}=Z^{(j+1)}+Z^\prime(j)\cdot{\bf 1}_{\|Z^\prime(j)\|<R_h(N)},
\]
and $Y^{(j)}=S^{(j)}-Z^{(j)}$.
Then we get that $\|Z^{(j)}\|\leq (n-j)R_{h}(N)$ and the distribution of $Y^{(j)}$ is
$\DD(N\sqrt{n-j+1})+D_2^{(j)})$ where $D_2^{(j)}$ is a signed measure such that
$\|D_2^{(j)}\|\leq\lambda^{(j)}$ with 
\[
\lambda^{(j)}\leq \lambda^{(j+1)}+\frac{2R_5(N)}{n-j}\lambda.
\]

We calculate the expectation of $Y^{(1)}$:
\begin{eqnarray*}
E(Y^{(1)})=E(S^{(1)})-E(Z^{(1)})=nE(Y_1)-E(Z^{(1)})=nE_{\DD(N)}-E(Z^{(1)}).
\end{eqnarray*}

Therefore, again by Claim \ref{claim:annealerr},
\begin{eqnarray*}
\|E(Y^{(1)})-E_{\DD(\sqrt{n}N)}\|\leq Cn+nR_h(N)<nR_{h+1}(N)
\end{eqnarray*}

As in the proof of Corollary \ref{cor:quenched}, we can find a variable $U$ which is independent of all 
of the variables we have seen so far, such that $\|U\|\leq nR_{h+1}(N)+1$ almost surely and 
$
E(U)=E_{\DD(\sqrt{n}N)}-E(Y^{(1)}).
$

We define $Y^\prime=Y^{(1)}+U$. By the same calculation as in \eqref{eq:fmqnc}, we get that $Y^\prime$
satisfies Parts \ref{item:yp_masbound}, \ref{item:yp_distbound} and \ref{item:yp_stmom}.

\vspace{0.25cm}

Thus, all that is left is to show that $Y^\prime$ also satisfies Part \ref{item:yp_ndmom}.
To this end, Let $D_{2}$ be the signed measure such that
$Y^\prime\sim(\DD(\sqrt{n}N)+D_{2})$.

We are interested in
\[
\sum_{x}|D_2(x)|\|x-E_{\DD(\sqrt{n}N)}\|_1^2.
\]

As a first step,  we estimate

\[
\var(D_{2},i)
:=\sum_z \langle z,e_i\rangle^2D_{2}(z)
\]

for a unit vector $e_i$ with $i\neq 1$.

For $x,y,z\in\Z^d$, we write $\hat x, \hat y, \hat z$ for their projection on the $e_i$ axis.

Let $W$ be a random variable distributed
according to $\DD(\sqrt{n}N)$. By Claim \ref{claim:annealerr}, there exists another random variable
$W^\prime$ such that $W^\prime\sim\DD(N)^{\star n}$ ($\DD(N)^{\star n}$ is the $n$-fold convolution of $\DD(N)$) and
$\prob(\|W-W^\prime\|>nk)<Cn\exp(-ck^{-\gamma})$
for every $k$.
  
By the definition of $Y^\prime$, we know that $U^\prime=Y^\prime-S^{(1)}$ satisfies
$\|U^\prime\|\leq 2nR_{h+1}(N)$.

In addition note that $\cov(Y_j,Y_k)=0$ for $j\neq k$,
and that for every $j$,
\begin{eqnarray*}
&&|\var\big(\langle Y_j,e_i\rangle\big)-\var_{\DD(N)}(\hat x)|\\
&=&\sum_x (\hat x-E_{\DD(N)}(\hat z))^2(\prob(Y_j=x)-\DD(N)(x))\\
&\leq& \sum_x (\hat x-E_{\DD(N)}(\hat z))^2|\prob(Y_j=x)-\DD(N)(x)|\leq \lambda N^2.
\end{eqnarray*}
Therefore, 
\begin{eqnarray}
\nonumber
\left|\var\big(\langle S^{(1)},e_i\rangle\big) - \var\big(\langle W^\prime,e_i\rangle\big)\right|
&=&\left|E\big(\langle S^{(1)},e_i\rangle^2\big) - E\big(\langle W^\prime,e_i\rangle^2\big)\right|\\
\leq \sum_{j=1}^n\left|\var \big(\langle Y_j,e_i\rangle\big)-\var_{\DD(N)}(\hat x)\right|
\label{eq:ywprime}
&\leq& \lambda n N^2.
\end{eqnarray}

Now,
\begin{eqnarray}\label{eq:yprimey}
\nonumber
&&\left|\var\big(\langle Y^\prime,e_i\rangle\big) - \var\big(\langle S^{(1)},e_i\rangle\big)\right|\\
\nonumber
&=&\left|\var\langle S^{(1)}+U^\prime,e_i\rangle - \var\langle S^{(1)},e_i\rangle\right|\\
\nonumber
&\leq& 2\esssup(\|U^\prime\|)\sqrt{\var(S^{(1)})}+\esssup(\|U^\prime\|)^2\\
&\leq& 2Cn^{3/2}R_{h+1}(N)N+n^2R_{h+1}^2(N)
\leq 3Cn^{3/2}R_{h+1}(N)N,
\end{eqnarray}

and

\begin{eqnarray}\label{eq:wprimew}
\nonumber
&&\left|\var\big(\langle W,e_i\rangle\big) - \var\big(\langle W^\prime,e_i\rangle\big)\right|\\
\nonumber
&\leq& N^2n^2\prob\big(\|W-W^\prime\|>nR_5(N)\big)
+2nR_5(N)\sqrt{{\var(W^\prime)}}+2n^2R_5(N)^2\\
&\leq& Cn^{3/2}R_5(N)N.
\end{eqnarray}

From \eqref{eq:ywprime}, \eqref{eq:yprimey} and \eqref{eq:wprimew} and the fact that
$E(Y^\prime)=E(W)$, we get that

\begin{eqnarray}\label{eq:vard2}
\nonumber
|\var(D_2,i)|&=&\left|\sum_{x}\langle x,e_i\rangle^2\big(\DD(\sqrt{n}N)(x)-\prob(Y^\prime=x)\big)\right|\\
\nonumber
&=&\left| E\big(\langle W,e_i\rangle^2\big) - E\big(\langle Y^\prime,e_i\rangle^2\big) \right|\\
\nonumber
&=&\left| \var\big(\langle W,e_i\rangle\big) - \var\big(\langle Y^\prime,e_i\rangle\big) \right|\\
&\leq& \lambda n N^2 + 4Cn^{3/2}R_h(N)N
\leq 2\lambda n N^2.
\end{eqnarray}

We now decompose the measure $D_{2}$ into its positive and negative parts
$D_{2}^+$ and $D_{2}^-$. 
We need to bound
\[
\sum_x (\hat x-E_{\DD(\sqrt{n}N)})^2 |D_{2}|=\sum_x (\hat x-E_{\DD(\sqrt{n}N)})^2 D_{2}^+
+ \sum_x (\hat x-E_{\DD(\sqrt{n}N)})^2 D_{2}^-.
\]
We know that 
\begin{eqnarray}\label{eq:from_var}
\nonumber
\left| \sum_x (\hat x-E_{\DD(\sqrt{n}N)})^2 D_{2}^+(x) 
- \sum_x (\hat x-E_{\DD(\sqrt{n}N)})^2 D_{2}^-(x)\right|\\
=\left| \sum_x \hat x^2 D_{2}^+(x) - \sum_x \hat x^2 D_{2}^-(x)\right|
=|\var(D_{2},i)|\leq 2\lambda nN^2.
\end{eqnarray}

In addition, note that $D_{2}^-(x)\leq\DD(\sqrt{n}N)(x)$ for all $x$, and therefore
\[
D_{2,n}^-(x)<e^{-(x-E_{\DD(\sqrt{n}N)})^2/CnN^2R_1(N)}.
\]
Combined with the fact that $\|D_{2}^-\|\leq\|D_{2}\|\leq\lambda R_{h}(N)$, we get that
\begin{equation}\label{eq:boundforminus}
\sum_x (\hat x-E_{\DD(\sqrt{n}N)})^2 D_{2}^- \leq R_2(N)R_h(N)\lambda nN^2.
\end{equation}

Thus, by \eqref{eq:from_var} and \eqref{eq:boundforminus} we get that
\begin{eqnarray*}
&&\sum_x (\hat x-E_{\DD(\sqrt{n}N)})^2 D_{2}^+ + \sum_x (\hat x-E_{\DD(\sqrt{n}N)})^2 D_{2}^-\\
&\leq& 2\sum_x (\hat x-E_{\DD(\sqrt{n}N)})^2 D_{2}^- +\big|\sum_x (\hat x-E_{\DD(\sqrt{n}N)})^2 D_{2}^+ - \sum_x (\hat x-E_{\DD(\sqrt{n}N)})^2 D_{2}^-\big|\\
&\leq& CR_{h}R_2(N)(N)\lambda nN^2.
\end{eqnarray*}

Therefore, 
\begin{eqnarray*}
\sum_x \|x-E_{\DD(\sqrt{n}N)}\|_1^2|D_{2}(x)|
&\leq& (d-1)\sum_x \| x-E_{\DD(\sqrt{n}N)}\|_2^2|D_{2}(x)|\\
= (d-1)\sum_{i=2}^d \sum_x \langle  x-E_{\DD(\sqrt{n}N)}, e_i \rangle^2|D_{2}(x)|
&\leq& (d-1)^2 R_{h}(N)R_{2}(N)\lambda nN^2\\
&\leq& R_{h+1}(N)\lambda nN^2.
\end{eqnarray*}

\ignore{
Therefore, 


\begin{eqnarray*}
&&\sum_{x}|\prob(Y^\prime=x)-\DD(\sqrt{n}N)(x)|\|x\|_1^2 \\
&=& \sum_{x}|\prob(H_n+U=x)-\DD(\sqrt{n}N)(x)|\|x\|_1^2 \\
&\leq& \sum_{x}\left(|\prob(H_n=x)-\DD(\sqrt{n}N)(x)|+\frac{nR_h(N)}{N^{-d}n^{-d/2}}\right)\|x\|_1^2 \\
&\leq& \sum_{x}\left(|\prob(H_n=x)-\DD(N)^{\star n}(x)|+\frac{2nR_h(N)}{N^{d}n^{d/2}}\right)\|x\|_1^2 \\
&\leq& \sum_x \|x\|_1^2|D_{2,n}(x)|+2R_h(N)Nn^{d/2+2}
\leq CR_{h+1}(N)\lambda nN^2.
\end{eqnarray*}

}

Part \ref{item:yp_ndmom} follows.

\end{proof}

\ignore{
\begin{lemma}\label{lem:pntsum}
Let $X_1,\ldots,X_n$ satisfy the assumptions of Lemma \ref{lem:sumapprox}. Let $Y_1,\ldots,Y_n$ be as in the (beginning) of the proof of Lemma \ref{lem:sumapprox}.
Assume further that for some $C>0$ and every $k$, conditioned on $X_1,Y_1,\ldots,X_{k-1},Y_{k-1}$, for every $y\in\partial^+\PP(0,N)$ such that $\|y-E_{\DD(N)}\|\leq N$, we have $P(Y_k=y)>CN^{1-d}$.

Then for every $y\in\partial^+\PP(0,\sqrt{n}N)$ such that $\|y-E_{\DD(\sqrt{n}N)}\|\leq N$, we have that
$P(Y^\prime=y)>\frac{1}{2}CN^{1-d}n^{\frac{1-d}{2}}$.

\end{lemma}

\begin{proof}
We continue with notations as in the proof of Lemma \ref{lem:sumapprox}. Conditioned on $X_1$ and $Y_1$, we have that $Y^{(2)}\sim(\DD(N\sqrt{n-1})+D_2^{(2)})$, with $\|D_2^{(2)}\|\leq\lambda^{(2)}$. 
Recall that 
\[
\lambda^{(2)}\leq\lambda R_6(N)
\]

\end{proof}
}

\section{Reduction to quenched return probabilities}\label{sec:qrp}
\subsection{Basic calculations}\label{sec:fromsznit}
In this subsection we repeat a calculation from \cite{SznitmanT}.  Our main goal is to control the probability of the event $\tau_1>u$. to this end, we take $L=\left\lceil(\log u)^{\frac{d}{\gamma}}\right\rceil$ and notice that
\[
\annealedP(\tau_1>u)
\leq \annealedP(\tau_1>T_L)+\annealedP(T_L>u)
\leq e^{-(\log u)^d}+\annealedP(T_L>u),
\]
where the last inequality follows from  \eqref{eq:regrad}. Let
$
B_L:=[-L,L]\times[-L^2,L^2]^{d-1}.
$
Then, again by \eqref{eq:regrad}, $\annealedP(T_L\neq T_{\partial B_L})\leq e^{-(\log u)^d}$, and thus it is sufficient to show that 
\begin{equation*}
\annealedP(T_{\partial B_L}>u)<Ce^{-c(\log u)^\alpha}
\end{equation*}
for appropriate constants $C$ and $c$.

On the event $\{T_{\partial B_L}>u\}$, there exists a point $x\in B_L$ that is visited more than $\frac{u}{|B_L|}$ times before the walk leaves $B_L$. Therefore, it is sufficient to show that
\begin{equation}\label{eq:toshowbl}
\annealedP\left(\exists_{x\in B_L} \st T^{\frac{u}{|B_L|}}_x<T_{\partial B_L}\right)<Ce^{-c(\log u)^\alpha},
\end{equation}
where $T_x^k$ is defined to be the $k^{\mbox{th}}$ hitting time of $x$.
Let $G\subseteq\Omega$ be an event. Then,
\begin{equation*}
\label{eq:withstuff}
\annealedP\left(\exists_{x\in B_L} \st T^{\frac{u}{|B_L|}}_x<T_{\partial B_L}\right)
\leq
P(G^c)+\sup_{\omega\in G}
\quenchedP_\omega\left(\exists_{x\in B_L} \st T^{\frac{u}{|B_L|}}_x<T_{\partial B_L}\right),
\end{equation*}
and
\begin{eqnarray*}
\quenchedP_\omega\left(\exists_{x\in B_L}\st  T^{\frac{u}{|B_L|}}_x<T_{\partial B_L}\right)
&\leq&
\sum_{x\in B_L}P_\omega^0
\left(T^{\frac{u}{|B_L|}}_x<T_{\partial B_L}\right)\\
=
\sum_{x\in B_L}P_\omega^0 
\left(T_x<T_{\partial B_L}\right)
P_\omega^x
\left(T^{\frac{u}{|B_L|}-1}_x<T_{\partial B_L}\right)
&\leq&
\sum_{x\in B_L}P_\omega^x
\left(T^{\frac{u}{|B_L|}-1}_x<T_{\partial B_L}\right).
\end{eqnarray*}

Note that due to the strong Markov property,
\begin{eqnarray*}
\quenchedP_\omega^x\left(T^{\frac{u}{|B_L|}-1}_x<T_{\partial B_{L}}\right)
=
\left[\quenchedP_\omega^x\left(T_x<T_{\partial B_{L}}\right)\right]^{\frac{u}{|B_L|}-1},
\end{eqnarray*}
and therefore \eqref{eq:withstuff} will follow 
if we find an event $G$ such that $P(G^c)<\frac 12e^{-(\log u)^\alpha}$ and
for some $\epsilon>0$, every $\omega\in G$  and every $x$,
\begin{eqnarray}
\label{eq:toshowomeg1}
\quenchedP_\omega^x\left(T_{\partial B_{L}}<T_x\right)>u^{\epsilon-1}.
\end{eqnarray}
In turn, we may replace \eqref{eq:toshowomeg1} by 
\begin{eqnarray}
\label{eq:toshowomeg}
\quenchedP_\omega^x\left(T_{\partial B_{2L}(x)}<T_x\right)>u^{\epsilon-1},
\end{eqnarray}
where $B_{2L}(x)$ is the cube of the same dimensions as $B_{2L}$, centered at $x$. The choice of 
$B_{2L}(x)$ is slightly more convenient than $B_{L}$ because now the condition is translation invariant with respect to the choice of $x$. 
\subsection{Definition of the event $G$}\label{sec:defgood}

We now define the event $G$, and show that $P(G^c)<\frac 12e^{-(\log u)^\alpha}$. In Sections 
\ref{sec:auxwalk}, \ref{sec:randirev} and \ref{sec:prfmain} we will show that \eqref{eq:toshowomeg} holds for every 
$\omega\in G$.

Let $\epsilon>0$ be so that
\begin{equation}\label{eq:chooseepsilon}
2d\epsilon < d-\alpha.
\end{equation}

Fix $\psi>0$ so that
\begin{equation}\label{eq:choosepsi}
\psi\leq\frac{\gamma\epsilon}{30d}
\end{equation}

and $\chi>0$ so that
\begin{equation}\label{eq:choosechi}
\chi<\frac{\psi^2}{2}\cdot\frac{d-1}{2(d+1)}.
\end{equation} 
We say that a basic {\paral} $\PP(z,N)$ is {\em good} with respect to the environment $\omega$ if the assertion of
proposition \ref{prop:quenched} holds for every {\paral} of size at least $N^{\chi}$ that is contained in $\PP(z,N)$, with $\theta=\psi/2$.
Otherwise, we say that $\PP(z,N)$ is {\em bad}.
We define our {\em scales} $N_1,\ldots,N_\iota$ as follows:
\begin{enumerate}
\item
$N_1:= \lceil L^\psi\rceil$
\item 
We define $\rho_k=\frac{\chi}{2}+\frac{\chi}{2^k}$.
\item
$N_{k+1}:= N_k\cdot \lceil L^{\rho_k}\rceil$
\item
$\iota$ is defined to be the largest $k$ s.t. $N_k^2<2L$.
\end{enumerate}

For every $k=1,\ldots,\iota$, we let $B_{2L}(k)$ be the set of all $z\in\LL_{N_k}$ such that $\PP(z,N_k)\cap B_{2L}\neq\emptyset$. We now define the event $G$:
We say that the environment $\omega$ is in $G$ if for every $k=1,\ldots,\iota$,
\begin{equation}\label{eq:defgood}
\left|
\left\{
z\in B_{2L}(k)\ :\ \PP(z,N_k)\mbox{ is not good w.r.t. } \omega
\right\}
\right|<
(\log u)^{\alpha+\epsilon}
\end{equation}

\begin{lemma}\label{lem:probgoodevent} For $u$ large enough,
$P(G)\geq 1-\frac 12e^{-(\log u)^\alpha}$.
\end{lemma}

\begin{proof}
Let 
\[
J_k:=
\left|
\left\{
z\in B_{2L}(k)\ :\ \PP(z,N_k)\mbox{ is not good w.r.t. } \omega
\right\}
\right|.
\]
First we note that
\[
P(G^c)\leq\sum_{k=1}^\iota P\left[
J_k\geq
(\log u)^{\alpha+\epsilon}
\right],
\]
and $\iota$ is bounded. Now by Proposition \ref{prop:quenched} and Corollary \ref{cor:quenched}, for given $k$ and $z\in B_{2L}(k)$,
\[
p_k:=P(\PP(z,N_k)\mbox{ is not good})=N_k^{-\xi(1)}=o\left(\left|B_{2L}\right|^{-1}\right).
\]  

By Lemma \ref{lem:parandlat}, we can present $J_k$ as $J_k=J_k^{(1)}+\ldots+J_k^{(9^d)}$, and 
\[
J_k^{(h)}\sim \mbox{Bin}(p_k,D_k)
\]
with $D_k<|B_{2L}|$. Thus for $u$ large enough, $J_k^{(h)}$ is binomial with expected value which is less than 1. Therefore, again assuming that $u$ is large enough,
\[
P\left[J_k^{(h)}>\frac{(\log u)^{\alpha+\epsilon}}{9^d }\right]<
\exp\left(-
\frac{(\log u)^{\alpha+\epsilon}}{9^d }
\right).
\]

Therefore,
\begin{eqnarray*}
P(G^c)
&\leq& P\left(\bigcup_{k=1}^\iota\bigcup_{h=1}^{9^d}J_k^{(h)}\right)\\
\leq 9^d\iota\exp\left(-
\frac{(\log u)^{\alpha+\epsilon}}{9^d }
\right)
&\leq&\frac 12 e^{-(\log u)^\alpha}.
\end{eqnarray*}

\end{proof}

\section{The auxiliary walk}\label{sec:auxwalk}

Fix an environment $\omega\in G$. In this section we define a new random walk $\{Y_n\}$ on the environment $\omega$, whose law is different from that of the quenched random walk $\{X_n\}$ on $\omega$. However, we show an obvious relation between the laws of $\{Y_n\}$ and $\{X_n\}$ that we will exploit in sections \ref{sec:randirev} and \ref{sec:prfmain} in order to prove \eqref{eq:toshowomeg}.

We first give an informal description of the random walk $\{Y_n\}$ in Subsection \ref{sec:auxinf}, then define it properly in Subsection \ref{sec:auxdef}, and then collect some useful facts about it in Subsection \ref{sec:auxprop}

\subsection{Informal description of $\{Y_n\}$}\label{sec:auxinf}
$\{Y_n\}$ is a quenched random walk on $\omega$, which is forced to ``behave well''  in a number of different ways, which we list below.

\begin{enumerate}
\item\label{item:exright}
Once the walk  $\{Y_n\}$ reached the center of certain basic \parals, it is only allowed to exit them through their right boundaries.
\item
If the walk is in a bad \paral, then once it exists the block, it is forced to make a number of steps on the right boundary of the \paral { }that will force the eventual exit distribution to be similar to the annealed distribution. We use Lemma \ref{lem:sumapprox} to control the number of forced steps that are needed. When the walk exits a good basic \paral, no
such correction is necessary, because the distribution is already close enough to the annealed.
\item
Upon leaving the origin the walk is forced to make a number of steps to the right. This together with part (\ref{item:exright}) makes sure that $\{Y_n\}$ leaves $B_{2L}$ before returning to the origin.
\end{enumerate}

The resulting random walk $\{Y_n\}$ is a random walk that, most of the time, behaves locally similarly to the quenched random walk, but behaves globally similarly to the annealed random walk. We will quantify and then use those similarities in order to control the behavior of the quenched walk.

\subsection{Definition of $\{Y_n\}$}\label{sec:auxdef}

The process $\{Y_n\}$ is a nearest neighbor random walk, which starts at $0$ and stops when it reaches $\partial^+B_{2L}$. Below we describe its law.

We first need some preliminary definitions.

For every $j=1,2,\ldots$ and every $k=1,2,\ldots,\iota$, we let $U_k(j)$ be the layer
\[
U_k(j)=H_{jN_k^2}=\{x\ :\ \langle x,e_1\rangle=jN_k^2\}.
\]

We define $T^Y_k(j)=\inf\{n:Y_n\in U_k(j)\}.$ 


For every $x\in B_{2L}$, and for every $k$, 
we define $z(x,k)$ as follows: if $\langle x,e_1\rangle$ is divisible by $N_k^2$, then $z(x,k)$ is a point $z\in\LL_{N_k}$ such that $\langle x,e_1\rangle=\langle z,e_1\rangle$ and $x\in\tilde\PP(z,N_k)$. If more than one such point exists, then we choose one according to some arbitrary rule. If $\langle x,e_1\rangle$ is not divisible by $N_k^2$, then we take $z(x,k)$ to be $0$.


For $x\in B_{2L}$, and for every $k$, we define $\PP^{(k)}(x)=\PP(z(x,k),N_k)$.



For every $x\in B_{2L}$, we define
\begin{equation}\label{eq:levelx}
k(x)=\max
\{k\leq\iota\ :\ 
	\langle x,e_1\rangle=\langle z(x,k),e_1\rangle \mbox{ and }
	\PP(z(x,k),N_k) \mbox{ is 
	good}
\},
\end{equation}
and $k(x)=0$ if no such $k$ exists.

In addition, for a random variable $X$, a distribution $\DD$ and a number $\lambda<1$, we define a \underline{$(\lambda,\DD)$-companion} of $X$ as follows:
Let $\nu$ be the distribution of $X$, and let $K$ be the smallest number such that $\nu$ is $(\lambda,K)$-close to $\DD$ . Let $\mu$ be an arbitrarily chosen coupling of three variables $Z_0,Z_1,Z_2$ demonstrating, as in Definition \ref{def:close}, that $\nu$ is $(\lambda,k)$-close to $\DD$. The roles of the variables $Z_0,Z_1,Z_2$ are exactly as in Definition \ref{def:close}. In particular, $Z_1\sim\DD$ and $Z_2\sim\nu$.
We say that a variable $Y$ is a \underline{$(\lambda,\DD)$-companion} of $X$ if the joint distribution of $X$ and $Y$ is the same as the $\mu$-joint distribution of $Z_2$ and $Z_0$. For every $X$, $\lambda$ and $\DD$ we can construct such companion: For every $x$, on the event $\{X=x\}$, we sample $Y$ according to the $\mu$-distribution of $Z_0$ conditioned on the event $\{Z_2=x\}$.
Similarly, we can define the $(\lambda,\DD)$-companion of $X$ conditioned on a $\sigma$-algebra $\FF$: We work with the conditional distribution of $X$ given $\FF$ instead of the (unconditional) distribution, and proceed as before.
Note that $\|Y-X\|<K$ and that by Claim \ref{claim:interm} the distribution of $Y$ is $(\lambda,0)$-close to $\DD$.

We now simultaneously  define the walk $\{Y_n\}$, its accompanying sequence of times $\{\zeta_m\}$, and random variables $\beta_{k,j}$.
The precise definition of the variables $\{\beta_{k,j}\}$ is postponed to the end of the subsection. However, we make the following comment on $\{\beta_{k,j}\}$ at this point:
For every $j$ and $k$, a.s. $\langle\beta_{k,j},e_1\rangle=0$.


For $j\leq N_1^2$, we define $Y_j=je_1$. In addition, $\zeta_0=0$ and $\zeta_1=N_1^2$.


Given $\zeta_0,\ldots,\zeta_n$ and $\{Y_\ell\ :\ \ell=0,\ldots,\zeta_n\}$, we define
$x^\prime=Y_{\zeta_n}$. Let $k^\prime$ be the largest $k$ such that $x^\prime\in U_k(j)$ for some $j$. Then we let 
$x=x^\prime+\sum_{k=1}^{k^\prime}\beta_{k,j(k)}$, where $j(k)$ is the value of $j$ such that $x^\prime\in U_k(j)$.
We let $\kappa=\|x^\prime-x\|_1+2$ and choose
$\{Y_{\zeta_n},\ldots,Y_{\zeta_n+\kappa-2}\}$ to be a shortest path from $x^\prime$ to $x$.
We then take $Y_{\zeta_n+\kappa-1}=x+e_1$ and $Y_{\zeta_n+\kappa}=x$.
Let $\zeta_n^\prime=\zeta_n+\kappa$.

Let $k=k(x)$. If $k(x)>0$ then $\{Y_\ell\ :\ \ell=\zeta^\prime_n,\ldots,T_{\partial\PP^{(k)}(x)}\}$
is chosen to be a random walk starting at $x$ on the random environment $\omega$ conditioned on the event
$\{T_{\partial\PP^{(k)}(x)}=T_{\partial^+\PP^{(k)}(x)}\}$ and $\zeta_{n+1}=T_{\partial^+\PP^{(k)}(x)}$.
Conditioned on $\omega$, $\zeta^\prime_n$ and $x$,
the path $\{Y_\ell\ :\ \ell=\zeta^\prime_n,\ldots,T_{\partial\PP^{(k)}(x)}\}$ is chosen
independently of the path prior to $\zeta^\prime_n$ and of
$
\big\{\beta_{k,j(k)}\ : \ k \mbox{ and } j \mbox{ are such that }
jN_k^2\leq\langle x,e_1\rangle\big\}.$
If $k=0$ then $\zeta_{n+1}=\zeta^\prime_n+N_1^2$ and for $\zeta^\prime_n<j\leq\zeta_{n+1}$, we take $Y_j=x+(j-\zeta^\prime_n)e_1$.

We define $x^\prime_n=Y_{\zeta_n}$ and $x_n=Y_{\zeta^\prime_n}$.
Note that for every $n$, both $\langle x_n,e_1\rangle$ and $\langle x^\prime_n,e_1\rangle$ are divisible by $N_1^2$ (remember that $\langle \beta_{k,j},e_1\rangle=0$).

All that is now left is to define $\beta_{k,j}$. $\beta_{1,1}$ is simply defined to be the $(0,\DD(N_1))$-companion of the (deterministic) variable $Y_{N_1^2}$.


For $\beta_{k,j}$ for other values of $k$ and $j$,
we first list some conditions under which $\beta_{k,j}$ is zero.

\begin{enumerate}
\item\label{item:1}
If there exist no $n$ such that $\zeta_n=T^Y_{k}(j-1)$ then $\beta_{k,j}=0$.
\item\label{item:infty}
Otherwise, let $n$ be such that $\zeta_n=T^Y_{k}(j-1)$, and let $x=Y_{\zeta^\prime_n}$. 
If $\PP^{(k)}(x)$ is good, then $\beta_{k,j}=0$.
\end{enumerate}

Now assume that neither one of conditions \ref{item:1}--\ref{item:infty} holds. 
For $k=1,\ldots,\iota$ let $\lambda_k=L^{-\chi}R_{5+k}(L)$.

We define $\beta_{k,j}$ recursively - we use the values of
$\{\beta_{k^\prime,j^\prime}\,:\,k^\prime<k, j^\prime=j\frac{N_k^2}{N_{k^\prime}^2}\}$
in the definition of $\beta_{k,j}$.


Let $x=Y_{\zeta^\prime_n}$, where as before $n$ us such that $\zeta_n=T^Y_{k}(j-1)$, and for $k^\prime<k$ let $j(k^\prime)$ be the unique value satisfying $U_{k^\prime}(j(k^\prime))=U_k(j)$.

Let
\[X=Y_{T^Y_k(j)}-x+\sum_{k^\prime=1}^{k-1}\beta_{k^\prime,j(k^\prime)}\]

Recall the definition of $\DD(N)$ from Page \pageref{page:defdn}. Then $\DD(N_k)$ is the annealed distribution of $X_{T_{k}(j)}-x$ for a walk starting at $x$, conditioned on exiting $\PP(x,N_k)$ through the front.

We now take $\hat Z$ to be an (arbitrarily chosen) $(\lambda_k,\DD(N_k))$-companion of $X$, conditioned on 
$\{Y_{\ell}:\ell=1,\ldots,\zeta^\prime_n\}$ and $\omega$, and let $\beta_{k,j}=\hat Z - X$.

%



Thus we defined the process $\{Y_n\}$.

\begin{remark}\label{rem:Y_0}
Note that in our definition, if $\beta_{k,j}\neq 0$ then the distribution of $\hat Z=X+\beta_{k,j}$ is $(\lambda_k,0)$-close to $\DD(N_k)$.
\end{remark}

\subsection{Basic properties of $\{Y_n\}$}\label{sec:auxprop}
We prove a few facts regarding the process $\{Y_n\}$ which we will use in Sections \ref{sec:randirev} and \ref{sec:prfmain}.

\begin{lemma}\label{lem:noreturn}
$\{Y_n\}$ reaches $\partial B_L$ before returning to the origin.
\end{lemma}

\begin{proof}
By the definition, $U_1(1)$ is reached before returning to the origin. Then for every $n$, if $x=Y_{\zeta_n}$, then $\PP^{(k(x))}(x)$ is contained in the positive half space, and $\{Y_n\}$ exits $\PP^{(k(x))}(x)$ through $\partial^+\PP^{(k)}(x)$. Therefore $\{Y_n\}$ cannot 
return to the origin.
\end{proof}

\begin{lemma}\label{lem:betasmall}
For every $k$ and $j$, with probability $1$,
\begin{equation}\label{eq:betasmall}
\beta_{k,j}<L^{4\psi}
\end{equation}
\end{lemma}

\begin{proof}
For $k=1$, the size of the {\paral} $\PP(0,N_1)$ is less than $L^{4\psi}$, and therefore for every $j$, we have that $\beta_{1,j}<L^{4\psi}$.

Now assume that $k\geq 1$. In this case, we assume that there exists $n$ such that
$\zeta_n=T^Y_k(j-1)$, because otherwise $\beta_{k,j}=0$. Let $i$ be such that $U_k(j-1)=U_{k-1}(i)$.

Let $x=Y_{\zeta^\prime_n}$. If $\PP^{(k)}(x)$ is good, then $\beta_{k,j}=0$. Therefore we may assume that $\PP^{(k)}(x)$ is not good. In this case there exist $n_0=n, n_1, n_2, n_3, \ldots, n_m$ such that
$m$ satisfies that $U_k(j)=U_{k-1}(i+m)$ and for $h=0, 1, \ldots, m$, we have that
$\zeta_{n_h}=T^Y_{k-1}(i+h)$.

For $1\leq h< m$, let $X_h=Y_{\zeta^\prime_{n_h}}-Y_{\zeta^\prime_{n_{h-1}}}$,
and let $X_m=Y_{\zeta^\prime_{n_m}}-Y_{\zeta^\prime_{n_{m-1}}}-\beta_{k,j}$.
We now claim that for every $1\leq h\leq m$, conditioned on
$X_1,\ldots,X_{h-1}$, the distribution of $X_h$ is $(\lambda_{k-1},N_{k-1}^{\psi/2})$-close to $\DD(N_{k-1})$.
Indeed, if
$\PP^{(k-1)}\big(Y_{\zeta^\prime_{n_{h-1}}}\big)$
is good, then this claim follows from Corollary \ref{cor:quenched}. 
Otherwise,
as in Remark \ref{rem:Y_0},
the distribution of  $X_h$ is 
$(\lambda_{k-1},0)$-close to $\DD(N_{k-1})$ (and in particular
$(\lambda_{k-1},N_{k-1}^{\psi/2})$-close to $\DD(N_{k-1})$).

Therefore, by Lemma \ref{lem:sumapprox}, the distribution of 
\[
Y_{T^Y_k(j)}-x+\sum_{k^\prime=1}^{k-1}\beta_{k^\prime,j(k^\prime)}=\sum_{h=1}^mX_h
\]
is $(\lambda_k,R_{k+6}(L)N_{k-1}^{\psi/2}N^2_k/N^2_{k-1})$-close to $\DD(N_k)$. Therefore we get that with probability 1,
\[
\beta_{k,j}\leq N_{k-1}^{\psi/2}\cdot\frac{R_{k+6}(L)N^2_k}{N^2_{k-1}}\leq L^{4\psi}
.\]

\end{proof}

\begin{lemma}\label{lem:stpt}
For $j$ and $k$, if there exists $n$ such that $x_n\in U_k(j)$, then at least one of the following holds:
\begin{enumerate}
\item There exist $j^\prime$ such that $U_k(j)=U_\iota(j^\prime)$,
\item There exists $k^\prime$ and $j^\prime$ such that $U_k(j)=U_{k^\prime}(j^\prime)$ and
$x_{n-1}\in U_{k^\prime}(j^\prime-1)$ and $x_{n-1}$ is contained in a {\paral} $\PP(z,N_{k^\prime + 1})$ s.t. $z\in\LL_{N_{k^\prime+1}}$ and $\PP(z,N_{k^\prime + 1})$ is not good.
\item or 
$j\leq\left(\frac{N_{k+1}}{N_{k}}\right)^2$.
\end{enumerate}
\end{lemma}

\begin{proof}
Assume that $x_n\in U_k(j)$. Let $k^\prime=k(x_{n-1})$. If $k^\prime=0$ then case (2) holds. Assume 
$k^\prime>0$.
Then the intersection of
$U_k(j)$ and $\partial^+\PP^{(k^\prime)}(x_{n-1})$ is not empty. Therefore, by the definition of $k(x)$, we get that there is some $j^\prime$ such that $U_k(j)=U_{k^\prime}(j^\prime)$ and
$x_{n-1}\in U_{k^\prime}(j^\prime-1)$, and that one of the following occurs:
\begin{enumerate}
\item\label{item:iota} $k^\prime=\iota$.
\item\label{item:notgood} There exists $z\in\LL_{N_{k^\prime+1}}$ such that $\langle z,e_1\rangle=\langle x_{n-1},e_1\rangle$, and
$\PP^{(k^\prime+1)}(x_{n-1})$ is not good.
\item\label{item:z} No $z\in\LL_{N_{k^\prime+1}}$ exists with
$\langle z,e_1\rangle=\langle x_{n-1},e_1\rangle$.
\end{enumerate}
In cases \ref{item:iota} and \ref{item:notgood}, the lemma holds. Thus we assume that the case that occurs is \ref{item:z}.

In this case, there exists $n^\prime<n-1$ such that 
$\langle x_{n^\prime},e_1\rangle = N_{k^\prime+1}^2\left\lfloor\frac{\langle x_{n-1},e_1\rangle}{N_{k^\prime+1}^2}\right\rfloor$.
$x_{n-1}$ is in $\PP^{(k^\prime+1)}(x_{n^\prime})$. If $\PP^{(k^\prime+1)}(x_{n^\prime})$ is not good, 
then  $x_{n-1}$ is contained in a {\paral} $\PP(z,N_{k^\prime + 1})$ s.t. $z\in\LL_{N_{k^\prime+1}}$
and $\PP(z,N_{k^\prime + 1})$ is not good. If $\PP^{(k^\prime+1)}(x_{n^\prime})$ is good and
$\langle x_{n^\prime},e_1\rangle\neq 0$ then $\zeta_{n^\prime+1}$ is the exit time from
$\PP^{(k^\prime+1)}(x_{n^\prime})$, which stands in contradiction to the assumptions. If
$\langle x_{n^\prime},e_1\rangle=0$, then $j\leq\left(\frac{N_{k+1}}{N_{k}}\right)^2$.

\end{proof}

We now let $M$ be the number of stopping times $\zeta_n$ in the definition of $\{Y_n\}$.

\begin{lemma}\label{lem:nstp}
Let $[Y]$ be the set of points visited by $\{Y_n\}$. For every $k=1,\ldots,\iota$, let
\[
Q_k(\{Y_n\})=\#
\left\{
z\in\LL_{N_k}\ :\ [Y]\cap\PP(z,N_k)\neq\emptyset
\mbox{ and }
\PP(z,N_k)
\mbox{ is bad }
\right\}.
\]
Then 
\[
M\leq \frac{2L}{N_\iota^2} + L^{2\chi}\sum_{k=1}^\iota Q_k(\{Y_n\}) + \iota L^{2\chi}
\leq L^{2\chi}\cdot\left(\iota + 2 + \sum_{k=1}^\iota Q_k(\{Y_n\})\right).
\]

\end{lemma}

\begin{proof}
This follows from Lemma \ref{lem:stpt}. There are at most $\frac{2L}{N_\iota^2}$ stopping times 
that are caused by reaching the end of a $N_\iota$ block, $\iota L^{2\chi}$ stopping times 
that are caused by the beginning and at most $L^{2\chi}\sum_{k=1}^\iota Q_k(\{Y_n\}) $ stopping times 
that are caused by visiting {\parals} that are not good.
\end{proof}

We now draw the connection between the walks $\{Y_n\}$ and $\{X_n\}$.

\begin{lemma}\label{lem:xnyn}
Let $\upsilon=(v_1,v_2,\ldots v_{N_\upsilon})$ ($N_\upsilon$ is the length of the path $\upsilon$) be a nearest-neighbor path starting at the origin, never returning to the origin, and ending at $\partial^+ B_L$. For every $k=1,\ldots,\iota$, let
\[
Q_k(\upsilon)=\#
\left\{
z\in\LL_{N_k}\ :\ \upsilon\cap\PP(z,N_k)\neq\emptyset
\mbox{ and }
\PP(z,N_k)
\mbox{ is bad }
\right\}.
\]

and let 
$Q(\upsilon)=L^{2\chi}\cdot\big(\iota + 2 + \sum_{k=1}^\iota Q_k(\upsilon)\big).$
Then
\begin{equation}\label{eq:xnyn}
\frac{
{\quenchedP_\omega}\left(
X_j=v_j \mbox{ for all } j<N_\upsilon
\right)
}{
{\quenchedP_\omega}\left(
Y_j=v_j \mbox{ for all } j<N_\upsilon
\right)
}
\geq
\frac{1}{2}\eta^{
Q(\upsilon)\cdot(L^{4\psi}+2)
},
\end{equation}
where $\eta$ is the ellipticity constant, as in \eqref{eq:unifelliptic}.
\end{lemma}


\begin{proof}
First note that due to uniform ellipticity, 
\[
\quenchedP_\omega\left(X_j=v_j \mbox{ for all } j<N_\upsilon\right)>0
\]
for every $\upsilon$. Therefore without loss of generality we can restrict ourselves to considering only $\upsilon$-s such that 
\[
{\quenchedP_\omega}\left(Y_j=v_j \mbox{ for all } j<N_\upsilon\right)>0.
\]

For such $\upsilon$, we define the sequences of times $\zeta_n$ and $\zeta^\prime_n$ in a fashion that is
very similar to the definition in the construction of $Y$-process:
$\zeta_0=\zeta^\prime_0=0$ and $\zeta_1=N_1^2$. Given $\zeta_0,\ldots,\zeta_n$ and $\zeta^\prime_0,\ldots,\zeta^\prime_{n-1}$, let $x^\prime_n=\upsilon_{\zeta_n}$. Let $\zeta^\prime_n$ be the smallest $\ell>\zeta_n$ such that $\langle \upsilon_{\ell-1},e_1\rangle>\langle x^\prime_n,e_1\rangle$, and let 
$x_n=\upsilon_{\zeta^\prime_n}$. Let $k=k(x_n)$. If $k>0$, then we let $\zeta_{n+1}=T_{\partial^+\PP^{(k)}(x_n)}(\upsilon)$. Otherwise, $\zeta_{n+1}=\zeta^\prime_n+N_1^2$.

Then,
\begin{eqnarray*}
\quenchedP_\omega\left(Y_j=v_j \mbox{ for all } j<N_\upsilon\right)\\
\leq
\prod_{n:k(x_n)>0} \quenchedP^{x_n}_\omega  
\left(X_\ell=\nu_{\ell+\zeta^\prime_{n}}\,;\,\ell=1,\ldots,\zeta_{n+1}-\zeta^\prime_n
|T_{\partial\PP^{(k)}(x_n)}=T_{\partial^+\PP^{(k)}(x_n)}
\right),
\end{eqnarray*}
and
\begin{eqnarray}
\nonumber
\quenchedP_\omega\left(X_j=v_j \mbox{ for all } j<N_\upsilon\right)\\
\nonumber
\geq
\prod_{n:k(x_n)>0} \quenchedP^{x_n}_\omega  
\left(X_\ell=\nu_{\ell+\zeta^\prime_{n}}\,;\,\ell=1,\ldots,\zeta_{n+1}-\zeta^\prime_n
|T_{\partial\PP^{(k)}(x_n)}=T_{\partial^+\PP^{(k)}(x_n)}
\right)\\ \label{eq:condit}
\cdot
\prod_{n:k(x_n)>0} \quenchedP^{x_n}_\omega  
\left(T_{\partial\PP^{(k)}(x)}=T_{\partial^+\PP^{(k)}(x_n)}
\right)\\ \label{eq:ellipt}
\cdot
\prod_{n}\eta^{\|x^\prime_n-x_n\|_1+2}
\cdot
\prod_{n:k(x_{n})=0}\eta^{L^{2\psi}}.
\end{eqnarray}

The first inequality follows from the fact that inside the good \parals { } $\{Y_n\}$ performs quenched random walk on the environment $\omega$. For the second inequality, the first term and \eqref{eq:condit} count the probability of all steps in the
good \parals. In addition, at each stopping time, the process $\{X_n\}$ has to walk from $x_n^\prime$ to $x_n$, and when $k(x_{n})=0$ it also needs to traverse through an $N_1$ block. In \eqref{eq:ellipt} we bound the probability of all of these steps by ellipticity.

By Proposition \ref{prop:quenched}, the product in \eqref{eq:condit} is no less than a half. By the definitions of $k(x)$ and $\beta_{k,j}$, by Lemma \ref{lem:betasmall}, and by uniform ellipticity with constant $\eta$, the product in \eqref{eq:ellipt} is bounded below by 
\[
\eta^{
Q(\upsilon)\cdot(L^{4\psi}+2)
}.
\]
 Therefore,
 \[
 \frac{
{\quenchedP_\omega}\left(
X_j=v_j \mbox{ for all } j<N_\upsilon
\right)
}{
{\quenchedP_\omega}\left(
Y_j=v_j \mbox{ for all } j<N_\upsilon
\right)
}
\geq
\frac{1}{2}\eta^{
Q(\upsilon)\cdot(L^{4\psi}+2)
}.
\]

\end{proof}

\ignore{
\begin{lemma}\label{lem:entropy}
For every $j$ and $k$, let $\DD^{\omega}(k,j)$ be the distribution of $Y_{T^Y_k(j)}$. Then if $X\sim\DD^{\omega}(k,j)$ then $X$ can be presented as $X=W+Z$ where $Z\sim\DD(\sqrt{j}N_k)$ and
$\|W\|<L^{20\psi}$. Furthermore, we can present 
\[
Y_{T^Y_k(j)}-Y_{T^Y_k(j-1)}=Z^\prime+W^\prime
\]
with $\|W^\prime\|<L^{20\psi}$ and  $Z^\prime\sim\DD(N_k)$ and independent of 
\[\{Y_j:j=1,\ldots,T^Y_k(j-1)\}.\]
\end{lemma}
}

For $k$ and $j$, we define $T^{\prime Y}_k(j)$ as follows:
If there exists $n$ such that $\zeta_n=T^Y_k(j)$, then
$T^{\prime Y}_k(j)=\zeta^\prime_n$. Otherwise,
$T^{\prime Y}_k(j)=T^Y_k(j)$.

\begin{lemma}\label{lem:entropy}
Conditioned on $\{Y_\ell\, :\, \ell\leq T^{\prime Y}_k(j-1)\}$, the distribution
of $Y_{T^{\prime Y}_k(j)}-Y_{T^{\prime Y}_k(j-1)}$ 
is $(\lambda_\iota,2L^{4\psi})$-close to $\DD(N_k)$.
\end{lemma}

\begin{proof}
We look into two different cases: If
$\PP^{(k)}\big(Y_{T^{\prime Y}_k(j-1)}\big)$
is good, then it follows from Corollary \ref{cor:quenched}. Otherwise, it follows from the definition of $\beta_{k,j}$.
\end{proof}

From Lemma \ref{lem:entropy}, we get the following useful corollary.

\begin{corollary}\label{cor:entropy} Assume that $u$ is large enough.
Condition on $\{Y_\ell\, :\, \ell\leq T^{\prime Y}_k(j-1)\}$, and let
$\bar Y=Y_{T^{\prime Y}_k(j-1)}+\annealedE(X_{T_{N_k^2}})$. For 
every $x\in U_k(j)$ such that $\|x-\bar Y\|<4N_k$,


\begin{equation}\label{eq:ent}
\quenchedP_\omega\big(
\|Y_{T^{\prime Y}_k(j)} - x \| < N_k
\, \big| \, 
Y_\ell\, :\, \ell\leq T^{\prime Y}_k(j-1)
\big) > \rho
\end{equation}
for some constant $\rho>0$.
\end{corollary}

\ignore{
\begin{proof}
If
$\PP^{(k)}\big(Y_{T^{\prime Y}_k(j-1)}\big)$
is not good, then it follows from $(\lambda_k,0)$-closeness to $\DD(N_k)$ and 
Lemma \ref{lem:lbound}. If $\PP^{(k)}\big(Y_{T^{\prime Y}_k(j-1)}\big)$ is good, then
it follows from Corollary \ref{cor:quenched}, Lemma \ref{lem:lbound} and uniform ellipticity with
constant $\eta$.
\end{proof}
}

\begin{proof} By Lemma \ref{lem:entropy},
the quenched distribution of 
$Y_{T^{\prime Y}_k(j)}-Y_{T^{\prime Y}_k(j-1)}$
conditioned on the history of the walk 
is $(\lambda_\iota,2L^{4\psi})$-close to the annealed distribution $\DD(N_k)$.
Therefore,
\[
\quenchedP_\omega\big(
\|Y_{T^{\prime Y}_k(j)} - x \| < N_k
\, \big| \, 
Y_\ell\, :\, \ell\leq T^{\prime Y}_k(j-1)
\big) > 
\DD(N_k)(y:\|y-x\|<\frac{N_k}{2})-\lambda_\iota.
\]
By Lemma \ref{lem:lbound}, $\DD(N_k)(y:\|y-x\|<\frac{N_k}{2})$ is bounded away from zero. On the other hand, $\lambda_\iota$ goes to zero as $L$ goes to infinity. The corollary follows.
\end{proof}

\begin{lemma}\label{lem:bjk}
Conditioned on $\{Y_\ell\, :\, \ell\leq T^{\prime Y}_k(j-1)\}$, the (quenched) probability
that $\{Y_\ell\}_{\ell\geq T^{\prime Y}_k(j-1)}$
exits $\PP^{(k)}(Y_{T^{\prime Y}_k(j-1)})$ through
$\partial^+\PP^{(k)}(Y_{T^{\prime Y}_k(j-1)})$
is 
$1-L^{-\xi(1)}$.
\end{lemma}

\begin{proof}
We denote by $E$ the event whose probability we are trying to estimate.
If $\PP^{(k)}(Y_{T^{\prime Y}_k(j-1)})$ is good, then the lemma follows by the definition of a good block.
Therefore we may assume that $\PP^{(k)}(Y_{T^{\prime Y}_k(j-1)})$ is a bad block.
In this case we prove the lemma using induction on $k$. For $k=1$ this follows immediately from the definition of
the auxiliary walk on bad $N_1$ blocks.

Now assume $k>1$. We assume that the lemma holds for $Y_{T^{\prime Y}_{k-1}(h)}$ for every $h$.
(if the block
$\PP^{(k-1)}(Y_{T^{\prime Y}_{k-1}(h-1)})$
is good, then we already proved it. If the block is bad then this is the induction hypothesis).

Let $l$ be such that $lN^2_{k-1}=(j-1)N_k^2$, and let $m$ be such that $(l+m)N^2_{k-1}=jN_k^2$

For $h=1,\ldots,m$, let
\[
I_h=Y_{T^{\prime Y}_{k-1}(l+h)}-Y_{T^{\prime Y}_{k-1}(l+h-1)}.
\]

Let $A$ be the event that for every $h=1,\ldots,m$, the walk $\{Y_\ell\}$ leaves 
$\PP^{(k-1)}\big(Y_{T^{\prime Y}_{k-1}(j-1)}\big)$ through its front. Then by the induction hypothesis,
$\quenchedP_\omega\left(A|Y_\ell\ ;\ \ell=1,\ldots,T^{\prime Y}_{k}(j-1)\right) = 1-L^{-\xi(1)}$.

Now,
\begin{eqnarray}\label{eq:wntbjk}
\nonumber
&&\quenchedP_\omega\left(\left.
E^c
\right|Y_\ell\ ;\ \ell=1,\ldots,T^{\prime Y}_{k}(j-1)\right) \\
&\leq& 
\nonumber
\quenchedP_\omega\left(A^c|Y_\ell\ ;\ \ell=1,\ldots,T^{\prime Y}_{k}(j-1)\right)\\
&+& \quenchedP_\omega\left(E^c|A\ ;\ Y_\ell\ ;\ \ell=1,\ldots,T^{\prime Y}_{k}(j-1)\right),
\end{eqnarray}

and 

\newcommand{\AY}{\begin{array}{c}A,\\ Y_\ell:\ell=1,\ldots,T^{\prime Y}_{k}(j-1)\end{array}}

\begin{eqnarray}\label{eq:wntbjk2}
\nonumber
&&\quenchedP_\omega\left(E^c\left|
\AY
\right.\right)\\
\nonumber
&\leq&
\quenchedP_\omega
\left(\left.
\exists_{1\leq h\leq m} 
\left\|\sum_{i=1}^h I_i - hN^2_{k-1}\frac{\vartheta}{\langle \vartheta,e_1\rangle}\right\| > \frac 12N_kR_5(N_k)
\right| \AY
\right)\\
&\leq&
\sum_{h=1}^m
\quenchedP_\omega
\left(\left.
\left\|\sum_{i=1}^h I_i - hN^2_{k-1}\frac{\vartheta}{\langle \vartheta,e_1\rangle}\right\| > \frac 12N_kR_5(N_k)
\right|\AY 
\right)
\end{eqnarray}

It is sufficient to show that for every $h$, the probability in \eqref{eq:wntbjk2} is $L^{-\xi(1)}$.

Fix $h$. 
Conditioned on $A$, the variable $J_i=I_i-N^2_{k-1}\frac{\vartheta}{\langle \vartheta,e_1\rangle}$ is bounded
by $2N_{k-1}R_5(N_{k-1})$. Furthermore, the quenched expectation of $J_i$ conditioned on
$A,J_1,\ldots,J_{i-1}$ and $Y_\ell\ ;\ \ell=1,\ldots,T^{\prime Y}_{k}(j-1)$ is bounded by $N_{k-1}^{\psi/2}$ (see \eqref{eq:choosepsi}).

Therefore, using the Azuma-H\"offding inequality, we get that 

\begin{eqnarray*}
&&\quenchedP_\omega\left(E^c\left|\AY\right.\right)\\
&\leq&
C\exp\left(
\frac
{-N_k^2R_5^2(N_k)}
{8N_{k-1}^2R_5^2(N_{k-1})\cdot(N_k/N_k-1)^2}
\right)\\
&=&
C\exp\left(
\frac
{-R_5^2(N_{k})}
{8R_5^2(N_{k-1})}
\right)\\
&\leq&
C_1\exp\left(-C_2
\exp\left([\log(\rho_1+\ldots+\rho_{k-1}+\rho_k)-\log(\rho_1+\ldots+\rho_{k-1})][\log\log L]^5
\right)
\right)\\
&=&L^{-\xi(1)},
\end{eqnarray*}
where the last inequality follows from the definition of $N_k$, the definition of $R_k(N)$, and a first order Taylor approximation.
\end{proof}

\section{The random direction event}\label{sec:randirev}

In this section we consider an event $W^{(w)}$ which we call the random direction event. First we construct an event $W^{(w)}$. Then we show that the probability that $W^{(w)}$ occurs is more than $u^{\epsilon-\hetzi}$. Then we show some estimates on the hitting probabilities of the walk conditioned on the occurence of $W^{(w)}$. In the next section we will show that these estimates
are sufficient for proving \eqref{eq:toshowomeg}, and thus Theorem \ref{thm:main}.

\ignore{
\begin{enumerate}
\item\label{item:probw}

\item\label{item:probifw}
For every $\upsilon\in W^{(w)}$,
\[
\frac{P_\omega(\forall_n\, X_n=\upsilon_n)}{P_\omega(\forall_n\, Y_n=\upsilon_n)}
>u^{-\hetzi-\epsilon}.
\]
\end{enumerate}
}

\subsection{Definition of $W^{(w)}$}\label{sec:wdef}
Let $M=\left[(\log u)^{1-\epsilon}\right]$, and for $k=1,\ldots,\iota$ let 
$\Epsilon_k=\annealedE^0(X_{T_{\partial\PP(0,N_k)}}|T_{\partial\PP(0,N_k)}=T_{\partial^+\PP(0,N_k)})$
be the annealed expectation of the point of exit of $\PP(0,N_k)$. Let $A_1=1$, and for every $k>1$, let $A_k$ be the smallest integer number such that
$
A_k N_k^2 > (M+A_{k-1})N_{k-1}^2.
$

Note that $A_k\leq M$.

For $k=1,\ldots,\iota$ and $j>A_k$, we define
${\mathbb B}_k(j)$ to be the event that $\{Y_n\}$ leaves $\PP^{(k)}(Y_{T^{\prime Y}_k(j-1)})$ through
$\partial^+\PP^{(k)}(Y_{T^{\prime Y}_k(j-1)})$.

Fix $w\in[-1,1]^{d-1}$.
For $j>A_k$ we then define the event $W^{(w)}_k(j)$ as follows:
\[
W^{(w)}_k(j)=\big\{\|Y_{T^{\prime Y}_k(j)} - Y_{T^{\prime Y}_k(A_k)}  - (j-A_k)\Epsilon_k-w(j-A_k)N_k\|<N_k\big\}
.\]

Then,
\[
W^{(w)}_k=\bigcap_{j=A_k+1}^{A_k+M} [W^{(w)}_k(j)\cap {\mathbb B}_k(j)],
\]
and $W^{(w)}$ is defined to be the intersection
\[
W^{(w)}=\bigcap_{k=1}^\iota W^{(w)}_k.
\]

\subsection{The probability of $W^{(w)}$}
In this subsection we bound from below the probability of the event $W^{(w)}$.
\begin{lemma}\label{lem:wkj}
\begin{enumerate}
\item\label{item:wkj}
There exists some $\rho>0$ such that for $1\leq k\leq \iota$ and $A_k<j\leq A_k+M$,
\[
\quenchedP_{\omega}
\big(
W^{(w)}_k(j)\,
\big|\,W^{(w)}_1,\ldots,W^{(w)}_{k-1},W^{(w)}_k(A_k+1),\ldots,W^{(w)}_k(j-1),
{\mathbb B}_k(A_k+1),\ldots,{\mathbb B}_k(j-1)
\big)>\rho
\]
\item\label{item:bkj}
For $1\leq k\leq \iota$ and $A_k<j\leq A_k+M$,
\[
\quenchedP_{\omega}
\big(
{\mathbb B}_k(j)\,
\big|\,W^{(w)}_1,\ldots,W^{(w)}_{k-1},W^{(w)}_k(A_k+1),\ldots,W^{(w)}_k(j-1),
{\mathbb B}_k(A_k+1),\ldots,{\mathbb B}_k(j-1)
\big) = 1-o(1).
\]
\end{enumerate}
\end{lemma}

\begin{proof} For Part \ref{item:wkj},
Conditioned on
$W^{(w)}_1\cap\ldots\cap W^{(w)}_{k-1}\cap W^{(w)}_k(A_k+1)\cap\ldots\cap W^{(w)}_k(j-1),
{\mathbb B}_k(A_k+1),\ldots,{\mathbb B}_k(j-1)$,
we get that
\[
\|Y_{T^{\prime Y}_k(j-1)} - Y_{T^{\prime Y}_k(A_k)}  - (j-1)\Epsilon_k-w(j-1)N_k\|<N_k
.\]
Therefore,
\[
\|Y_{T^{\prime Y}_k(A_k)}  + j\Epsilon_k+wjN_k - (Y_{T^{\prime Y}_k(j-1)}+\Epsilon_k)\|<4N_k.
\]
By Corollary \ref{cor:entropy} and the definition of $W_k(j)$,
we get that
\begin{equation*}
\quenchedP_{\omega}
\big(
W^{(w)}_k(j)\,
\big|\,W^{(w)}_1,\ldots,W^{(w)}_{k-1},W^{(w)}_k(A_k+1),\ldots,W^{(w)}_k(j-1),
{\mathbb B}_k(A_k+1),\ldots,{\mathbb B}_k(j-1)
\big)>\rho
\end{equation*}

as desired.

Part \ref{item:bkj} follows from Lemma \ref{lem:bjk}.

\end{proof}

As a result of Lemma \ref{lem:wkj} and the choice of $M$, we get the following lemma:
\begin{lemma}\label{lem:probw}
The probability of $W^{(w)}$ is bounded from below by $u^{\epsilon-1/2}$.
\end{lemma}

\subsection{Hitting probability estimates}\label{sec:entropy} In this subsection we bound from above the probability, conditioned on $W^{(w)}$, of a {\paral} to be hit.
We begin with a simple claim.

\begin{claim}\label{claim:hitpar}
Fix $k$ between $1$ and $\iota$, and let
\[
A_k+M\leq j\leq \left(\frac{N_{k+1}}{N_k}\right)^2(A_{k+1}+M).
\]
Let $z\in\LL_{N_k}\cap U_{k}(j)$.
Then,
\begin{equation}\label{eq:intw}
\mathop\int_{[-1,1]^{d-1}}
\quenchedP_\omega\left(
\{Y_n\}\cap\PP(z,N_k)\neq\emptyset\, |\, W^{(w)}
\right)dw
\leq (\log u)^{(1-d)(1-2\epsilon)}.
\end{equation}
\end{claim}

\begin{proof}

First note that there exists $A_{k+1}<j^\prime\leq A_{k+1}+M$ 
and $z^\prime\in\LL_{N_{k+1}}\cap U_{k+1}(j^\prime)$ such that
$\PP(z,N_k)\subseteq\PP(z^\prime,N_{k+1})$.

Then by the definition of $W^{(w)}$ (and using the fact that $W^{(w)}$ implies ${\mathbb B}_{k+1}(j^\prime)$), the probability 
\[
\quenchedP_\omega\left(\{Y_n\}\cap\PP(z,N_k)\neq\emptyset\, |\, W^{(w)}\right)
\]
is positive only if 
\[
\|z^\prime - M\EE_{k}-j^\prime \EE_{k+1}
-MwN_k-j^\prime wN_{k+1}
\|<N_{k+1}R_5(N_{k+1})
\]
and in particular $w$ needs to be in an area of side length which is no more than
$\frac{N_{k+1}R_5(N_{k+1})}{MN_k}\leq M^{-1}L^{\chi}$
and thus the integral in \eqref{eq:intw} is bounded by $(\log u)^{(1-d)(1-2\epsilon)}$.
\end{proof}

\begin{lemma}\label{lem:hitpar}
Fix $k$ between $1$ and $\iota$, and let $j>A_k+M$. Let $z\in\LL_{N_k}\cap U_{k}(j)$.
Then,
\[
\int_{[-1,1]^{d-1}}
\quenchedP_\omega\left(
\{Y_n\}\cap\PP(z,N_k)\neq\emptyset\, |\, W^{(w)}
\right)dw
\leq (\log u)^{(1-d)(1-2\epsilon)}.
\]
\end{lemma}

\begin{proof}
For $j<\left(\frac{N_{k+1}}{N_k}\right)^2(A_{k+1}+M)$, this follows from Claim \ref{claim:hitpar}. If 
$j\geq \left(\frac{N_{k+1}}{N_k}\right)^2(A_{k+1}+M)$, then there exists $k^\prime>k$ and 
$z^\prime\in\LL_{N_{k^\prime}}$ such that $z^\prime\in U_{k^\prime}(j^\prime)$ with
\[
A_{k^\prime}+M\leq j^\prime\leq \left(\frac{N_{{k^\prime}+1}}{N_k}\right)^2(A_{{k^\prime}+1}+M)
\]
and $\PP(z,N_k)\subseteq\PP(z^\prime,N_{k^\prime})$. Then by Claim \ref{claim:hitpar} applied to $k^\prime$ we get that

 \begin{eqnarray*}
&&\int_{[-1,1]^{d-1}}\quenchedP_\omega\left(
\{Y_n\}\cap\PP(z,N_k)\neq\emptyset\, |\, W^{(w)}
\right)dw\\
&\leq&\int_{[-1,1]^{d-1}}\quenchedP_\omega\left(
\{Y_n\}\cap\PP(z^\prime,N_{k^\prime})\neq\emptyset\, |\, W^{(w)}
\right)dw\\
&\leq& (\log u)^{(1-d)(1-2\epsilon)}.
\end{eqnarray*}

\end{proof}

\subsection{Expected number of bad {\parals} that are visited}
Fix $k$. Let 
\[
\DD(k)=\left\{
z\in\LL_k \cap B_{2L} \,\left|\,
\PP(z,N_k) \mbox{ is not good}
\right.\right\},
\]
and let 
\[
\BB(k)=\#\big\{
z\in\DD(k)\,\left|\,
\{Y_\ell\}\cap \PP(z,N_k) \neq \emptyset
\right.\big\}.
\]

We are interested in the distribution of the variable $\BB(k)$.

\begin{lemma}\label{lem:sizebk}
Fix $k$ and $\omega\in G$. Then
\begin{equation*}
\int_{[-1,1]^{d-1}}\quenchedE_\omega\left(\left.
\BB(k)\, \right|\,
W^{(w)}
\right)dw
\leq 3(\log u)^{1-\epsilon}.
\end{equation*}
\end{lemma}

\begin{proof}
Let
\[
\DD^{(1)}(k)=\DD(k)\cap\{z\,:\,\langle z,e_1\rangle\leq N_k^2(A_k+M)\}
\]
and
\[
\DD^{(2)}(k)=\DD(k)\cap\{z\,:\,\langle z,e_1\rangle> N_k^2(A_k+M)\},
\]
and for $i=1,2$ let
\[
\BB^{(i)}(k)=\left|\left\{
z\in\DD^{(i)}(k)\,\left|\,
\{Y_\ell\}\cap \PP(z,N_k) \neq \emptyset
\right.\right\}\right|.
\]

Then $\{Y_\ell\}$ visits no more than $A_k+M$ elements of $\DD^{(1)}(K)$, and thus
$\BB^{(1)}(k)\leq A_k+M\leq 2(\log u)^{1-\epsilon}$

Let $z\in\DD^{(2)}(k)$. Then by Lemma \ref{lem:hitpar},
\[
\int_{[-1,1]^{d-1}}\quenchedP_\omega\left(
\{Y_n\}\cap\PP(z,N_k)\neq\emptyset\, |\, W^{(w)}
\right)dw
\leq (\log u)^{(1-d)(1-2\epsilon)}.
\]
Therefore, using \eqref{eq:defgood} and \eqref{eq:chooseepsilon}, we get that
\begin{eqnarray*}
\int_{[-1,1]^{d-1}}\quenchedE_\omega\big(\BB^{(2)}(k)\,|\,W^{(w)}\big)dw
&\leq& (\log u)^{\alpha +\epsilon+(1-d)(1-2\epsilon)}\\
= (\log u)^{\alpha-d + 1 + (2d-1) \epsilon}
&\leq& (\log u)^{1 - \epsilon}.
\end{eqnarray*}

Combined, we get that

\begin{eqnarray*}
&&\int_{[-1,1]^{d-1}}
\quenchedE_\omega\left(\left.
\BB(k)\, \right|\,
W^{(w)}
\right)dw\\
&\leq&
\int_{[-1,1]^{d-1}}
\quenchedE_\omega\left(\left.
\BB^{(1)}(k)\, \right|\,
W^{(w)}
\right)dw\\
&+&\int_{[-1,1]^{d-1}}
\quenchedE_\omega\left(\left.
\BB^{(2)}(k)\, \right|\,
W^{(w)}
\right)dw\\
&\leq& 3(\log u)^{1-\epsilon}.
\end{eqnarray*}

\end{proof}

\section{Proof of main result}\label{sec:prfmain}

In this section we prove Theorem \ref{thm:main}.

\begin{proof}[Proof of Theorem \ref{thm:main}]
%

by Lemma \ref{lem:sizebk},

\begin{equation*}
\int_{[-1,1]^{d-1}}
\quenchedE_\omega\left(\left.
\sum_{k=1}^\iota \BB(k)\, \right|\,
W^{(w)}
\right)dw
\leq 3\iota(\log u)^{1-\epsilon}.
\end{equation*}

Therefore, there exists $w$ such that
\begin{equation*}
\quenchedE_\omega\left(\left.
\sum_{k=1}^\iota \BB(k)\, \right|\,
W^{(w)}
\right)
\leq 3\iota(\log u)^{1-\epsilon}.
\end{equation*}

We now fix $w$ to be such value.

Let 
\begin{equation}\label{eq:wbarint}
\bar W=W^{(w)}\bigcap \left\{\sum_{k=1}^\iota \BB(k) \leq 6\iota(\log u)^{1-\epsilon} \right\}.
\end{equation}

Then by Markov's inequality,
$\quenchedP_\omega(\bar W)\geq 0.5\quenchedP_\omega(W^{(w)})\geq \frac 12u^{\epsilon - 1/2}.$
Note that there is a set $V$ of paths, such that
\[
\bar W=\big\{ \{Y_n\}\in V  \big\},
\]

and for every $\upsilon\in V$, by Lemma \ref{lem:xnyn} and by \eqref{eq:wbarint} and the choice of $\chi$ and $\psi$ (\eqref{eq:choosepsi},\eqref{eq:choosechi}),
\begin{equation*}
\frac{
{\quenchedP_\omega}\left(
X_j=v_j \mbox{ for all } j<N_\upsilon
\right)
}{
{\quenchedP_\omega}\left(
Y_j=v_j \mbox{ for all } j<N_\upsilon
\right)
}
\geq
\frac{1}{2}\eta^{
(\iota + 2 + 6\iota (\log u)^{1-\epsilon}) L^{3\chi + 4\psi}
}
\geq u^{\epsilon - 1/2}.
\end{equation*}

Therefore,
\[
\quenchedP_\omega(\{X_n\}\in V)
\geq u^{\epsilon - 1/2} \quenchedP_\omega(\{Y_n\}\in V)
\geq u^{\epsilon - 1}
\]

Every path in $V$ reaches $\partial^+B_{2L}$ before returning to $0$, and therefore we get
\eqref{eq:toshowomeg}, from which we get Proposition \ref{prop:main} and Theorem \ref{thm:main}.

\end{proof}

\section*{Acknowledgment}
I wish to thank A.-S.~Sznitman for introducing this problem to me, and to thank
G.~Kozma, T.~Schmitz and O.~Zeitouni for useful discussions.
In addition I thank O.~Zeitouni for suggesting that I use the methods of
\cite{gantertzeitouni} in order to 
prove Corollary \ref{cor:mainquenched}.
A very careful and detailed referee report contributed significantly to the quality of the presentation, and I am grateful for that.

\bibliography{slowdown}

\begin{thebibliography}{10}

\bibitem{BZ07}
Noam Berger and Ofer Zeitouni.
\newblock Invariance principle for certain ballistic random walks in i.i.d.
  envronments.
\newblock In {\em In and out of equilibrium 2}. Birkh\"auser Verlag, 2008.

\bibitem{bolthausen_sznitman}
Erwin Bolthausen and Alain-Sol Sznitman.
\newblock On the static and dynamic points of view for certain random walks in
  random environment.
\newblock {\em Methods Appl. Anal.}, 9(3):345--375, 2002.
\newblock Special issue dedicated to Daniel W. Stroock and Srinivasa S. R.
  Varadhan on the occasion of their 60th birthday.

\bibitem{DPZ}
Amir Dembo, Yuval Peres, and Ofer Zeitouni.
\newblock Tail estimates for one-dimensional random walk in random environment.
\newblock {\em Comm. Math. Phys.}, 181(3):667--683, 1996.

\bibitem{drewitz}
Alexander Drewitz and Alejandro~F. Ram\'irez.
\newblock Ballisticity conditions for random walk in random environment.
\newblock {\em Preprint}, 2009.
\newblock Available at \texttt{http://arxiv.org/abs/0903.4465}.

\bibitem{durrett}
R.~Durrett.
\newblock {\em Probability: Theory and Examples}.
\newblock Duxbury Press, 1996.

\bibitem{gantertzeitouni}
Nina Gantert and Ofer Zeitouni.
\newblock Quenched sub-exponential tail estimates for one-dimensional random
  walk in random environment.
\newblock {\em Comm. Math. Phys.}, 194(1):177--190, 1998.

\bibitem{peterson}
Jonathon Peterson.
\newblock {\em Limiting distributions and large deviations for random walks in
  random environments}.
\newblock PhD thesis, University of Minnesota, 2008.

\bibitem{sznitmanquenched}
Alain-Sol Sznitman.
\newblock Slowdown estimates and central limit theorem for random walks in
  random environment.
\newblock {\em J. Eur. Math. Soc. (JEMS)}, 2(2):93--143, 2000.

\bibitem{SznitmanT}
Alain-Sol Sznitman.
\newblock On a class of transient random walks in random environment.
\newblock {\em Ann. Probab.}, 29(2):724--765, 2001.

\bibitem{SznitmanTprime}
Alain-Sol Sznitman.
\newblock An effective criterion for ballistic behavior of random walks in
  random environment.
\newblock {\em Probab. Theory Related Fields}, 122(4):509--544, 2002.

\bibitem{sznit_zer}
Alain-Sol Sznitman and Martin P.~W. Zerner.
\newblock A law of large numbers for random walks in random environment.
\newblock {\em Ann. Probab.}, 27(4):1851--1869, 1999.

\bibitem{varadhan}
S.~R.~S. Varadhan.
\newblock Large deviations for random walks in a random environment.
\newblock {\em Comm. Pure Appl. Math.}, 56(8):1222--1245, 2003.
\newblock Dedicated to the memory of J{\"u}rgen K. Moser.

\bibitem{yilmaz}
Atilla Yilmaz.
\newblock Averaged large deviations for random walk in a random environment.
\newblock {\em Preprint}, 2008.

\bibitem{zerner}
Martin P.~W. Zerner.
\newblock A non-ballistic law of large numbers for random walks in i.i.d.\
  random environment.
\newblock {\em Electron. Comm. Probab.}, 7:191--197 (electronic), 2002.

\end{thebibliography}


\begin{thebibliography}{11}

\bibitem{BZ07}
Noam Berger and Ofer Zeitouni.
\newblock Invariance principle for certain ballistic random walks in i.i.d.
  envronments.
\newblock In Vladas Sidoravicius and Maria~Eul\'alia Vares, editors, {\em In
  and Out of Equilibrium 2}, volume~60 of {\em Progress in Probabability},
  pages 137--160. Birkh\"auser, 2008.

\bibitem{bolthausen_sznitman}
Erwin Bolthausen and Alain-Sol Sznitman.
\newblock On the static and dynamic points of view for certain random walks in
  random environment.
\newblock {\em Methods Appl. Anal.}, 9(3):345--375, 2002.
\newblock Special issue dedicated to Daniel W. Stroock and Srinivasa S. R.
  Varadhan on the occasion of their 60th birthday.

\bibitem{gantertzeitouni}
Nina Gantert and Ofer Zeitouni.
\newblock Quenched sub-exponential tail estimates for one-dimensional random
  walk in random environment.
\newblock {\em Comm. Math. Phys.}, 194(1):177--190, 1998.

\bibitem{peterson}
Jonathon Peterson.
\newblock Limiting distributions and large deviations for random walks in
  random environments.
\newblock {\em PhD Thesis, University of Minnesota}, 2008.

\bibitem{sznitmanquenched}
Alain-Sol Sznitman.
\newblock Slowdown estimates and central limit theorem for random walks in
  random environment.
\newblock {\em J. Eur. Math. Soc. (JEMS)}, 2(2):93--143, 2000.

\bibitem{SznitmanT}
Alain-Sol Sznitman.
\newblock On a class of transient random walks in random environment.
\newblock {\em Ann. Probab.}, 29(2):724--765, 2001.

\bibitem{SznitmanTprime}
Alain-Sol Sznitman.
\newblock An effective criterion for ballistic behavior of random walks in
  random environment.
\newblock {\em Probab. Theory Related Fields}, 122(4):509--544, 2002.

\bibitem{sznit_zer}
Alain-Sol Sznitman and Martin Zerner.
\newblock A law of large numbers for random walks in random environment.
\newblock {\em Ann. Probab.}, 27(4):1851--1869, 1999.

\bibitem{varadhan}
S.~R.~S. Varadhan.
\newblock Large deviations for random walks in a random environment.
\newblock {\em Comm. Pure Appl. Math.}, 56(8):1222--1245, 2003.
\newblock Dedicated to the memory of J{\"u}rgen K. Moser.

\bibitem{yilmaz}
Atilla Yilmaz.
\newblock Averaged large deviations for random walk in a random environment.
\newblock {\em Preprint}, 2008.

\bibitem{zerner}
Martin P.~W. Zerner.
\newblock A non-ballistic law of large numbers for random walks in i.i.d.\
  random environment.
\newblock {\em Electron. Comm. Probab.}, 7:191--197 (electronic), 2002.

\end{thebibliography}
\bibliographystyle{plain}

\ignore{

}

\end{document}